\documentclass[12pt,reqno]{amsart}

\date{14 March 2024}
\title[Perfect matching modules]
{Perfect matching modules, dimer partition functions and cluster characters}
\author{\.{I}lke \c{C}anak\c{c}\i}
\address{\.{I}lke \c{C}anak\c{c}\i\\Department of Mathematics\\Vrije Universiteit Amsterdam\\De Boelelaan 1111\\1081 HV Amsterdam\\The Netherlands}
\email{i.canakci@vu.nl}
\author{Alastair King}
\address{Alastair King\\Department of Mathematical Sciences \\
 University of Bath \\ Claverton Down \\ Bath BA2 7AY \\ United Kingdom}
 \email{a.d.king@bath.ac.uk}
\author{Matthew Pressland}
\address{Matthew Pressland\\School of Mathematics \& Statistics\\University of Glasgow\\University Place\\Glasgow G12 8QQ\\United Kingdom}
\email{matthew.pressland@glasgow.ac.uk}

\usepackage{amsmath,amssymb,amsthm}
\usepackage[usenames,dvipsnames,svgnames]{xcolor}
\usepackage[colorlinks, urlcolor=Navy, linkcolor=Navy, citecolor=Navy]{hyperref}
\usepackage{url}
\usepackage{enumerate}
\usepackage[UKenglish]{babel}
\usepackage[babel=true]{microtype}
\usepackage{scrextend}
\usepackage[twoside=semi, DIV=default, headinclude]{typearea}

\usepackage{tikz}
  \usetikzlibrary{calc}
  \usetikzlibrary{arrows,decorations.markings}
\usepackage{tikz-cd}

\newcommand{\strandcolor}{teal}
\newcommand{\graphcolor}{black}
\newcommand{\quivcolor}{red}
\newcommand{\bdrycolor}{gray}

\tikzset{strand/.style={\strandcolor}, boundary/.style={thick, \bdrycolor},  bipedge/.style={\graphcolor, thick},
 quivarrow/.style={\quivcolor, -latex, thick} }

\newcommand{\strarrow}{\arrow{angle 60}}
\pgfmathsetmacro{\bstart}{130} \pgfmathsetmacro{\seventh}{360/7}
\newcommand{\dotrad}{0.9pt}

\usepackage[normalem]{ulem}

\theoremstyle{plain}
\newtheorem{theorem*}{Theorem}
\newtheorem{theorem}{Theorem}[section]
\newtheorem{proposition}[theorem]{Proposition}
\newtheorem{lemma}[theorem]{Lemma}
\newtheorem{corollary}[theorem]{Corollary}
\newtheorem{claim}[theorem]{Claim}

\theoremstyle{definition}
\newtheorem{definition}[theorem]{Definition}

\newtheorem{remark}[theorem]{Remark}

\numberwithin{equation}{section}
\numberwithin{figure}{section}

\newcommand{\Gm}{(\CC^\times)}

\renewcommand{\leq}{\leqslant}
\renewcommand{\geq}{\geqslant}
\renewcommand{\subset}{\subseteq}

\newcommand{\dual}{^\vee}
\newcommand{\op}{^{\mathrm{op}}}

\newcommand{\isom}{\cong}
\newcommand{\compo}{\circ}

\newcommand{\setmin}{\smallsetminus}

\newcommand{\restr}[1]{\rho(#1)}

\newcommand{\bdry}{\partial}
\newcommand{\pmbdry}[1]{\partial{#1}}

\newcommand{\lra}[1]{\stackrel{#1}{\longrightarrow}}

\newcommand{\add}{\operatorname{add}}
\newcommand{\Hom}{\operatorname{Hom}}
\newcommand{\stHom}{\operatorname{\underline{Hom}}}
\newcommand{\End}{\operatorname{End}}
\newcommand{\Ext}{\operatorname{Ext}}
\newcommand{\CM}{\operatorname{CM}}
\newcommand{\GP}{\operatorname{GP}}
\newcommand{\proj}{\operatorname{proj}}
\newcommand{\md}{\operatorname{mod}}

\newcommand{\wt}[1][]{\operatorname{wt}^{#1}}
\newcommand{\fgmod}{\operatorname{mod}}
\renewcommand{\top}{\operatorname{top}}
\newcommand{\gldim}{\operatorname{gldim}}

\newcommand{\grass}[2]{\operatorname{Gr}_{#1}^{#2}}

\newcommand{\C}{\mathcal{C}}
\newcommand{\CC}{\mathbb{C}}
\newcommand{\ZZ}{\mathbb{Z}}
\newcommand{\Ftil}{\widetilde{F}}

\renewcommand{\epsilon}{\varepsilon}

 \newcommand{\clus}{\mathfrak{C}} \newcommand{\posid}{\mathfrak{P}}
\newcommand{\openposidvar}{\Pi^\circ}
\newcommand{\openposidcone}{\widetilde{\Pi}^\circ}

\newcommand{\cluschar}[1]{\Phi_{#1}}
\newcommand{\clucha}[2]{\Phi_{#1}(#2)}
\newcommand{\jkscluschar}{\Psi}
\newcommand{\jksclucha}[1]{\Psi(#1)}

\newcommand{\cluva}{x}
\newcommand{\cto}{T}
\newcommand{\plabic}{\Gamma}

\newcommand{\projcov}{P}
\newcommand{\canprojcov}[1][]{\mathbf{P}^{#1}}
\newcommand{\syz}{\Omega}
\newcommand{\cansyz}[1][]{\mathbf{\Omega}^{#1}}

\newcommand{\compl}[1]{\widehat{#1}}

\newcommand{\powser}[2]{#1[\hspace{-0.1em}[#2]\hspace{-0.1em}]}
\newcommand{\fracpowser}[2]{#1(\hspace{-0.1em}(#2)\hspace{-0.1em})}
\newcommand{\im}{\operatorname{im}}
\newcommand{\id}{\mathrm{id}}
\newcommand{\Gr}{\operatorname{Gr}}

\DeclareMathOperator{\rk}{rk}

\newcommand{\Qcyc}{Q_{\text{cyc}}}

\newcommand{\jksalg}{C}

\newcommand{\Pluck}[1]{\varphi_{#1}}
\newcommand{\twPluck}[1]{\overleftarrow{\varphi_{#1}}}
\newcommand{\twist}[1]{\overleftarrow{#1}}
\DeclareMathOperator{\MuSptw}{tw}

\newcommand{\isoto}{\stackrel{\sim}{\longrightarrow}}

\newcommand{\Kgp}{\mathrm{K}_0}

\newcommand{\comp}{^{\mathrm{c}}}

\newcommand{\smallF}{F'}

\newcommand{\black}{\bullet}
\newcommand{\white}{\circ}

\newcommand{\res}[1]{\xi_{#1}}

\newcommand{\K}{\mathrm{K}}
\renewcommand{\H}{\mathrm{H}}
\renewcommand{\d}{\mathrm{d}}

\newcommand{\openclustalg}[1]{\mathcal{A}_{#1}}
\newcommand{\clustalg}[1]{\overline{\mathcal{A}}_{#1}}

\newcommand{\MSform}[2][]{\operatorname{MS}#1(#2)}

\newcommand{\MSmat}{\mathfrak{m}}

\begin{document}
\begin{abstract}

Cluster algebra structures for Grassmannians and their (open) positroid strata 
are controlled by a Postnikov diagram $D$ or, equivalently, a dimer model on the disc, 
as encoded by either a bipartite graph or the dual quiver (with faces).
The associated dimer algebra $A$, determined directly by the quiver with a certain potential,
can also be realised as the endomorphism algebra of a cluster-tilting object in an
associated Frobenius cluster category.

In this paper, we introduce a class of $A$-modules corresponding to perfect matchings 
of the dimer model of $D$
and show that, when $D$ is connected, the indecomposable projective $A$-modules are in this class. 
Surprisingly, this allows us to deduce that the cluster category associated to $D$ 
embeds into the cluster category for the appropriate Grassmannian. 
We show that the indecomposable projectives correspond to certain matchings which have appeared previously in work of Muller--Speyer. This allows us to identify 
the cluster-tilting object associated to $D$, by showing that it is determined by 
one of the standard labelling rules constructing a cluster of Pl\"ucker coordinates from $D$. 
By computing a projective resolution of every perfect matching module, 
we show that Marsh--Scott's formula for twisted Pl\"ucker coordinates, 
expressed as a dimer partition function, is a special case of the general cluster character formula, and thus observe that the Marsh--Scott twist can be categorified by
a particular syzygy operation in the Grassmannian cluster category.
\end{abstract}

\maketitle

\section{Introduction} \label{sec:intro}

A key example of a cluster algebra (with frozen variables) is given by Scott's cluster structure \cite{Sco} on the homogeneous coordinate ring $\CC[\grass{k}{n}]$ of the Grassmannian of $k$-planes in $\CC^n$. 
The Pl\"ucker coordinates $\Pluck{I}$, for $I$ a $k$-subset of $\{1,\dotsc,n\}$ 
(written $I\in\binom{n}{k}$ below), are cluster variables of this cluster algebra, 
and a set $\{\Pluck{I}:I\in\clus\}$ of Pl\"ucker coordinates is a cluster if and only if
 $\clus\subset\binom{n}{k}$ is maximal with respect to the property that its elements are pairwise non-crossing \cite[Def.~3]{Sco} (or weakly separated \cite{LZ}). 
The frozen variables, which appear in every cluster, are the Pl\"ucker coordinates 
$\Pluck{I}$ for $I=\{i,\dotsc,i+k-1\}$ a cyclic interval in $\{1,\dotsc,n\}$,
considered modulo $n$.

Recently, it has been shown by Galashin--Lam \cite{GL} (see also \cite{SSBW}) that the coordinate rings of open positroid varieties $\openposidvar(\posid)$ in $\grass{k}{n}$ also have cluster algebra structures. 
These varieties are defined from a positroid $\posid\subset\binom{n}{k}$, and consist of those points in $\grass{k}{n}$ on which the Pl\"ucker coordinates $\Pluck{I}$ with $I\notin\posid$ vanish, while another set of Pl\"ucker coordinates depending on $\posid$
(the frozen variables in the cluster algebra structure) do not vanish. 
Again, the cluster algebra has a cluster $\{\Pluck{I}:I\in\clus\}$ of (restricted) Pl\"ucker coordinates for each 
maximal non-crossing subset $\clus$ of $\posid$ containing the indices of these frozen variables.

Both for the full Grassmannian and for open positroid varieties, the quivers of these clusters of Pl\"ucker coordinates are described via Postnikov (alternating strand) diagrams, which are given by a collection of $n$ strands in a disc with $n$ marked points on its boundary, satisfying various consistency conditions, as we recall in Section~\ref{sec:clucats-dimers}.
Such a diagram $D$ is also equivalent to the data of a bipartite graph $\Gamma(D)$ in the interior of the disc, joined to the $n$ marked points on the boundary by `half-edges'. 
To describe a cluster in $\CC[\grass{k}{n}]$, each strand in $D$ should go from the $i$-th point to the $(i+k)$-th point.
We call this a \emph{uniform} strand permutation 
and any corresponding diagram is a \emph{uniform} Postnikov diagram. 
Diagrams with non-uniform strand permutations define clusters for more general open positroid varieties.

The cluster algebra structure on $\CC[\grass{k}{n}]$ has a categorical model given by the Frobenius cluster category $\CM(\jksalg)$, introduced by Jensen--King--Su \cite{JKS}, whose objects are Cohen--Macaulay modules over a Gorenstein order $\jksalg=\jksalg_{k,n}$, that is, modules free over a central subalgebra $Z=\powser{\CC}{t}$. Each Pl\"ucker coordinate $\Pluck{I}$ of $\CC[\grass{k}{n}]$, for $I\in\binom{n}{k}$, corresponds to a certain `rank $1$' $\jksalg$-module $M_I$. The indecomposable projective-injective objects of $\CM(\jksalg)$ are $M_I$ for $I$ a cyclic interval, and so these objects are in bijection with the frozen variables of the cluster algebra. A direct relationship between the category $\CM(\jksalg)$ and uniform Postnikov diagrams is given by Baur--King--Marsh \cite{BKM}.

More recently, Pressland \cite{Pre3} has shown that the cluster algebra describing the coordinate ring of any open positroid variety associated to a connected Postnikov diagram admits a similar categorification, given by a category $\GP(B)$ of Gorenstein projective modules over an Iwanaga--Gorenstein ring $B$ (depending on the positroid).

Our goal here is to further develop the dictionary between these combinatorial and categorical constructions, with a focus on labelling rules (for example, those determining which collection of Pl\"ucker coordinates forms the initial cluster attached to a Postnikov diagram) and on twist maps as described by Marsh--Scott \cite{MaSc} and Muller--Speyer \cite{MuSp}. While many of our results are new even for uniform diagrams, we also extend some results of Baur--King--Marsh to the case of general connected Postnikov diagrams. Most notably, this applies to \cite[Thm.~10.3]{BKM}, which realises the dimer algebra of a Postnikov diagram as an endomorphism algebra in $\CM(C)$. Our main results in more detail are as follows.

\subsection{Vertex labelling and projective modules}
\label{ss:labels}
The set $\clus$, indexing an initial cluster of Pl\"ucker coordinates, is obtained from a Postnikov diagram $D$ via the following rule. For each boundary marked point $i$, write $i$ on each vertex of the quiver $Q=Q(D)$ (equivalently, each alternating region of $D$) to the left of the strand starting at $i$. At the end of this process, each vertex $j\in Q_0$ has a label $I_j\in\binom{n}{k}$, and we take $\clus=\{I_j:j\in Q_0\}$.

One of our main results is that this labelling rule is a natural one from a representation-theoretic point of view, at least when the Postnikov diagram $D$ is connected. From the quiver $Q$, 
one can define a frozen Jacobian algebra $A=A_D$, and the boundary algebra $B=eAe$, where $e$ is a suitable `boundary' idempotent, is then used to define the categorification $\GP(B)$ from \cite{Pre3}. We show here that $B$ contains Jensen--King--Su's algebra $\jksalg$ as a subalgebra, and further that $\GP(B)\subset\CM(B)\subset\CM(\jksalg)$, 
where the second inclusion is strictly the fully faithful embedding given by restriction
(Proposition ~\ref{p:restriction}). The labelling rule to produce $\clus$ from $D$ is then explained with reference to the rank $1$ modules $M_I\in\CM(\jksalg)$, as follows.

\begin{theorem*}[Theorem~\ref{t:labels}]
\label{t:projective-labelling}
Let $D$ be a connected Postnikov diagram with quiver $Q(D)$ and dimer algebra $A=A_D$. For each vertex $j\in Q_0$, consider the $B$-module $eAe_j$ as a $\jksalg$-module via restriction, as above. Then there is an isomorphism
\[eAe_j\isom M_{I_j}.\]
\end{theorem*}

Pressland \cite{Pre3} has shown that
\begin{equation}\label{eq:TDdef}
T_D = eA = \bigoplus_{j\in Q_0} eAe_j
\end{equation}
is a cluster-tilting object in the categorification $\GP(B)$, and furthermore that $A\isom \End_B(T_D)\op$. It follows by Theorem~\ref{t:projective-labelling} that restricting $T_D$ to $\jksalg$ produces the object
\begin{equation}\label{eq:TCdef}
T_\clus = \bigoplus_{I\in\clus} M_I = \bigoplus_{j\in Q_0} M_{I_j}
\end{equation}
and thus we also have
\begin{equation}
A\isom \End_C(T_\clus)\op.
\end{equation}
This generalises a result of Baur--King--Marsh \cite[Thm.~10.3]{BKM} from the uniform case to arbitrary connected Postnikov diagrams.  The labelling rule to produce $\clus$ has several natural variations, by replacing `left' by `right' and / or `source' by `target'. We discuss the representation-theoretic analogues of these in Remarks~\ref{r:target-labels} and \ref{r:BKM-comparison}.

By using similar ideas to those behind Theorem~\ref{t:projective-labelling}, we may further show (Proposition~\ref{pro:positroid}) that the rank 1 modules $M_I\in\CM(\jksalg)$ lying in 
the subcategory $\CM(B)$ are precisely those with $I\in\posid$.

\subsection{Perfect matching modules}
The main new ingredient in this paper is the introduction of a \emph{perfect matching module} $N_\mu\in\CM(A)$ attached to a perfect matching on the bipartite graph $\Gamma(D)$, or  equivalently (using Definition~\ref{d:pm}) on the quiver $Q(D)$. These provide a combinatorial description of the rank $1$ modules in $\CM(A)$, analogous to that of the rank $1$ modules $M_I\in\CM(C)$ (Corollary~\ref{c:proj-pm-mods}). The restriction of $N_\mu$ to the boundary algebra $B$, and further to $C$, is encoded combinatorially by the boundary value $\bdry\mu$ (Definition~\ref{d:bdry-value}) of the matching $\mu$; see Proposition~\ref{p:res-pm-mods} for a precise statement.

The consistency conditions for Postnikov diagrams imply that 
the projective $A$-modules $Ae_j$ are rank 1, and hence must be perfect matching modules. 
One key result of the paper, and the main step in the proof of Theorem~\ref{t:projective-labelling}, is to identify the corresponding matchings explicitly. 
\begin{theorem*}[Corollary~\ref{c:MSmatch}]
\label{t:MSmatch-intro}
Let $D$ be a connected Postnikov diagram with dimer algebra $A=A_D$. Then for each vertex $j\in Q_0$ we have
\[Ae_j\isom N_{\MSmat_j},\]
where $\MSmat_j$ is an explicit perfect matching defined by Muller--Speyer~\cite{MuSp}, see \eqref{eq:MSmat}.
\end{theorem*} 
The matching $\MSmat_j$ has boundary value $\pmbdry\MSmat_j=I_j\in\clus$ by construction, so combining Theorem~\ref{t:MSmatch-intro} and Proposition~\ref{p:res-pm-mods} leads directly to Theorem~\ref{t:projective-labelling}.

The core technical result used to prove Theorem~\ref{t:MSmatch-intro} and other main results in the paper is the determination 
of a projective resolution of every perfect matching module $N_\mu$ (Theorem~\ref{thm:match-res})
and consequently its class $[N_\mu]$ in the Grothendieck group $\Kgp(\proj A)$ 
(Proposition~\ref{prop:Nmu-class}).

\subsection{Cluster characters and twists}
An important use of perfect matchings is to define partition functions for dimer models.
For example, Marsh--Scott~\cite{MaSc} defined
\begin{equation}
\label{eq:MS1-intro}
  \MSform[^\white]{I} = x^{-\wt(D)}\sum_{\mu:\pmbdry{\mu}=I} x^{\wt[\white](\mu)}, 
\end{equation}
for $I\in\binom{n}{k}$, where $\wt(D)$ and $\wt[\white](\mu)$ are certain elements of $\Kgp(\proj A)$,
so this expression is a (formal) Laurent polynomial
in the cluster algebra associated to the diagram $D$.
Strictly speaking, Marsh--Scott only studied the uniform case, but their partition function \eqref{eq:MS1-intro} 
makes sense in the general case.
In the uniform case, 
under the substitution $x^{[Ae_j]} \mapsto \Pluck{I_j}$,
they showed that $\MSform[^\white]{I}$ is a twisted Plucker coordinate 
$\twPluck{I}\in\CC[\grass{k}{n}]$.

A further application of perfect matching modules in this paper is to give a representation-theoretic interpretation of the Marsh--Scott formula \eqref{eq:MS1-intro}.
Our first step in this direction (Proposition~\ref{prop:MS2}, Theorem~\ref{thm:MS3})
is to reformulate \eqref{eq:MS1-intro} in terms of modules:
\begin{equation} \label{eq:MS3-intro}
  \MSform[^\white]{I} = x^{[P_I^\white]}\sum_{\mu:\pmbdry{\mu}=I} x^{-[N_\mu]}
  = x^{[F\canprojcov[\white] M_I]}\sum_{\substack{N\leq F M_I \\eN=M_I}} x^{-[N]},
 \end{equation}
where $F=\Hom_B(T_D,-)\colon \CM(B)\to\CM(A)$ is the right adjoint functor to boundary restriction 
$N\mapsto eN$.
In addition $P_I^\white$ is a certain projective $A$-module constructed combinatorially from $I$, whereas $\canprojcov[\white] M_I$ is a (non-minimal) projective cover of $M_I$ in $\CM(B)$ with the property that $F\canprojcov[\white] M_I=P_I^\white$.
To get from \eqref{eq:MS1-intro} to \eqref{eq:MS3-intro}, we use Proposition~\ref{prop:Nmu-class}, 
to show that $[P_I^\white]-[N_\mu]=\wt[\white](\mu)-\wt(D)$, and Proposition~\ref{p:matchings-to-modules} (see also Remark~\ref{r:combinatorics}) to rewrite the summation set in terms of modules.

By a further transformation of \eqref{eq:MS3-intro} we obtain the following.
\begin{theorem*}[Theorem~\ref{thm:MS=CC}]
\label{t:MS=CC-intro}
If $M_I\in\CM(B)$ for some $I\in\binom{n}{k}$, then
\[\MSform[^\white]{I} =x^{[F\cansyz[\white] M_I]}\sum_{E\leq G\cansyz[\white] M_I}x^{-[E]},\]
where $G=\Ext^1_B(T_D,-)\colon \CM(B)\to\CM(A)$ and 
$\cansyz[\white] M = \ker (\canprojcov[\white] M\to M)$. That is, $\MSform[^\white]{I}$ is given by Fu--Keller's version \cite{FK} of the Caldero--Chapoton formula \cite{CC} 
for the cluster character of $\cansyz[\white] M_I$, using the cluster-tilting object $T_D$ in $\GP(B)$.
\end{theorem*}

In the uniform case, the cluster character of $M_I$ is the Pl\"ucker coordinate $\Pluck{I}$, 
under the same substitution $x^{[Ae_j]} \mapsto \Pluck{I_j}$.
Thus (see~\eqref{eq:CC=twPluck}) the Marsh--Scott twist on Pl\"ucker coordinates is categorified 
by the syzygy $\cansyz[\white]$ on rank 1 modules in $\CM(C)$. While Fu--Keller's formula makes sense as a function on $\CM(B)$, it is only a cluster character when restricted to $\GP(B)$. However, we may show (Lemma~\ref{lem:syzygy}) that in fact $\cansyz[\white] M_I\in\GP(B)$, justifying the language used in stating Theorem~\ref{t:MS=CC-intro}.

Muller--Speyer~\cite{MuSp} define a slightly different twist, including in the non-uniform case, 
which we also relate to a syzygy via a cluster character formula (Theorem~\ref{t:MuSp=CC}).
This relationship is used in \cite[Thm.~7.2]{Pre4} to prove both that this twist is a quasi-cluster morphism in the sense of Fraser \cite{Fraser}, and the related fact \cite[Thm.~6.16]{Pre4} that the source and target labelling conventions for a Postnikov diagram (see \S\ref{ss:labels}) lead to quasi-coincident cluster algebra structures for the open positroid variety.
This second statement was originally conjectured by Muller--Speyer \cite[Rem.~4.7]{MuSp} (see also \cite[Conj.~1.1]{FSB}).
Alternative proofs of these facts, using different methods, can be found in \cite[Thm.~B, Cor.~8.1]{CLSBW}.

\subsection{Structure of the paper}
In Section~\ref{sec:clucats-dimers}, we describe the notion of a consistent dimer model on a disc, 
in terms of a Postnikov diagram $D$, a bipartite graph $\plabic(D)$ or a quiver with faces $Q(D)$.
We also describe the dimer algebra $A=A_D$, introduced in \cite{BKM}.
A fundamental invariant of a dimer model is its type $(k,n)$ (Definition~\ref{d:type}). In Section~\ref{sec:bdry-alg}, we describe the two algebras associated to the boundary of the dimer model.
The first is $B=eAe$, for $e$ a certain boundary idempotent in $A$, and the second is the algebra $\jksalg=\jksalg_{k,n}$ introduced in \cite{JKS} to categorify the Grassmannian cluster algebra $\CC[\grass{k}{n}]$.

In Section~\ref{sec:PM-mods}, we explain how a perfect matching on $Q(D)$ 
determines a module $N_\mu$ for the corresponding dimer algebra $A$. 
In the context of Postnikov diagrams, this allows us to prove Proposition~\ref{p:restriction}, showing that $\jksalg$ is canonically a subalgebra of $B$,
in such a way that the restriction map $\CM(B)\to \CM(C)$ is fully faithful.
Thus Cohen--Macaulay $B$-modules are effectively a special class of Cohen--Macaulay $C$-modules.

In Section~\ref{sec:ind-res}, we study the `induction-restriction' relationship between modules for the algebras $B$ and $A$. We recall results from elsewhere showing that $A$ 
is the endomorphism algebra of a cluster-tilting object $T\in\GP(B)$,
when the Postnikov diagram is connected.

In Section~\ref{sec:proj-res} we use the combinatorics of $Q(D)$ to write down a projective resolution of $N_\mu$. 
This projective resolution is used in Section~\ref{sec:MuSp} to show that Muller--Speyer matchings correspond to projective modules, yielding Theorem~\ref{t:MSmatch-intro}. 
In Section~\ref{sec:labels}, we deduce that the combinatorial labelling of Postnikov diagrams agrees with the categorical labelling arising from restricting projective $A$-modules to the boundary, thus proving Theorem~\ref{t:projective-labelling}.

Sections~\ref{sec:newMS} and \ref{sec:newCC} recall the Marsh--Scott formula and Fu--Keller's cluster character. In Section~\ref{sec:MS=CC} we relate these by proving Theorem~\ref{t:MS=CC-intro}, in particular showing that the twisted Pl\"ucker coordinate $\twPluck{I}$ 
is the cluster character $\clucha{T}{\cansyz[\white] M_I}$, 
where $\cansyz[\white] M_I$ is a particular syzygy of the $\jksalg$-module $M_I$. Finally, in Section~\ref{sec:MuSp=CC}, we relate Muller--Speyer's twist for more general open positroid varieties to the cluster character formula.

\subsection*{Acknowledgements}
We thank Bernt Tore Jensen, Xiuping Su, Chris Fraser and Melissa Sherman-Bennett for fruitful discussions, and are particularly grateful to Bernt Tore Jensen for the proof of Lemma~\ref{lem:syzygy}. 
Important progress on this project was made during visits to Universit\"at Bonn in 2016 and Universit\"at Stuttgart in 2018, and we thank Jan Schr\"oer and Steffen Koenig for facilitating these. 
We are also grateful for support and hospitality during
the \emph{20 Years of Cluster Algebras} conference at CIRM, Luminy in 2018, 
the cluster algebras programme in Kyoto in 2019, 
the Hausdorff School on stability conditions in Bonn, also in 2019,
and the \emph{Cluster algebras and representation theory} programme in 2021
at the Isaac Newton Institute for Mathematical Sciences (supported by EPSRC grant no EP/R014604/1).
We thank the anonymous referee for suggesting various improvements to the text.

At different stages of the project, the first author was supported by EPSRC grants EP/K026364/1, EP/N005457/1 and EP/P016014/1, while the third author was supported by a fellowship from the Max-Planck-Gesellschaft and the EPSRC postdoctoral fellowship EP/T001771/1.

\section{Grassmannian cluster categories and dimer models} \label{sec:clucats-dimers}

In this section we introduce the related notions of Postnikov diagrams and dimer models with boundary. Our exposition largely follows \cite[\S2]{BKM}, but at a slightly higher level of generality.

Let $\C=(\C_0,\C_1)$ be a circular graph with vertex set 
$\C_0$ and edge set $\C_1$, both of size~$n$.
We will often label the edges with $\{1,\dotsc,n\}$, in cyclic order, but will not explicitly label the vertices.
The case $n=7$ is illustrated in Figure~\ref{f:graphC}.

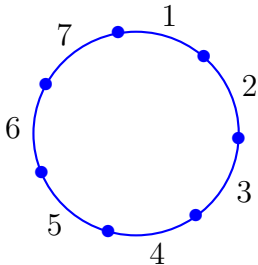
\begin{figure}[h]
\begin{tikzpicture} [scale=0.90]  
\pgfmathsetmacro{\seventh}{360/7}
\pgfmathsetmacro{\cstart}{100}
\pgfmathsetmacro{\crad}{1.5}
\pgfmathsetmacro{\labrad}{1.2*\crad}
\draw [blue,thick] (0,0) circle(\crad);
\foreach \n in {1,...,7}
{ \draw (\cstart-\seventh*\n:\crad) node [blue] {$\bullet$};
  \draw (\cstart+0.5*\seventh-\seventh*\n:\labrad) node [black] {\n}; }
\end{tikzpicture}
\caption{The circular graph $\C$.}
\label{f:graphC}
\end{figure}

\begin{definition} \label{d:asd}
Consider a disc with $n$ marked points on its boundary, identified with $\C_1$ in the same cyclic order.
A \emph{Postnikov} (or \emph{alternating strand}) \emph{diagram} $D$ consists of a set of $n$ oriented curves in the disc, 
called \emph{strands}, 
connecting the boundary marked points, such that each marked point is incident with one incoming and one outgoing strand. 
The following axioms must be satisfied.
\goodbreak
\emph{Local axioms:}
\begin{enumerate}[({a}1)]
\item Only two strands can cross at a given point and all crossings are
transverse.
\item There are finitely many crossing points.
\item
Proceeding along a given strand, the other strands crossing it alternate between crossing it left to right and right to left.
\end{enumerate}
\goodbreak
\emph{Global axioms:}
\begin{enumerate}[({b}1)]
\item A strand cannot cross itself.
\item If two strands cross at distinct points $U$ and $V$, then one strand is oriented from $U$ to $V$ and the other is oriented from $V$ to $U$.
\end{enumerate}

For axioms~(a3) and (b2), the two strands meeting at a marked point are regarded as crossing at this point in the
obvious way. We call $D$ \emph{connected} if the union of its strands is a connected set. An example of a connected Postnikov diagram is shown in Figure~\ref{f:postfree37}.
\end{definition}

\begin{figure}[h]
\begin{tikzpicture}[scale=3,baseline=(bb.base),yscale=-1]

\path (0,0) node (bb) {};

\draw [boundary] (0,0) circle(1.0);

\foreach \n/\m/\a in {1/4/0, 2/3/0, 3/2/5, 4/1/10, 5/7/0, 6/6/-3, 7/5/0}
{ \coordinate (b\n) at (\bstart-\seventh*\n+\a:1.0);
  \draw (\bstart-\seventh*\n+\a:1.1) node {$\m$}; }

\foreach \n/\m in {8/1, 9/2, 10/3, 11/4, 14/5, 15/6, 16/7}
  {\coordinate (b\n) at ($0.65*(b\m)$);}

\coordinate (b13) at ($(b15) - (b16) + (b8)$);
\coordinate (b12) at ($(b14) - (b15) + (b13)$);

\foreach \n/\x\y in {13/-0.03/-0.03, 12/-0.22/0.0, 14/-0.07/-0.03, 11/0.05/0.02, 16/-0.02/0.02}
  {\coordinate (b\n)  at ($(b\n) + (\x,\y)$); } 

\foreach \e/\f/\t in {2/9/0.5, 4/11/0.5, 5/14/0.5, 7/16/0.5, 
 8/9/0.5, 9/10/0.5, 10/11/0.5,11/12/0.5, 12/13/0.45, 8/13/0.6, 
 14/15/0.5, 15/16/0.6, 12/14/0.45, 13/15/0.4, 8/16/0.6}
{\coordinate (a\e-\f) at ($(b\e) ! \t ! (b\f)$); }

\draw [strand] plot[smooth]
coordinates {(b1) (a8-16) (a15-16) (b6)}
[postaction=decorate, decoration={markings,
 mark= at position 0.2 with \strarrow,
 mark= at position 0.5 with \strarrow, 
 mark= at position 0.8 with \strarrow }];
 
\draw [strand] plot[smooth]
coordinates {(b6) (a14-15) (a12-14)(a11-12) (a10-11) (b3)}
[postaction=decorate, decoration={markings,
 mark= at position 0.15 with \strarrow, mark= at position 0.35 with \strarrow,
 mark= at position 0.53 with \strarrow, mark= at position 0.7 with \strarrow,
 mark= at position 0.87 with \strarrow }];
 
\draw [strand] plot[smooth]
coordinates {(b3) (a9-10) (a8-9) (b1)}
[postaction=decorate, decoration={markings,
 mark= at position 0.2 with \strarrow,
 mark= at position 0.5 with \strarrow, 
 mark= at position 0.8 with \strarrow }];

\draw [strand] plot[smooth]
coordinates {(b2) (a9-10) (a10-11) (b4)}
 [postaction=decorate, decoration={markings,
 mark= at position 0.2 with \strarrow,
 mark= at position 0.5 with \strarrow, 
 mark= at position 0.8 with \strarrow }];

\draw [strand] plot[smooth]
coordinates {(b4) (a11-12) (a12-13) (a13-15) (a15-16) (b7)}
[postaction=decorate, decoration={markings,
 mark= at position 0.15 with \strarrow, mark= at position 0.35 with \strarrow,
 mark= at position 0.55 with \strarrow, mark= at position 0.7 with \strarrow,
 mark= at position 0.87 with \strarrow }];

\draw [strand] plot[smooth]
coordinates {(b7) (a8-16) (a8-13) (a12-13) (a12-14) (b5)}
[postaction=decorate, decoration={markings,
 mark= at position 0.15 with \strarrow, mark= at position 0.315 with \strarrow,
 mark= at position 0.5 with \strarrow, mark= at position 0.7 with \strarrow,
 mark= at position 0.88 with \strarrow }];

\draw [strand] plot[smooth]
coordinates {(b5) (a14-15) (a13-15) (a8-13) (a8-9) (b2)}
[postaction=decorate, decoration={markings,
 mark= at position 0.13 with \strarrow, mark= at position 0.33 with \strarrow,
 mark= at position 0.5 with \strarrow, mark= at position 0.7 with \strarrow,
 mark= at position 0.88 with \strarrow }];

\end{tikzpicture}

\caption{A Postnikov diagram.}
\label{f:postfree37}
\end{figure}
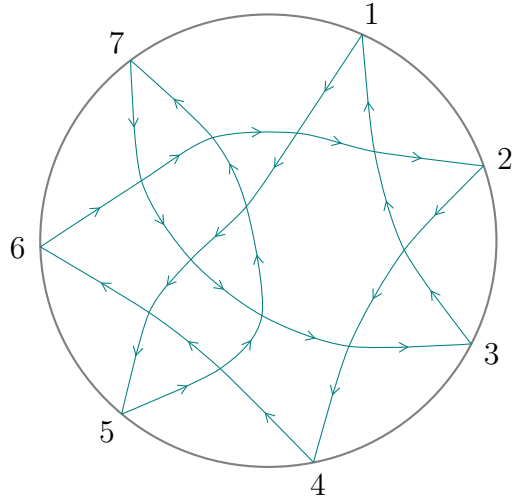

A Postnikov diagram divides the interior of the disc into regions, the connected components of the complement of the union of the strands. A region is \emph{alternating} if the strands incident with it alternate in orientation going around its boundary. 
It is \emph{oriented} if the strands around its boundary are all oriented clockwise, or all anticlockwise. 
It is easy to see that every region of a Postnikov diagram must be alternating or oriented.

An alternating region with an edge on the boundary is called a \emph{boundary} region, otherwise it is an \emph{internal} region.
The labelling of boundary points by $\C_1$ gives a canonical map from $\C_0$ onto the set of boundary regions.
This map is a bijection when the Postnikov diagram is connected, 
so that each boundary region meets the boundary in a single edge.

A Postnikov diagram $D$ determines a permutation $\pi_D$ of $\C_1$, with $\pi_D(i)=j$ when the strand starting at $i$ ends at $j$. If $\pi_D(i)=i+k$ (modulo $n$) for some fixed $k$, we call $D$ a $(k,n)$-diagram. These diagrams will play a special role for us, since they are related to the $(k,n)$-Grassmannian cluster algebra and its categorification by Jensen--King--Su \cite{JKS}. However, most of our results apply to more general diagrams, which describe cluster structures on more general positroid varieties \cite{GL,SSBW}, categorified in \cite{Pre3}.

\begin{definition}
A \emph{lollipop} in a Postnikov diagram $D$ is a strand starting and ending at the same point on the disc, corresponding to a fixed point of $\pi_D$.
\end{definition}

\begin{proposition}
\label{p:no-big-lollipops}
A lollipop has no crossings with other strands of $D$.
\end{proposition}
\begin{proof}
Let $s$ be a lollipop, and consider its first crossing with another strand $s'$. By (b1), $s$ is a simple closed curve, and by (a3) strand $s'$ must cross into the inside of this curve. To end on the boundary $s'$ must therefore cross $s$ a second time, later on~$s$. But this violates (b2), hence we obtain a contradiction.
\end{proof}

As a result, connected Postnikov diagrams have no lollipops, except in the most degenerate case $n=1$.

The information in a Postnikov diagram may also be encoded in a reduced plabic (planar bicoloured) graph in the disc, 
as in~\cite[\S11--14]{postnikov}. 
For our purposes it is sufficient to assume that this graph is actually bipartite.

\begin{definition} \label{d:Post-to-bipartite}
To any Postnikov diagram $D$, there is an associated bipartite graph $\plabic(D)$ embedded into the disc, defined as follows.
The nodes correspond to the oriented regions of $D$ and are coloured black or white when the boundary of the region is oriented anticlockwise or clockwise, respectively, and the internal edges of $\plabic(D)$ correspond to the points of intersection of pairs of oriented regions. 
We call the nodes corresponding to regions meeting the boundary of the disc \emph{boundary nodes}, 
and the others \emph{internal nodes}. 
We also include in $\plabic(D)$ the data of \emph{half-edges}, which connect each boundary node to the marked points on the boundary that its corresponding region meets. 
We label the half-edges by $\C_1$, so that half-edge $i$ meets marked point $i$. 
The tiles of $\plabic(D)$, i.e.\ the connected components of its complement in the disc, 
correspond to the alternating regions of $D$.
\end{definition}

\begin{figure}[h]
\begin{tikzpicture}[scale=3,baseline=(bb.base),yscale=-1]

\path (0,0) node (bb) {};

\draw [boundary] (0,0) circle(1.0);

\foreach \n/\m/\a in {1/4/0, 2/3/0, 3/2/5, 4/1/10, 5/7/0, 6/6/-3, 7/5/0}
{ \coordinate (b\n) at (\bstart-\seventh*\n+\a:1.0);
  \draw (\bstart-\seventh*\n+\a:1.1) node {$\m$}; }

\foreach \n/\m in {8/1, 9/2, 10/3, 11/4, 14/5, 15/6, 16/7}
  {\coordinate (b\n) at ($0.65*(b\m)$);}

\coordinate (b13) at ($(b15) - (b16) + (b8)$);
\coordinate (b12) at ($(b14) - (b15) + (b13)$);

\foreach \n/\x\y in {13/-0.03/-0.03, 12/-0.22/0.0, 14/-0.07/-0.03, 11/0.05/0.02, 16/-0.02/0.02}
  {\coordinate (b\n)  at ($(b\n) + (\x,\y)$); } 

\foreach \h/\t in {1/8, 2/9, 3/10, 4/11, 5/14, 6/15, 7/16, 
 8/9, 9/10, 10/11,11/12, 12/13, 13/8, 14/15, 15/16, 12/14, 13/15, 8/16}
{ \draw [bipedge] (b\h)--(b\t); }

\foreach \n in {8,10,12,15} 
  {\draw [\graphcolor] (b\n) circle(\dotrad) [fill=white];} \foreach \n in {9,11,13, 14,16}  
  {\draw [\graphcolor] (b\n) circle(\dotrad) [fill=\graphcolor];} 

\foreach \e/\f/\t in {2/9/0.5, 4/11/0.5, 5/14/0.5, 7/16/0.5, 
 8/9/0.5, 9/10/0.5, 10/11/0.5,11/12/0.5, 12/13/0.45, 8/13/0.6, 
 14/15/0.5, 15/16/0.6, 12/14/0.45, 13/15/0.4, 8/16/0.6}
{\coordinate (a\e-\f) at ($(b\e) ! \t ! (b\f)$); }

\draw [strand] plot[smooth]
coordinates {(b1) (a8-16) (a15-16) (b6)}
[postaction=decorate, decoration={markings,
 mark= at position 0.2 with \strarrow,
 mark= at position 0.5 with \strarrow, 
 mark= at position 0.8 with \strarrow }];
 
\draw [strand] plot[smooth]
coordinates {(b6) (a14-15) (a12-14)(a11-12) (a10-11) (b3)}
[postaction=decorate, decoration={markings,
 mark= at position 0.15 with \strarrow, mark= at position 0.35 with \strarrow,
 mark= at position 0.53 with \strarrow, mark= at position 0.7 with \strarrow,
 mark= at position 0.87 with \strarrow }];
 
\draw [strand] plot[smooth]
coordinates {(b3) (a9-10) (a8-9) (b1)}
[postaction=decorate, decoration={markings,
 mark= at position 0.2 with \strarrow,
 mark= at position 0.5 with \strarrow, 
 mark= at position 0.8 with \strarrow }];

\draw [strand] plot[smooth]
coordinates {(b2) (a9-10) (a10-11) (b4)}
 [postaction=decorate, decoration={markings,
 mark= at position 0.2 with \strarrow,
 mark= at position 0.5 with \strarrow, 
 mark= at position 0.8 with \strarrow }];

\draw [strand] plot[smooth]
coordinates {(b4) (a11-12) (a12-13) (a13-15) (a15-16) (b7)}
[postaction=decorate, decoration={markings,
 mark= at position 0.15 with \strarrow, mark= at position 0.35 with \strarrow,
 mark= at position 0.55 with \strarrow, mark= at position 0.7 with \strarrow,
 mark= at position 0.87 with \strarrow }];

\draw [strand] plot[smooth]
coordinates {(b7) (a8-16) (a8-13) (a12-13) (a12-14) (b5)}
[postaction=decorate, decoration={markings,
 mark= at position 0.15 with \strarrow, mark= at position 0.315 with \strarrow,
 mark= at position 0.5 with \strarrow, mark= at position 0.7 with \strarrow,
 mark= at position 0.88 with \strarrow }];

\draw [strand] plot[smooth]
coordinates {(b5) (a14-15) (a13-15) (a8-13) (a8-9) (b2)}
[postaction=decorate, decoration={markings,
 mark= at position 0.13 with \strarrow, mark= at position 0.33 with \strarrow,
 mark= at position 0.5 with \strarrow, mark= at position 0.7 with \strarrow,
 mark= at position 0.88 with \strarrow }];

 \end{tikzpicture}

\caption{The bipartite graph corresponding to the Postnikov diagram in Figure~\ref{f:postfree37}.}
\label{f:bipartite37}
\end{figure}
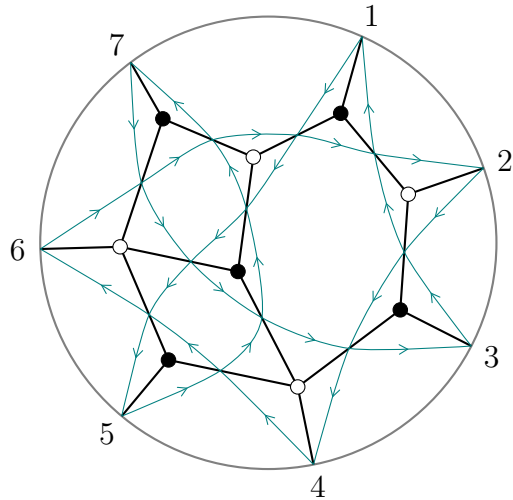

\begin{definition}[cf.~{\cite[\S3.1]{MuSp}}]
\label{d:type}
Let $D$ be a Postnikov diagram. The \emph{type} of $D$ is $(k,n)$, where
\begin{multline*}
k=\#\{\text{white nodes in $\plabic(D)$}\}-\#\{\text{black nodes in $\plabic(D)$}\}\\+\#\{\text{half-edges in $\plabic(D)$ incident with a black node}\},
\end{multline*}
and $n$ is the number of strands, or equivalently the number of half-edges in $\plabic(D)$.
\end{definition}

We will see in Section~\ref{sec:labels} that a $(k,n)$-diagram has type $(k,n)$. The Postnikov diagram in Figure~\ref{f:bipartite37} has type $(3,7)$, but it is not a $(3,7)$-diagram. In this example, each black boundary node is incident with a unique half-edge, so that
\[k=\#\{\text{white nodes}\}-\#\{\text{internal black nodes}\},\]
but this need not always be the case, as in the example in Figure~\ref{f:not-surj}.

We may also associate a quiver to any Postnikov diagram. Recall that a quiver $Q$ is a directed graph encoded by a tuple $Q=(Q_0,Q_1,h,t)$, where $Q_0$ is the
set of vertices, $Q_1$ is the set of arrows and $h,t\colon Q_1\to Q_0$,
so that each $\alpha\in Q_1$ is an arrow $t\alpha\to h\alpha$.
We will write $Q=(Q_0,Q_1)$, with the remaining data implicit, and we will also regard
it as an oriented 1-dimensional CW-complex. Given a quiver $Q$, we write $\Qcyc$ for the set of oriented cycles in $Q$
(up to cyclic equivalence).

\begin{definition}
\label{d:quiver-faces}
A \emph{quiver with faces} is a quiver $Q=(Q_0,Q_1)$, together with a set $Q_2$ of faces
and a map $\bdry\colon Q_2\to \Qcyc$,
which assigns to each $F\in Q_2$ its \emph{boundary} $\bdry F\in \Qcyc$.
\end{definition}

We shall often denote a quiver with faces by the same letter $Q$, regarded now as the
triple $(Q_0,Q_1,Q_2)$. We say that $Q$ is \emph{finite} if $Q_0$, $Q_1$
and $Q_2$ are all finite sets.
The number of times an arrow $\alpha\in Q_1$ appears in the boundaries of the faces in
$Q_2$ will be called the \emph{face multiplicity} of $\alpha$.
The (unoriented) \emph{incidence graph} of $Q$, at a vertex $i\in Q_0$,
has vertices given by the arrows incident with $i$. The edges between two
arrows $\alpha,\beta$ correspond to the paths of the form
\[ \lra{\alpha} i \lra{\beta} \]
occurring in the cycle $\bdry F$ for some face $F$.

\begin{definition} \label{d:dimermodel}
A (finite, connected, oriented) \emph{dimer model with boundary} is a finite connected
quiver with faces $Q=(Q_0,Q_1,Q_2)$, where $Q_2$ is written as disjoint union $Q_2=Q_2^+\cup Q_2^-$,
satisfying the following properties:
\begin{enumerate}[(a)]
\item the quiver $Q$ has no loops, i.e.\ no 1-cycles (but 2-cycles are allowed),
\item all arrows in $Q_1$ have face multiplicity $1$ (\emph{boundary} arrows) or $2$
(\emph{internal} arrows),
\item each internal arrow lies in a cycle bounding a face
in $Q_2^+$ and in a cycle bounding a face in $Q_2^-$,
\item the incidence graph of $Q$ at each vertex is non-empty and connected.
\end{enumerate}
Note that, by (b), each incidence graph in (d) must be either a line
(at a \emph{boundary} vertex) or an unoriented cycle (at an \emph{internal} vertex).
\end{definition}

In cluster algebras literature, internal vertices are usually called mutable and boundary vertices called frozen, terminology which is sometimes \cite{PreFJA} extended to internal and boundary arrows, but we opt here for the more geometric terms.

If we realise each face $F$ of a quiver with faces $Q$ as a polygon,
whose edges are labelled (cyclically) by the arrows in $\bdry F$,
then we may, in the usual way, form a topological space $|Q|$ by gluing together
the edges of the polygons labelled by the same arrows, in the manner indicated by the directions of the arrows.
If $Q$ is a dimer model with boundary then, arguing as in~\cite[Lemma 6.4]{bocklandt}, we see that conditions (b) and (d)
ensure that $|Q|$ is a surface with boundary, while (c) means that it can be oriented
by declaring the boundary cycles of faces in $Q_2^+$ to be oriented positive (or anticlockwise)
and those of faces in $Q_2^-$ to be negative (or clockwise).
Note also that each component of the boundary of $|Q|$ is (identified with) an
unoriented cycle of boundary arrows in $Q$.

On the other hand, suppose that we are given an embedding of a finite quiver $Q=(Q_0,Q_1)$
into a compact oriented surface $\Sigma$ with
boundary, such that the complement of $Q$ in $\Sigma$ is a disjoint union of discs, each of which
is bounded by a cycle in $Q$. Then we may make $Q$ into an oriented dimer model in the above
sense, for which $|Q|\isom \Sigma$, by setting $Q_2$ to be the set of connected components of the
complement of $Q$ in $\Sigma$, separated into $Q_2^+$ and $Q_2^-$
using the orientation of $\Sigma$.

\begin{definition} \label{d:quiver}
The quiver $Q(D)$ with faces of a Postnikov diagram $D$ has vertices $Q_0(D)$ given by the alternating regions of $D$. 
The arrows $Q_1(D)$ correspond to intersection points of two alternating regions, with orientation consistent with the strand orientation,
as in Figure~\ref{f:quiver37}.
We refer to the arrows between boundary vertices as \emph{boundary arrows}; these are naturally labelled by $\C_1$ in an analogous way to the half-edges of $\plabic(D)$. The faces $Q_2(D)$ are the cycles of arrows determined by an oriented region of $D$; these lie in $Q_2^+(D)$ if the region (equivalently the cycle) is oriented anticlockwise, and in $Q_2^-(D)$ if it is clockwise.
\end{definition}

As in \cite[Rem.~3.4]{BKM}, the quiver $Q(D)$ associated to a connected Postnikov diagram $D$ in a disc is naturally a dimer model in the disc as above---connectedness of $D$ is required for connectedness of incidence graphs as in Definition~\ref{d:dimermodel}(d). The Postnikov diagram is recovered as the collection of zig-zag paths of the dimer model; the global conditions (b1) and (b2) on the Postnikov diagram correspond to zig-zag consistency for the dimer model \cite[Thm.~5.5]{bocklandt12}, \cite[Defn.~3.5]{IU}.

We may also describe $Q(D)$, as a quiver with faces, directly and more combinatorially
as the dual of the bipartite graph $\plabic(D)$,
as in \cite[\S 2.1]{francopre12} for a general bipartite field theory.
In other words, $Q_0(D)$ is in bijection with the set of tiles of $\plabic(D)$ and $Q_1(D)$ with the set of edges,
with boundary arrows corresponding to half-edges.
An arrow joins the
two tiles in $\plabic(D)$ that share the corresponding edge and is oriented so that the black node is
on the left and/or the white node is on the right.
The faces (plaquettes in \cite{francopre12}) $F\in Q^+_2(D)$ correspond to the black nodes, while
those in $Q^-_2(D)$ correspond to the white nodes. For this reason, we will usually refer to the faces of a general dimer model with boundary as black, if they lie in $Q^+_2$, or white, if they lie in $Q^-_2$.
The boundary $\bdry F$ of a face $F$ is given by the arrows corresponding
to the edges incident with the node of $\plabic(D)$ corresponding to $F$,
ordered anticlockwise round black nodes and clockwise round white ones.
This duality is illustrated in Figure~\ref{f:quiver37}, for $D$ as in Figure~\ref{f:postfree37}.

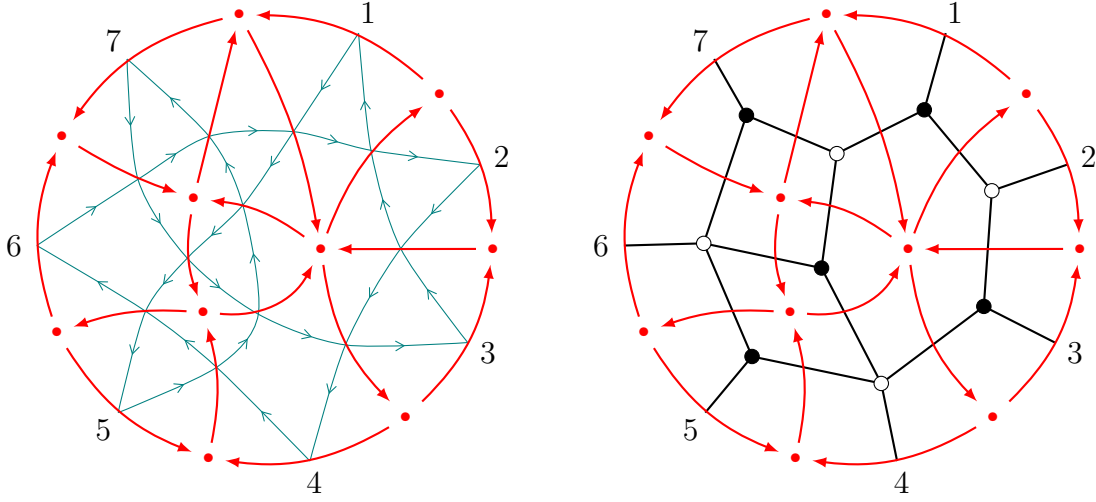
\begin{figure}[h]
\[
\begin{tikzpicture}[scale=3,baseline=(bb.base),yscale=-1]

\path (0,0) node (bb) {};

\foreach \n/\m/\a in {1/4/0, 2/3/0, 3/2/5, 4/1/10, 5/7/0, 6/6/-3, 7/5/0}
{ \coordinate (b\n) at (\bstart-\seventh*\n+\a:1.0);
  \draw (\bstart-\seventh*\n+\a:1.1) node {$\m$}; }

\foreach \n/\m in {8/1, 9/2, 10/3, 11/4, 14/5, 15/6, 16/7}
  {\coordinate (b\n) at ($0.65*(b\m)$);}

\coordinate (b13) at ($(b15) - (b16) + (b8)$);
\coordinate (b12) at ($(b14) - (b15) + (b13)$);

\foreach \n/\x\y in {13/-0.03/-0.03, 12/-0.22/0.0, 14/-0.07/-0.03, 11/0.05/0.02, 16/-0.02/0.02}
  {\coordinate (b\n)  at ($(b\n) + (\x,\y)$); } 

\foreach \e/\f/\t in {2/9/0.5, 4/11/0.5, 5/14/0.5, 7/16/0.5, 
 8/9/0.5, 9/10/0.5, 10/11/0.5,11/12/0.5, 12/13/0.45, 8/13/0.6, 
 14/15/0.5, 15/16/0.6, 12/14/0.45, 13/15/0.4, 8/16/0.6}
{\coordinate (a\e-\f) at ($(b\e) ! \t ! (b\f)$); }

\draw [strand] plot[smooth]
coordinates {(b1) (a8-16) (a15-16) (b6)}
[postaction=decorate, decoration={markings,
 mark= at position 0.2 with \strarrow,
 mark= at position 0.5 with \strarrow, 
 mark= at position 0.8 with \strarrow }];
 
\draw [strand] plot[smooth]
coordinates {(b6) (a14-15) (a12-14)(a11-12) (a10-11) (b3)}
[postaction=decorate, decoration={markings,
 mark= at position 0.15 with \strarrow, mark= at position 0.35 with \strarrow,
 mark= at position 0.53 with \strarrow, mark= at position 0.7 with \strarrow,
 mark= at position 0.87 with \strarrow }];
 
\draw [strand] plot[smooth]
coordinates {(b3) (a9-10) (a8-9) (b1)}
[postaction=decorate, decoration={markings,
 mark= at position 0.2 with \strarrow,
 mark= at position 0.5 with \strarrow, 
 mark= at position 0.8 with \strarrow }];

\draw [strand] plot[smooth]
coordinates {(b2) (a9-10) (a10-11) (b4)}
 [postaction=decorate, decoration={markings,
 mark= at position 0.2 with \strarrow,
 mark= at position 0.5 with \strarrow, 
 mark= at position 0.8 with \strarrow }];

\draw [strand] plot[smooth]
coordinates {(b4) (a11-12) (a12-13) (a13-15) (a15-16) (b7)}
[postaction=decorate, decoration={markings,
 mark= at position 0.15 with \strarrow, mark= at position 0.35 with \strarrow,
 mark= at position 0.55 with \strarrow, mark= at position 0.7 with \strarrow,
 mark= at position 0.87 with \strarrow }];

\draw [strand] plot[smooth]
coordinates {(b7) (a8-16) (a8-13) (a12-13) (a12-14) (b5)}
[postaction=decorate, decoration={markings,
 mark= at position 0.15 with \strarrow, mark= at position 0.315 with \strarrow,
 mark= at position 0.5 with \strarrow, mark= at position 0.7 with \strarrow,
 mark= at position 0.88 with \strarrow }];

\draw [strand] plot[smooth]
coordinates {(b5) (a14-15) (a13-15) (a8-13) (a8-9) (b2)}
[postaction=decorate, decoration={markings,
 mark= at position 0.13 with \strarrow, mark= at position 0.33 with \strarrow,
 mark= at position 0.5 with \strarrow, mark= at position 0.7 with \strarrow,
 mark= at position 0.88 with \strarrow }];

\foreach \n/\m\a in {1/124/0, 2/234/-1, 3/345/1, 4/456/10, 5/256/5, 6/267/0, 7/127/0}
{ \draw [\quivcolor] (\bstart+\seventh/2-\seventh*\n+\a:1) node (q\m) {\tiny $\bullet$}; }

\foreach \m/\a/\r in {247/\bstart/0.42, 245/10/0.25, 257/210/0.36}
{ \draw [\quivcolor] (\a:\r) node (q\m) {\tiny $\bullet$}; }

\foreach \t/\h/\a in {234/124/-20, 234/345/19, 456/345/-16, 456/256/22, 256/267/21, 127/267/-20,
 127/124/19, 124/247/13, 247/127/11, 245/234/18, 247/245/34, 257/247/14, 245/257/15,
 267/257/6, 257/256/0, 256/245/-9, 345/245/0, 245/456/-16}
{ \draw [quivarrow]  (q\t) edge [bend left=\a] (q\h); }

\end{tikzpicture}
\qquad \begin{tikzpicture}[scale=3,baseline=(bb.base),yscale=-1]

\path (0,0) node (bb) {};

\foreach \n/\m/\a in {1/4/0, 2/3/0, 3/2/5, 4/1/10, 5/7/0, 6/6/-3, 7/5/0}
{ \coordinate (b\n) at (\bstart-\seventh*\n+\a:1.0);
  \draw (\bstart-\seventh*\n+\a:1.1) node {$\m$}; }

\foreach \n/\m in {8/1, 9/2, 10/3, 11/4, 14/5, 15/6, 16/7}
  {\coordinate (b\n) at ($0.65*(b\m)$);}

\coordinate (b13) at ($(b15) - (b16) + (b8)$);
\coordinate (b12) at ($(b14) - (b15) + (b13)$);

\foreach \n/\x\y in {13/-0.03/-0.03, 12/-0.22/0.0, 14/-0.07/-0.03, 11/0.05/0.02, 16/-0.02/0.02}
  {\coordinate (b\n)  at ($(b\n) + (\x,\y)$); } 

\foreach \h/\t in {1/8, 2/9, 3/10, 4/11, 5/14, 6/15, 7/16, 
 8/9, 9/10, 10/11,11/12, 12/13, 13/8, 14/15, 15/16, 12/14, 13/15, 8/16}
{ \draw [bipedge] (b\h)--(b\t); }

\foreach \n in {8,10,12,15} 
  {\draw [\graphcolor] (b\n) circle(\dotrad) [fill=white];} \foreach \n in {9,11,13, 14,16}  
  {\draw [\graphcolor] (b\n) circle(\dotrad) [fill=\graphcolor];} 

\foreach \e/\f/\t in {2/9/0.5, 4/11/0.5, 5/14/0.5, 7/16/0.5, 
 8/9/0.5, 9/10/0.5, 10/11/0.5,11/12/0.5, 12/13/0.45, 8/13/0.6, 
 14/15/0.5, 15/16/0.6, 12/14/0.45, 13/15/0.4, 8/16/0.6}
{\coordinate (a\e-\f) at ($(b\e) ! \t ! (b\f)$); }

\foreach \n/\m\a in {1/124/0, 2/234/-1, 3/345/1, 4/456/10, 5/256/5, 6/267/0, 7/127/0}
{ \draw [\quivcolor] (\bstart+\seventh/2-\seventh*\n+\a:1) node (q\m) {\tiny $\bullet$}; }

\foreach \m/\a/\r in {247/\bstart/0.42, 245/10/0.25, 257/210/0.36}
{ \draw [\quivcolor] (\a:\r) node (q\m) {\tiny $\bullet$}; }

\foreach \t/\h/\a in {234/124/-20, 234/345/19, 456/345/-16, 456/256/22, 256/267/21, 127/267/-20,
 127/124/19, 124/247/13, 247/127/11, 245/234/18, 247/245/34, 257/247/14, 245/257/15,
 267/257/6, 257/256/0, 256/245/-9, 345/245/0, 245/456/-16}
{ \draw [quivarrow]  (q\t) edge [bend left=\a] (q\h); }

 \end{tikzpicture}
\]
\caption{The quiver and bipartite graph associated to the Postnikov diagram in Figure~\ref{f:postfree37}.}
\label{f:quiver37}
\end{figure}

\begin{remark}
The reverse of the above procedure can be used to exhibit an arbitrary dimer model with boundary $Q$ as the dual of a bipartite graph $\plabic$ in the surface $|Q|$, and it is this graph that is sometimes, more traditionally, called the dimer model \cite{HK}. When $Q=Q(D)$ is the quiver of a Postnikov diagram, the dual bipartite graph is precisely $\plabic(D)$ as in Definition~\ref{d:Post-to-bipartite}.
\end{remark}

\begin{remark}
\label{r:conventions}
Note that Marsh--Scott \cite{MaSc} associate white nodes of the bipartite graph to anticlockwise regions and black nodes to clockwise regions, 
whereas our convention is more consistent with the rest of the literature, e.g.\ \cite{fhkvw,MuSp}. 
Thus when quoting results from \cite{MaSc}, we will swap black and white, 
usually without further comment.
\end{remark}

\begin{definition} \label{d:dimeralgebra}
Given a dimer model with boundary $Q$,
we define the \emph{dimer algebra} $A_Q$
as follows.
For each internal arrow $\alpha\in Q_1$,
there are (unique) faces $F^+\in Q_2^+$ and $F^-\in Q_2^-$
such that $\bdry F^{\pm}=\alpha p^{\pm}_{\alpha}$,
for paths $p^+_{\alpha}$ and $p^-_{\alpha}$ from $h\alpha$ to $t\alpha$.
Then the dimer algebra $A_Q$ is the quotient of the complete path algebra
$\compl{\CC Q}$ by (the closure of) the ideal generated by relations
\begin{equation} \label{e:definingrelations}
p^+_{\alpha}=p^-_{\alpha},
\end{equation}
for internal arrows $\alpha\in Q_1$. When $D$ is a connected Postnikov diagram, so that $Q(D)$ is a dimer model with boundary, we abbreviate $A_D=A_{Q(D)}$.
\end{definition}

\begin{remark} \label{r:potential}
Note that the orientation is not strictly necessary to define $A_Q$;
we only need to know that $F^{\pm}$ are the two faces that contain the internal arrow
$\alpha$ in their boundaries, but not which is which.
On the other hand, given the orientation, we may also define a (super)potential $W_Q$ by the usual formula
(e.g.\ \cite[\S 2]{fhkvw})
\[
  W_Q=\sum_{F\in Q_2^+}\bdry F-\sum_{F\in Q_2^-}\bdry F.
\]
Then $A_Q$ may also be described as the quotient of
$\compl{\CC Q}$ by the so-called `F-term' relations
\[\partial_{\alpha}(W_Q)=0,\]
for each \emph{internal} arrow $\alpha$ in $Q$, where $\partial_{\alpha}$ is the usual cyclic derivative
(e.g.~\cite[\S1.3]{ginz} or \cite[\S3]{bocklandt}).
Thus the algebra $A_Q$
is a frozen Jacobian algebra
(e.g.~\cite[Defn.~5.1]{Pre}).
\end{remark}

\begin{definition} \label{d:centralelement}
Let $Q$ be a dimer model with boundary. Since the incidence graph of $Q$ at each vertex is connected, it follows from the defining relations of $A_Q$ that, for any vertex $i\in Q_0$,
the products in $A_Q$ of the arrows in any two cycles that start at $i$ and bound a face are the same. We denote such a product by $t_i$, and write
\begin{equation} \label{e:defu}
  t=\sum_{i\in Q_0(D)} t_i.
\end{equation}
It similarly follows from the relations that $t$ commutes with every arrow
and hence is in the centre of $A_Q$. Thus $A_Q$ is a $Z$-algebra for $Z=\powser{\CC}{t}$.
\end{definition}

A key property of dimer algebras that arise from Postnikov diagrams in the disc is the following.
It is the analogue of algebraic consistency \cite[\S5]{Bro} in this context.

\begin{definition}
\label{d:thin}
We say that a $Z$-algebra $A$ is \emph{thin} if $\Hom_A(P,Q)$ is a free rank one module over $Z$ for any indecomposable projective $A$-modules $P$ and $Q$.
\end{definition}

In practice, we will only consider $Z$-algebras $A$ defined via quivers, for which the indecomposable projectives are, up to isomorphism, $Ae_i$ for $i\in Q_0$. Such an algebra is thin if and only if $\Hom_A(Ae_j,Ae_i)=e_jAe_i$ is a free rank one module over $Z$ for each $i,j\in Q_0$, and in this case $A$ is free and finitely generated over $Z$.

It was shown in \cite[Cor.~9.4]{BKM} that the dimer algebra $A_D$ is thin when $D$ is a $(k,n)$-diagram. In fact, this is true for any connected Postnikov diagram $D$.

\begin{proposition}
\label{p:thin}
If $D$ is a connected Postnikov diagram in the disc, then $A_D$ is thin.
\end{proposition}
\begin{proof}

As in \cite[\S4]{BKM}, we may weight the arrows of $Q=Q(D)$ by elements of $\ZZ^{\C_0}$. A path in $Q$ is weighted by the sum of $w$ weights of its arrows, and its total weight is defined to be $\sum_{i\in \C_0}w(i)$, which is always at least $1$.

The proof of \cite[Cor.~4.4]{BKM}, stated for $(k,n)$-diagrams, remains valid in our more general setting to show that the path bounding any face of $Q$ has constant weight $w(i)=1$ for all $i\in \C_0$. If $p_+=p_-$ is an F-term relation, then there is an arrow $\alpha\in Q_1$ such that both $\alpha p_+$ and $\alpha p_-$ are such boundary cycles, from which it follows that the weights of $p_+$ and $p_-$ agree. Therefore the weight, and hence the total weight, is invariant under F-term equivalence, and thus descends to a grading of $A_D$.

Now let $i,j\in Q_0$. Since the disc is connected, there is some path from $i$ to $j$ in $Q$ \cite[Rem.~3.3]{BKM}, and we choose $p$ to be such a path with minimal total weight. If $q$ is any other path from $i$ to $j$, then \cite[Prop.~9.3]{BKM} applies to show that there is a path $r\colon i\to j$ and non-negative integers $N_p$ and $N_q$ such that
\[p=t^{N_p}r,\qquad q=t^{N_q}r\]
in $A_D$. As before, this proposition is stated only in the case that $D$ is a $(k,n)$-diagram, but its proof is still valid under our weaker assumptions---the key property of $D$ here is (b2). 

Since the total weight of $t$ is non-zero, and $p$ has minimal total weight among paths from $i$ to $j$, we must have $N_p=0$ and $p=r$. Thus $q=t^{N_q}p$ is a $Z$-multiple of $p$, showing that $e_jAe_i$ is a rank one $Z$-module. It is free since each element of $\{t^Np:N\geq0\}$ has a different total weight, which implies that these elements are linearly independent in $A_D$.
\end{proof}

\begin{remark}
\label{r:boundary-convention}
In some parts of the paper, particularly Section~\ref{sec:newMS} concerning the Marsh--Scott formula, it will be necessary to consider bipartite graphs such that all boundary nodes have the same colour. Any bipartite graph can be made into one with this property by introducing a bivalent node
on any half-edge incident with a boundary node of the wrong colour; up to isomorphism, adding this extra node does not affect $A_Q$, where $Q$ is the dual dimer model. In $Q$, this addition of a node corresponds to gluing a digon (i.e.\ a $2$-cycle bounding a face) onto the boundary arrow of a boundary face. If $Q=Q(D)$ for some Postnikov diagram $D$, then one can achieve the same effect by modifying $D$ via a twisting move \cite[Defn.~2.2]{BKM} at the boundary.

We will refer to bipartite graphs with only white boundary nodes as \emph{$\white$-standardised} and those with only black boundary nodes as \emph{$\black$-standardised}, and extend this terminology to the associated Postnikov diagrams and dimer models with boundary. Note for example that in a $\white$-standardised diagram $D$ of type $(k,n)$, the value $k$ is simply the number of white nodes minus the number of black nodes in $\plabic(D)$, whereas in a $\black$-standardised diagram the number of black nodes minus the number of white nodes is $n-k$.

\end{remark}

\begin{definition}\label{d:opposite}
Given a Postnikov diagram $D$, we denote by $D\op$ its opposite diagram, obtained by reversing the orientation of each strand.
\end{definition}
\begin{remark}\label{r:opposite}
The quiver, dimer algebra, bipartite graph, and type of $D\op$ are related to the corresponding objects associated to $D$ in the following way.
\begin{enumerate}
\item We have $Q(D\op)=Q(D)\op$, where the opposite $Q\op$ of a dimer model $Q$ with boundary is the opposite quiver with faces. This has the same set of vertices, arrows and faces as $Q$, but with $h\op(a)=t(a)$ and $t\op(a)=h(a)$ on arrows, with $\bdry\op F=(\bdry F)\op$ on faces, and with $(Q_2\op)^\pm=Q_2^\mp$.
\item It then follows directly from Definition~\ref{d:dimeralgebra} that $A_{D\op}=A_D\op$, that is, the identity map on vertices and arrows of $Q(D\op)=Q(D)\op$ induces an isomorphism of these algebras.
\item The bipartite graph $\plabic(D\op)$ is obtained from $\plabic(D)$ by swapping the colours of all nodes.
\item It then follows from Definition~\ref{d:type} that if $D$ has type $(k,n)$ then $D\op$ has type $(n-k,n)$, using that the total number of half-edges in either associated bipartite graph is $n$. 
\end{enumerate}
\end{remark}

\section{Boundary algebras}	\label{sec:bdry-alg}

In this section, we fix $1\leq k<n$, and explain how Postnikov diagrams of type $(k,n)$ are related to the categorification of the Grassmannian $\grass{k}{n}$ by Jensen--King--Su \cite{JKS}.

Consider again the $n$-vertex circular graph $\C=(\C_0,\C_1)$, as in Figure~\ref{f:graphC}.
We associate to $\C$ a quiver $Q=Q(\C)$ with vertex set $Q_0=\C_0$ and
arrow set $Q_1=\{x_i,y_i: i\in \C_1\}$ 
with $x_i$ clockwise and $y_i$ anticlockwise,
as illustrated in Figure~\ref{fig:mckayQ} in the case $n=7$.

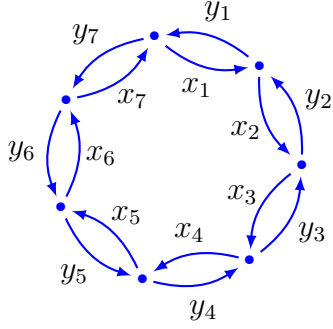
\begin{figure}[h]
\begin{tikzpicture} [scale=1.1,
qarrow/.style={-latex, blue, thick}]  
\pgfmathsetmacro{\cstart}{100}
\pgfmathsetmacro{\dstart}{\cstart+0.5*\seventh}
\pgfmathsetmacro{\estart}{\cstart+\seventh}
\pgfmathsetmacro{\qrad}{1.5}
\pgfmathsetmacro{\xrad}{0.9}
\pgfmathsetmacro{\yrad}{1.85}
\foreach \j in {1,...,7}
{ \path (\cstart-\seventh*\j:\qrad) node (w\j) {};
  \path (\estart-\seventh*\j:\qrad) node (v\j) {};
  \draw (w\j) node [blue] {\tiny $\bullet$};
  \path [qarrow] (v\j) edge [bend right=24] (w\j);
  \path [qarrow] (w\j) edge [bend right=24] (v\j);
  \draw (\dstart-\seventh*\j:\yrad) node[black] {$y_{\j}$};
  \draw (\dstart-\seventh*\j:\xrad) node[black] {$x_{\j}$}; 
}
\end{tikzpicture}
\caption{The double quiver $Q(\C)$.}
\label{fig:mckayQ}
\end{figure}

\begin{definition}
\label{d:preproj-Bkn}
Write $x=\sum_{i\in \C_1}x_i$ and $y=\sum_{i\in \C_1}y_i$. 
Then the \emph{(complete) preprojective algebra} $\Pi$ of $\C$ is the quotient of the complete path algebra of $Q(\C)$ by the closed ideal generated by $xy-yx$; multiplying this by the vertex idempotents produces one commutativity relation beginning at each vertex.

For our fixed $1\leq k < n$, we write $\jksalg$ for the quotient of $\Pi$ by the additional relation $y^k=x^{n-k}$. 
Again, this implies one relation of this kind beginning at each vertex. 
Writing $t=xy\in \jksalg$, the centre of $\jksalg$ is 
$Z=\powser{\CC}{t}$, and $\jksalg$ is a thin $Z$-algebra \cite[\S3]{JKS}.
\end{definition}

Since $\jksalg$ is free and finitely generated over $Z$, it is natural to consider the category
\[\CM(\jksalg)=\{X\in\fgmod{\jksalg}:\text{$X$ is free over $Z$}\}.\]
The notation here refers to (maximal) Cohen--Macaulay $\jksalg$-modules, meaning $C$-modules which are Cohen--Macaulay when restricted to the commutative (Gorenstein) ring $Z$; Auslander \cite[\S I.7]{Ausl} refers to these modules as $C$-lattices.
Note that $\CM(\jksalg)$ coincides with the category
\[\GP(\jksalg)=\{X\in\fgmod{\jksalg}:\text{$\Ext^i_{\jksalg}(X,\jksalg)=0$ for $i>0$}\}\]
of Gorenstein projective $\jksalg$-modules (see \cite[Cor.~3.7]{JKS} and \cite{KIWY,Pre}).

By Proposition~\ref{p:thin}, the dimer algebra $A$ of any Postnikov diagram $D$ is also free and finitely generated over $Z$ and, since $Z$ is a principal ideal domain, so is any subalgebra $B$ of $A$. Later, we will also consider the categories $\CM(A)$ and $\CM(B)$, but note that these do not usually coincide with $\GP(A)$ and $\GP(B)$.

The rank of $M\in\CM(\jksalg)$, when treated as a $Z$-module, is always divisible by $n=|Q_0|$, so we `normalise' by dividing out this constant. This normalised rank may also be computed as the length of $M\otimes_ZK$ over the simple algebra $\jksalg\otimes_ZK\isom M_n(K)$, where $K=\fracpowser{\CC}{t}$ is the field of fractions of $Z$ \cite[Defn.~3.5]{JKS}.

\begin{definition}[{\cite[Defn.~5.1]{JKS}}]
\label{d:rank1mod}
For any $I\subset \C_1$, we can define a $\Pi$-module $M_I$ as follows. 
For each $i\in \C_0$, set $e_iM_I=Z$. The arrows of $Q(\C)$ act by
\begin{equation*}
 x_i \cdot z = \begin{cases} tz & i\in I, \\ z & i\notin I, \end{cases}
 \qquad
 y_i \cdot z = \begin{cases} z & i\in I, \\ tz & i\notin I. \end{cases}
\end{equation*}
Then $xy$ and $yx$ both act as multiplication by $t$, and so $M_I$ is a $\Pi$-module. 

If $I$ is a $k$-subset, then $M_I$ is actually a $\jksalg$-module:
if the product of $n-k$ successive arrows $x_i$ acts by $t^s$, then the product of the remaining $k$ arrows $x_j$ acts by $t^{k-s}$. Hence the product of the corresponding $k$ arrows $y_j$ acts again by $t^{k-(k-s)}=t^s$, and so we conclude that $y^k$ and $x^{n-k}$ always have the same action.
By construction, $M_I$ is free and finitely generated as a $Z$-module, so it is in $\CM(\jksalg)$, and furthermore it has rank $1$.
\end{definition}

\begin{remark}
\label{r:JKS-comparison}
Note that we use complementary naming conventions to those in \cite{JKS}: our module $M_I$ would be denoted there by $M_{I\comp}$, that is, using the complementary subset of $\C_1$. It is explained in \cite{JKS} how the category
$\CM(\jksalg)$, for $C=\Pi/(y^k-x^{n-k})$, provides a categorification of Scott's cluster algebra structure \cite{Sco} on the Grassmannian $\Gr_{n-k}^n$ of $(n-k)$-planes in $\CC^n$; in particular, there is a cluster character $\CM(\jksalg)\to\CC[\Gr_{n-k}^n]$. Because of the difference in conventions, it takes our $C$-module $M_I$ to the Pl\"ucker coordinate $\Pluck{I\comp}$. However, by composing with the isomorphism $\CC[\Gr_{n-k}^n]\to\CC[\Gr_k^n]$ satisfying $\Pluck{I\comp}\mapsto\Pluck{I}$ for each $k$-subset $I\subset\C_1$, which relates Scott's cluster structures on these two isomorphic Grassmannians, we obtain a cluster character 
$\jkscluschar\colon\CM(\jksalg)\to\CC[\Gr_k^n]$ sending $M_I$ 
to the Pl\"ucker coordinate $\Pluck{I}$.
\end{remark}

Every rank $1$ module in $\CM(\jksalg)$ is isomorphic to $M_I$ for some $k$-subset $I\subset\C_1$ \cite[Prop.~5.2]{JKS}. Certain cluster-tilting objects in $\CM(\jksalg)$, all of which are mutation equivalent, have the property that all of their indecomposable summands have rank $1$, and the cluster character $\jkscluschar$ induces a bijection from the mutation class of these objects to the set of clusters of the Grassmannian cluster algebra \cite[Rem.~9.6]{JKS}.

Now let $D$ be any Postnikov diagram of type $(k,n)$, with dimer algebra $A_D$. 
We may define the boundary idempotent $e=\sum_{i\in \C_0}e_i\in A_D$, 
and consider the \emph{boundary algebra} $B=eA_De$. 
This algebra is quite closely related to the algebra $C$, which depends only on the type $(k,n)$, as we now explain.

\newcommand{\canPB}{\tilde\varepsilon}
\newcommand{\canCB}{\varepsilon}

Let $i\in\C_1$. If the boundary arrow of $Q(D)$ labelled by $i$ is clockwise, we name this arrow $\alpha_i$, and let $\beta_i$ be the (unique) path completing $\alpha_i$ to a boundary face. Conversely, if the boundary arrow labelled by $i$ is anticlockwise, then we call this arrow $\beta_i$, and write $\alpha_i$ for the path completing it to a face. Writing $\alpha=\sum_{i\in\C_1}\alpha_i$ and $\beta=\sum_{i\in\C_1}\beta_i$, we have $\alpha\beta= te = \beta\alpha$, and hence there is a canonical map 
\begin{equation*}
  \canPB\colon\Pi\to B, 
\end{equation*}
fixing the vertex idempotents $e_j$, for $j\in \C_0$, and with $\canPB(x_i)=\alpha_i$ and $\canPB(y_i)=\beta_i$ for each $i\in\C_1$. 
The existence of the map $\canPB$ can also be deduced from the description of $B$ as the boundary algebra of the frozen Jacobian algebra $A_D$, by \cite[Prop.~8.1]{Pre2}.

\begin{claim} \label{cm:C-to-B}
When $D$ is connected and has type $(k,n)$, the map $\canPB\colon\Pi\to B$ factors through a map $\canCB\colon\jksalg \to B$. In other words, $\canPB(y^k-x^{n-k})=0$.
\end{claim}

\begin{remark}
It would be nice to have a direct algebraic proof of Claim~\ref{cm:C-to-B}, 
but we currently use facts about perfect matching modules proved in Section~\ref{sec:PM-mods}, 
so the proof is postponed until after Proposition~\ref{p:pm-mods}.
The statement depends on consistency of the dimer model, 
as the example in Figure~\ref{f:inconsistent} shows.
Here the combinatorics tells us that $k=1$ and $n=3$, but the relation $x^2=y$ does not follow from the dimer relations.

\begin{figure} [h]
\begin{tikzpicture}[xscale=1.8,yscale=-1.8]
\pgfmathsetmacro{\brad}{1}
\pgfmathsetmacro{\bstart}{120}
\pgfmathsetmacro{\mrad}{0.75}
\pgfmathsetmacro{\srad}{0.5*\mrad}
\pgfmathsetmacro{\labsc}{1.2}
\begin{scope}
\draw [boundary] (0,0) circle(\brad);
\foreach \n\/\m in {1/2,2/1,3/3}
{ \coordinate (b\n) at (\bstart-120*\n:\brad);
  \draw ($\labsc*(b\n)$) node {\small $\m$}; }
\coordinate (b0) at (0,0);
\foreach \n in {1,2,3}
{ \coordinate (b1\n) at (\bstart-120*\n:\mrad);
  \coordinate (b2\n) at (\bstart+60-120*\n:\srad);
  \coordinate (c2\n) at (\bstart+60-120*\n:1.3*\srad); }
\foreach \h/\t in {1/11, 2/12, 3/13, 0/21, 0/22, 0/23, 11/21, 12/22, 13/23, 11/22, 12/23, 13/21}
{ \draw [bipedge] (b\h)--(b\t); }

\foreach \n in {11,12,13,0} 
  {\draw [\graphcolor] (b\n) circle(\dotrad) [fill=white];} \foreach \n in {21,22,23}  
  {\draw [\graphcolor] (b\n) circle(\dotrad) [fill=\graphcolor];} 

\foreach \e/\f in {0/21, 0/22, 0/23, 11/21, 12/22, 13/23, 11/22, 12/23, 13/21}
{\coordinate (a\e-\f) at ($(b\e) ! 0.5 ! (b\f)$); }

\draw [strand] plot[smooth]
coordinates {(b1) (a11-21) (a0-21) (a0-23) (a12-23) (b2)}
[postaction=decorate, decoration={markings,
 mark= at position 0.15 with \strarrow,  mark= at position 0.38 with \strarrow,
 mark= at position 0.52 with \strarrow,   mark= at position 0.65 with \strarrow,
 mark= at position 0.88 with \strarrow }];
 \draw [strand] plot[smooth]
coordinates {(b2) (a12-22) (a0-22) (a0-21) (a13-21) (b3)}
[postaction=decorate, decoration={markings,
 mark= at position 0.15 with \strarrow,  mark= at position 0.38 with \strarrow,
 mark= at position 0.52 with \strarrow,   mark= at position 0.65 with \strarrow,
 mark= at position 0.88 with \strarrow }];
 \draw [strand] plot[smooth]
coordinates {(b3) (a13-23) (a0-23) (a0-22) (a11-22) (b1)}
[postaction=decorate, decoration={markings,
 mark= at position 0.15 with \strarrow,  mark= at position 0.38 with \strarrow,
 mark= at position 0.52 with \strarrow,   mark= at position 0.65 with \strarrow,
 mark= at position 0.88 with \strarrow }];
 \draw [strand] (a11-22) to [out=-120, in=0] (a12-22)
[postaction=decorate, decoration={markings,
 mark= at position 0.52 with \strarrow}];
 \draw [strand] (a12-22) to [out=180, in=-60] (a12-23)
[postaction=decorate, decoration={markings,
 mark= at position 0.55 with \strarrow}];
\draw [strand] (a12-23) to [out=120, in=-120] (a13-23)
[postaction=decorate, decoration={markings,
 mark= at position 0.52 with \strarrow}];
 \draw [strand] (a13-23) to [out=60, in=180] (a13-21)
[postaction=decorate, decoration={markings,
 mark= at position 0.55 with \strarrow}];
 \draw [strand] (a13-21) to [out=0, in=120] (a11-21)
[postaction=decorate, decoration={markings,
 mark= at position 0.52 with \strarrow}];
 \draw [strand] (a11-21) to [out=-60, in=60] (a11-22)
[postaction=decorate, decoration={markings,
 mark= at position 0.55 with \strarrow}];
\end{scope}
\pgfmathsetmacro{\qrad}{0.9}
\pgfmathsetmacro{\qstart}{60}
\pgfmathsetmacro{\alp}{10}
\pgfmathsetmacro{\sqrad}{0.5*\qrad}
\begin{scope} [shift={(3,0)}]
\foreach \n/\m in {1/2,2/1,3/3}
{  \draw (\bstart-120*\n:1.1*\brad) node {\small $\m$}; 
  \coordinate (q\n) at (\qstart-120*\n:\qrad);
  \draw [\quivcolor] (q\n) node {\tiny $\bullet$}; 
 \coordinate (q1\n) at (\bstart-120*\n:\sqrad);   
  \draw [\quivcolor] (q1\n) node {\tiny $\bullet$}; }
\foreach \t/\h in {11/1, 1/12, 12/11, 12/2, 2/13, 13/12, 13/3, 3/11, 11/13}
  \draw [quivarrow] ($(q\t) ! 0.15 ! (q\h)$) to ($(q\t) ! 0.85 ! (q\h)$);
\foreach \a/\h in {-60, 60, 180}
  \draw [quivarrow] (\a+\alp:\qrad) arc (\a+\alp:\a+120-\alp:\qrad);
\end{scope}
\end{tikzpicture}
\caption{An inconsistent dimer model.}
\label{f:inconsistent}
\end{figure}
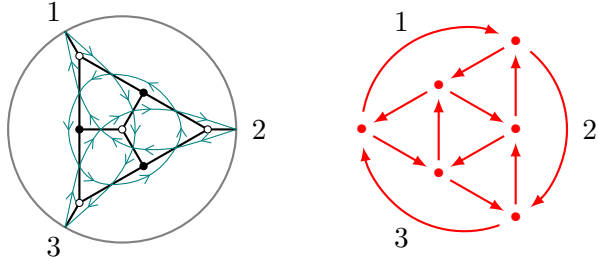
\end{remark}

Assuming Claim~\ref{cm:C-to-B} for the moment, we have the following.

\begin{proposition}
\label{p:restriction}
Let $B=eA_De$ be the boundary algebra of $A_D$,
for $D$ a connected Postnikov diagram of type $(k,n)$.
Then the canonical map $\canCB\colon \jksalg\to B$ is injective and the corresponding restriction functor
$\rho\colon\CM(B)\to\CM(\jksalg)$ is fully faithful.
\end{proposition}

\begin{proof}
It follows from Proposition~\ref{p:thin} and \cite[\S3]{JKS}
that both $B$ and $\jksalg$ are thin.
Hence when restricted to each piece $e_i C e_j$, for $i,j\in \C_0$,
the canonical map $\canCB\colon \jksalg\to B$ from Claim~\ref{cm:C-to-B}
becomes a map of free $Z$-modules of rank 1,
so it is either injective or zero.
The image of a generator of $e_i C e_j$ is a path in the dimer algebra $A_D$,
i.e.\ the F-term equivalence class of a path in the defining quiver.
It follows from the proof of Proposition~\ref{p:thin} that no path is zero in $A_D$, and
so $\canCB$ must be injective, as required.

Let $Z[t^{-1}]=\fracpowser{\CC}{t}$ be the field of formal Laurent series in $t$ and, for any $Z$-module $X$,
let $X[t^{-1}]=X\otimes_Z Z[t^{-1}]$.
In particular, if $M$ is a $B$-module, then $M[t^{-1}]$ is a $B[t^{-1}]$-module.
Because $B$ and $\jksalg$ are thin, the inclusion $\canCB\colon \jksalg\to B$ 
induces an isomorphism $\jksalg[t^{-1}]\isom B[t^{-1}]$ 
and so we may consider that $B\subset\jksalg[t^{-1}]$.

Thus any modules $M,N$ in $\CM(B)$ can be considered to be
$B$-submodules of the $\jksalg[t^{-1}]$-modules $M[t^{-1}], N[t^{-1}]$.
Now $t$ acts injectively on $M$ and $N$, so any map in $\Hom_\jksalg(\rho M,\rho N)$
commutes with $t^{-1}$ and so commutes with any element of $B$.
Thus $\Hom_\jksalg(\rho M,\rho N)=\Hom_B(M,N)$,
that is, $\rho$ is fully faithful, as required.
\end{proof}

Note that the canonical map $\canCB\colon \jksalg\to B$ is typically not surjective, and the more general restriction map $\md{B}\to\md{\jksalg}$ is typically not fully faithful.
An example is shown in Figure~\ref{f:not-surj}.
There every vertex of $Q(D)$ is on the boundary, so $B=A_D$,
but the map $\canCB\colon\jksalg\to B$ from Claim~\ref{cm:C-to-B} is not surjective,
since the internal arrow of $A_D$ is not in its image.

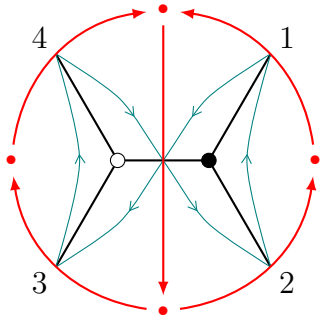
\begin{figure} [h]
\begin{tikzpicture}[scale=2,baseline=(bb.base)]
\renewcommand{\dotrad}{1.35pt}
\path (0,0) node (bb) {};
\foreach \n/\a in {1/0, 2/0, 3/0, 4/0}
{ \coordinate (b\n) at (135-90*\n+\a:1.0);
  \draw (135-90*\n+\a:1.15) node {$\n$}; }
\coordinate (b5) at (0:0.3);
\coordinate (b6) at (180:0.3);
\foreach \h/\t in {1/5,2/5,3/6,4/6,5/6}
{ \draw [bipedge] (b\h)--(b\t); }
\draw [\graphcolor] (b5) circle(\dotrad) [fill=\graphcolor]; \draw [\graphcolor] (b6) circle(\dotrad) [fill=white]; 

\foreach \e/\f in {5/6}
{\coordinate (a\e-\f) at ($(b\e) ! 0.5 ! (b\f)$); }

\draw [strand] plot[smooth]
coordinates {(b1) (55:0.5) (a5-6) (235:0.5) (b3)}
[postaction=decorate, decoration={markings,
 mark= at position 0.33 with \strarrow,  mark= at position 0.7 with \strarrow }];
 \draw [strand] plot[smooth]
coordinates {(b2) (0:0.55) (b1)}
[postaction=decorate, decoration={markings,
 mark= at position 0.53 with \strarrow}];
  \draw [strand] plot[smooth]
coordinates {(b3) (180:0.55) (b4)}
[postaction=decorate, decoration={markings,
 mark= at position 0.53 with \strarrow}];
 \draw [strand] plot[smooth]
coordinates {(b4) (125:0.5) (a5-6) (305:0.5) (b2)}
[postaction=decorate, decoration={markings,
 mark= at position 0.33 with \strarrow,  mark= at position 0.7 with \strarrow }];

\foreach \n/\m in {1/12,2/23,3/24,4/14}
{ \draw [\quivcolor] (180+90*\n:1) node (q\m) {\tiny $\bullet$}; }
\foreach \t/\h/\a in {12/23/-37, 23/24/-37, 14/24/37, 12/14/37, 24/12/0}
 { \draw [quivarrow] (q\t) edge [bend left=\a] (q\h); }
 \end{tikzpicture}
 \caption{A dimer algebra $A_D$ for a Postnikov diagram $D$ of type $(2,4)$.}
 \label{f:not-surj}
\end{figure}

\section{Perfect matching modules}  \label{sec:PM-mods}

Let $D$ be a Postnikov diagram of type $(k,n)$. In this section we associate a module for the dimer algebra $A_D$ to each perfect matching of the bipartite graph $\plabic(D)$.
To start with, we may consider an arbitrary quiver with faces $Q$ (see Definition~\ref{d:quiver-faces}).

\begin{definition}
\label{d:pm}
A \emph{perfect matching} on a quiver with faces $Q$ is a subset $\mu$ of $Q_1$ such that the boundary of each face in $Q_2$ contains precisely one arrow in $\mu$.
\end{definition}

An example of a perfect matching is given in Figure~\ref{f:pm}.

\begin{figure}[h]
\begin{tikzpicture}[scale=3,baseline=(bb.base),yscale=-1]

\path (0,0) node (bb) {};

\foreach \n/\m/\a in {1/4/0, 2/3/0, 3/2/5, 4/1/10, 5/7/0, 6/6/-3, 7/5/0}
{ \coordinate (b\n) at (\bstart-\seventh*\n+\a:1.0);
  \draw (\bstart-\seventh*\n+\a:1.1) node {$\m$}; }

\foreach \n/\m in {8/1, 9/2, 10/3, 11/4, 14/5, 15/6, 16/7}
  {\coordinate (b\n) at ($0.65*(b\m)$);}

\coordinate (b13) at ($(b15) - (b16) + (b8)$);
\coordinate (b12) at ($(b14) - (b15) + (b13)$);

\foreach \n/\x\y in {13/-0.03/-0.03, 12/-0.22/0.0, 14/-0.07/-0.03, 11/0.05/0.02, 16/-0.02/0.02}
  {\coordinate (b\n)  at ($(b\n) + (\x,\y)$); } 

\foreach \h/\t in {1/8, 2/9, 3/10, 4/11, 5/14, 6/15, 7/16, 
 8/9, 9/10, 10/11,11/12, 12/13, 13/8, 14/15, 15/16, 12/14, 13/15, 8/16}
{ \draw [bipedge] (b\h)--(b\t); }

\foreach \n in {8,10,12,15} 
  {\draw [\graphcolor] (b\n) circle(\dotrad) [fill=white];} \foreach \n in {9,11,13, 14,16}  
  {\draw [\graphcolor] (b\n) circle(\dotrad) [fill=\graphcolor];} 

\foreach \e/\f/\t in {2/9/0.5, 4/11/0.5, 5/14/0.5, 7/16/0.5, 
 8/9/0.5, 9/10/0.5, 10/11/0.5,11/12/0.5, 12/13/0.45, 8/13/0.6, 
 14/15/0.5, 15/16/0.6, 12/14/0.45, 13/15/0.4, 8/16/0.6}
{\coordinate (a\e-\f) at ($(b\e) ! \t ! (b\f)$); }

\foreach \n/\m\a in {1/124/0, 2/234/-1, 3/345/1, 4/456/10, 5/256/5, 6/267/0, 7/127/0}
{ \draw [\quivcolor] (\bstart+\seventh/2-\seventh*\n+\a:1) node (q\m) {\tiny $\bullet$}; }

\foreach \m/\a/\r in {247/\bstart/0.42, 245/10/0.25, 257/210/0.36}
{ \draw [\quivcolor] (\a:\r) node (q\m) {\tiny $\bullet$}; }

\foreach \t/\h/\a in {234/124/-20, 234/345/19, 456/345/-16, 456/256/22, 127/267/-20,
 127/124/19, 247/127/11, 245/234/18, 247/245/34, 245/257/15,
 267/257/6, 257/256/0, 245/456/-16}
{ \draw [quivarrow]  (q\t) edge [bend left=\a] (q\h); }

\foreach \t/\h/\a in {257/247/14, 124/247/13, 345/245/0, 256/267/21, 256/245/-9}
{ \draw [quivarrow, ultra thick, violet]  (q\t) edge [bend left=\a] (q\h); }
 \end{tikzpicture}
 \caption{A perfect matching of a dimer model with boundary. The perfect matching is indicated by the thicker arrows.}
 \label{f:pm}
\end{figure}
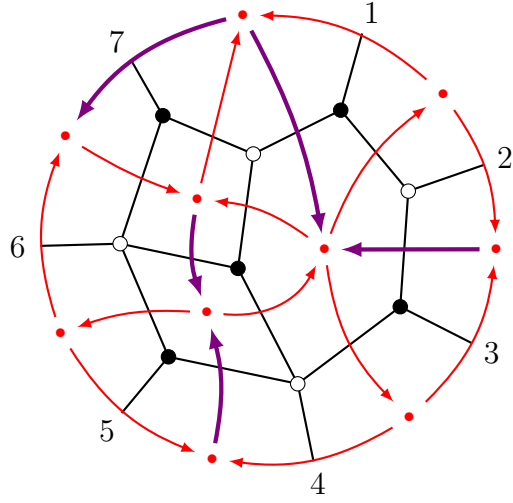

\begin{remark}
\label{r:pm-graph}
If $Q$ is a dimer model with boundary, its arrows are in bijection with the edges and half-edges of the dual bipartite graph $\plabic$, in such a way that the arrows incident with a given face correspond to the edges and half-edges incident with the dual node. A perfect matching of $Q$ is thus equivalent to the data of a subset $\mu$ of the edges and half-edges of $\plabic$ with the property that each node of $\plabic$ is incident with precisely one element of $\mu$. When $|Q|$ is closed, so that $\plabic$ is an honest bipartite graph, such a set $\mu$ is a perfect matching of $\plabic$ in the usual, graph-theoretic sense, hence the terminology. In general, a boundary node of $\plabic$ need not be matched with another node, but may instead be incident with a half-edge in $\mu$.
\end{remark}

Any perfect matching of a quiver $Q$ with faces determines a $\compl{\CC Q}$-module in the following way.

\begin{definition}
\label{d:n-mu}
To each perfect matching $\mu$ on $Q$, we associate a $\compl{\CC Q}$-module $N_\mu$ as follows. Let $e_iN_\mu=Z$ for all $i\in Q_0$. An arrow $\alpha$ acts as multiplication by $t$ if $\alpha\in\mu$, and as the identity otherwise.
\end{definition}

We may extend the quiver $Q(\C)$ from Section~\ref{sec:bdry-alg} to a quiver with $n$ faces, the boundaries of which are the $2$-cycles $x_iy_i$ for $i\in C_1$. Then, given any subset $I\subset C_1$, the set $\mu(I)=\{x_i:i\in I\}\cup\{y_j:j\notin I\}\subset Q_1$ is a perfect matching of $Q(\C)$, and the module $N_{\mu(I)}$ is precisely the $\Pi$-module $M_I$ from Definition~\ref{d:rank1mod}.

If $Q$ is a dimer model with boundary and $\mu$ is a perfect matching, $p^+_\alpha$ and $p^-_\alpha$ act on $N_\mu$ in the same way for any $\alpha\in Q_1$, and so $N_\mu$ is a module for the dimer algebra $A_Q$. Note that the central element $t\in A_Q$ from Definition~\ref{d:centralelement} acts on any $N_\mu$ as multiplication by $t\in Z$, justifying the abuse of notation.

\begin{definition}
\label{d:grading}
Any perfect matching $\mu$ of a quiver with faces $Q$ determines a grading of the path algebra $\compl{\CC Q}$ with
\[\deg_\mu{\alpha}=
\begin{cases}
1,&\alpha\in\mu,\\
0,&\alpha\notin\mu.
\end{cases}\]
If $Q$ is a dimer model with boundary, this descends to a grading of $A_Q$ since the defining relation 
$p_\alpha^+-p_\alpha^-$
has degree $0$ if $\alpha\in\mu$ and degree $1$ otherwise. For any $\mu$, we have $\deg_\mu t=1$. Grading $N_\mu$ by putting $\deg{1}=0$ for each generator $1\in e_jN_\mu= Z$ makes $N_\mu$ into a graded $\compl{\CC Q}$-module for the above grading on $\compl{\CC Q}$.
\end{definition}

\begin{proposition}
\label{p:pm-mods}
Let $Q$ be a quiver with faces such that $|Q|$ is simply connected. 
Let $N$ be a $\compl{\CC Q}$-module such that the vector space $e_jN$ is equipped with the structure of a free $Z$-module of rank $1$ for each $j\in Q_0$, in such a way that $\bdry F$ acts as multiplication by $t$ for every $F\in Q_2$.
Then there exists a unique perfect matching $\mu$ of $Q$ such that $N\isom N_\mu$.
\end{proposition}

\begin{proof}
We associate to $N$, as a representation of $Q$, the set $\mu$ of arrows whose arrow maps are non-invertible. Choosing a $Z$-module generator for $e_jN$ for $j\in Q_0$, each arrow $\alpha$ acts, relative to these generators, as multiplication by $\lambda_\alpha t^{m_{\alpha}}$ with $\lambda_\alpha\in Z^\times$; we use here that $Z$ is a local ring with maximal ideal $(t)$. Since the boundary of any face acts by $t$, we must have $m_\alpha\in\{0,1\}$, equal to $1$ for exactly one arrow in each face. These are precisely the arrows in $\mu$, which is thus a perfect matching.

Moreover, the $\lambda_\alpha$ multiply to $1$ around each face, so $\lambda=(\lambda_\alpha)_{\alpha\in Q_1}$ is a $1$-cocycle for $Q$ with coefficients in $Z^\times$. As $|Q|$ is simply connected, $\lambda=\d \kappa$ for some $0$-cochain $\kappa$. Rescaling the generators by $\kappa$ sets $\lambda_\alpha=1$ for all $\alpha$, and thus $N\isom N_\mu$.

Uniqueness follows because $\mu$ is the set of arrows acting non-invertibly on $N_\mu$, and this set is an isomorphism invariant.
\end{proof}

When $N$ is an $A_D$-module for some connected Postnikov diagram $D$, we will always give $N$ (and hence the fibres $e_jN$ for $j\in Q_0$) the $Z$-module structure arising from the restriction to $Z\subset A_D$ (see Definition~\ref{d:centralelement}). In particular, this means that $\bdry F$ always acts on $N$ as multiplication by $t$ for every $F$ in $Q_2$, and so Proposition~\ref{p:pm-mods} simplifies as follows.

\begin{corollary}
\label{c:proj-pm-mods}
Let $D$ be a connected Postnikov diagram with dimer algebra $A=A_D$, and let $N$ be an $A$-module such that the $Z$-module $e_jN$ is free of rank $1$ for each $j\in Q_0$. Then $N\cong N_\mu$ for a unique perfect matching $\mu$ of $Q(D)$. This applies in particular when $N=Ae_i$ is an indecomposable projective $A$-module.
\end{corollary}
\begin{proof}
The first statement is just Proposition~\ref{p:pm-mods}, the condition on $\bdry F$ being automatic as above. For an indecomposable projective $Ae_i$, the fibre $e_jAe_i=\Hom_A(Ae_j,Ae_i)$ is free of rank one since $A$ is thin (Proposition~\ref{p:thin}).
\end{proof}

\begin{definition}
\label{d:bdry-value}
Let $D$ be a Postnikov diagram with quiver $Q(D)$. Then a perfect matching $\mu$ on $Q(D)$ has a \emph{boundary value} $\bdry\mu\subset\C_1$, defined as follows: $\bdry\mu$ consists of those $i\in\C_1$ such that either the boundary arrow of $Q(D)$ labelled by $i$ is clockwise and contained in $\mu$, or this arrow is anticlockwise and not contained in $\mu$.
\end{definition}

The boundary value of the perfect matching in Figure~\ref{f:pm} is $\{1,3,5\}$. Note that if $D$ is $\white$-standardised in the sense of Remark~\ref{r:boundary-convention}, then all boundary arrows are clockwise and so $\bdry\mu$ consists simply of the labels in $\C_1$ of the boundary arrows in $\mu$. Conversely, in a $\black$-standardised diagram all of the boundary arrows are anticlockwise, and $\bdry\mu$ consists of the labels of those boundary arrows not in $\mu$.

\begin{proposition}
\label{p:bdry-value-card}
If $D$ is a Postnikov diagram of type $(k,n)$ and $\mu$ is a perfect matching of $Q(D)$, then $\bdry\mu$ has cardinality $k$.
\end{proposition}
\begin{proof}
Since $k$ is defined in terms of the graph $\plabic(D)$, we view $\mu$ as a subset of the edges and half-edges of $\plabic(D)$ as in Remark~\ref{r:pm-graph}. In this language, the boundary value $\bdry\mu$ consists of those $I\in\C_1$ such that the corresponding half-edge of $\plabic(D)$ is either incident with a white node and contained in $\mu$, or is incident with a black node and not contained in $\mu$.

Now consider the disjoint union $S$ of the set of white nodes of $\plabic(D)$ with the set of half-edges of $\plabic(D)$ incident with a black node, and its subset $S_\mu$ consisting of white nodes joined to a black node by an edge of $\mu$, together with the half-edges in $\mu\cap S$. Since $\mu$ is a perfect matching, the cardinality of $S_\mu$ is equal to the number of black nodes, and so $S\setminus S_\mu$ has cardinality $k$ by a direct comparison with Definition~\ref{d:type}. On the other hand, $S\setminus S_\mu$ consists of those white nodes incident with a (necessarily unique) half-edge in $\mu$, together with the half-edges of $S$ which are not in $\mu$, and so its cardinality also agrees with that of $\bdry\mu$.
\end{proof}

The modules $N_\mu$ provide a convenient way to prove Claim~\ref{cm:C-to-B} and hence to complete the proof of Proposition~\ref{p:restriction}.

\begin{proof}[Proof of Claim~\ref{cm:C-to-B}]
We need only check that $\canPB(y^k-x^{n-k})=0$, or equivalently, that $\canPB(y^k)$ and $\canPB(x^{n-k})$ have the same action on any indecomposable projective $B$-module $Be_i$, for $i\in\C_0$. 
Now $Be_i=eA_De_i$ is a subspace of the projective $A_D$-module $A_De_i$ which, 
by Corollary~\ref{c:proj-pm-mods}, is isomorphic to $N_\mu$ for some perfect matching $\mu$. 
In fact, the elements $\canPB(y^k)$ and $\canPB(x^{n-k})$ of $B\subset A_D$ act in the same way on $N_\mu$ for any perfect matching $\mu$, as we now show.

Fix $i\in\C_1$. By construction, $\canPB(x_i)=\alpha_i$ acts on the relevant fibres of $N_\mu$ either by the identity or as multiplication by $t$, and $\canPB(y_i)=\beta_i$ acts complementarily. 
If the boundary arrow of $Q(D)$ labelled by $i$ is clockwise, then $\alpha_i$ is this arrow, which acts as $t$ on $N_\mu$ if and only if $i\in\bdry\mu$. 
On the other hand, if the boundary arrow $\beta_i$ labelled by $i$ is anticlockwise, then $\alpha_i$ is the path completing $\beta_i$ to a face, which acts as $t$ on $N_\mu$ if and only if $\beta_i$ acts as $1$, again if and only if $i\in\bdry\mu$. 
Thus by Proposition~\ref{p:bdry-value-card}, exactly $k=|\bdry\mu|$ of the $\alpha_i$ act as multiplication by $t$. 
Verifying that $\canPB(y^k)$ and $\canPB(x^{n-k})$ have the same action on $N_\mu$ is then straightforward (cf.~Definition~\ref{d:rank1mod}).
\end{proof}

\begin{proposition}
\label{p:res-pm-mods}
Let $D$ be a Postnikov diagram, let $\mu$ be a perfect matching for the associated quiver $Q(D)$ with corresponding $A_D$-module $N_\mu$, and let $e$ be the boundary idempotent of $A_D$. 
Then $\restr{eN_\mu}=M_{\pmbdry{\mu}}$, where $\rho\colon\CM(B)\to\CM(\jksalg)$ denotes the restriction functor from Proposition~\ref{p:restriction}.
\end{proposition}

\begin{proof}
Consider $eN_\mu$, which by definition has vertex components $e_iN_\mu=Z$ for each $i\in \C_0$, under our identification of $\C_0$ with the boundary vertices of $Q$. Exactly as in the proof of Claim~\ref{cm:C-to-B}, the arrow $x_i$ acts as multiplication by $t$ if $i\in \pmbdry{\mu}$, and as the identity otherwise. Since $x_iy_i$ bounds a face, it must act by multiplication by $t$. Hence the arrow $y_i$ acts as the identity when $x_i$ acts by $t$, i.e.\ when $i\in \pmbdry{\mu}$, and as multiplication by $t$ otherwise. Comparing to Definition~\ref{d:rank1mod}, we see that $\restr{eN_\mu}=M_{\pmbdry{\mu}}$.
\end{proof}

Since $\rho$ is fully faithful by Proposition~\ref{p:restriction}, we immediately have the following.
 
\begin{corollary}
\label{c:bdy-val-mod}
The boundary module $M=eN_\mu\in\CM(B)$ of a perfect matching $\mu$ is determined up to isomorphism by the boundary value $\pmbdry{\mu}$ of the matching.
\end{corollary}

\begin{definition}
\label{d:pmm}
We will refer to an $A$-module $N$ together with a preferred isomorphism $N_\mu\isoto N$ as \emph{a perfect matching module}. Specifying such an isomorphism is equivalent to choosing a preferred generator $g_j$ for each $e_jN$ (necessarily a rank one $Z$-module) in such a way that the arrows act by multiplication by a power of $t$, relative to these generators. The isomorphism is then given by mapping $1\in e_jN_\mu=Z$ to $g_j$, and the power of $t$ is necessarily $0$ or $1$, as in the proof of Proposition~\ref{p:pm-mods}.
\end{definition}

\begin{lemma}
\label{l:pm-submods}
Any submodule $N$ of a perfect matching module $M$ is canonically a perfect matching module. \end{lemma}
\begin{proof}
We have a generator $g_j$ for each $e_jM$ as in Definition~\ref{d:pmm}. Each $e_jN$ is a $Z$-submodule of $e_jM$ and is thus canonically generated by $t^mg_j$ for some $m$ (depending on $j$). Since $t$ is central in $A$, the arrows still act on these new generators by multiplication by a power of $t$, as required.
\end{proof}

\section{Induction and restriction} \label{sec:ind-res}

Let $D$ be a connected Postnikov diagram in the disc with quiver $Q=Q(D)$, and write $A=A_D$ for its dimer algebra, with boundary idempotent $e$. Write $B=eAe$ and $\cto =eA$. The restriction functor 
\[
e\colon  \md A \to\md B\colon L\mapsto eL = T\otimes_A L = \Hom_A(Ae,L)
\] 
has right and left adjoints $F,\Ftil \colon \md B \to\md A$ given by
\begin{align*}
  FM &=\Hom_B(\cto,M),\\
 \Ftil M &= Ae\otimes _B M.
\end{align*}
Since $eF$ and $e\Ftil$ are naturally isomorphic to the identity on $\md B$, there is a universal map 
\begin{equation}\label{eq:iota}
 \iota_M\colon \Ftil M\to FM.
\end{equation}  
We write $\smallF M=\im{\iota_M}$; this defines a functor $\smallF\colon\md B\to \md A$, sometimes called the intermediate extension associated to the idempotent $e$. See \cite{CBS,Kuhn} for some general discussion of this construction. We may also compute $F'M$ as the torsion-free part of $\Ftil M$ (as a $Z$-module), so that $\smallF$ becomes the honest left adjoint of $e$ upon its restriction to a functor $\CM(A)\to\CM(B)$.

Let $N\in\CM(A)$. Then, viewing $N$ as a quiver representation, the fibre of $e_iN$ over each $i\in Q_0$ is a free and finitely generated $Z$-module. Moreover, each $a\in Q_1$ begins a cycle bounding a face. Since the cycle acts on $N$ as multiplication by $t$, and so in particular injectively, $a$ must also act injectively on $M$, and so $\rk_Z(e_{ha}N)\leq\rk_Z(e_{ta}N)$. Since $Q$ is strongly connected, meaning any two vertices lie on some cycle, it follows that in fact $\rk_Z(e_iN)$ is constant in $i$. We define $\rk(N)$ to be this constant value. Observe that $\rk(N_\mu)=1$ for any perfect matching $\mu$ by construction, and that if $\rk(N)=1$ then $N\isom N_\mu$ for some perfect matching $\mu$ by Corollary~\ref{c:proj-pm-mods}. 

\begin{lemma}
\label{l:cm-lift}
When $M$ is in $\CM(B)$, both $F M$ and $\smallF M$ are in $\CM(A)$.
\end{lemma}

\begin{proof}
Since $Z$ is a principal ideal domain, any submodule of a free and finitely generated $Z$-module is again free and finitely generated. By Proposition~\ref{p:thin}, $T\in\CM(A)$.

It follows that $\smallF M\subset FM\subset\Hom_Z(T,M)$ are free and finitely generated whenever $M\in\CM(B)$, since $T$ and $M$ are. Thus $\smallF M, FM\in\CM(A)$ for all $M\in\CM(B)$.
\end{proof}

A consequence of Lemma~\ref{l:cm-lift} is that for any $M\in\CM(B)$, there exists $N\in\CM(A)$ with $eN=M$ (for example, take $N=FM$). Thus $e_iM=e_iN$ for any boundary vertex $i$, and so $\rk_Z(e_iM)=\rk(N)$ is constant in $i$. We define $\rk(M)$ to be this constant, the above argument showing that $\rk(M)=\rk(N)$ for any $N\in\CM(A)$ with $eN=M$. It also follows by a direct comparison of the definitions that the rank of $M\in\CM(B)$ agrees with that of $\rho(M)\in\CM(C)$.

\begin{lemma}
\label{l:inj-on-boundary}
Let $f\colon M\to N$ be a morphism in $\CM(A)$ such that its restriction $e(f)\colon eM\to eN$ is injective. Then $f$ is injective.
\end{lemma}
\begin{proof}
Since $\CM(A)$ is closed under submodules, $K=\ker{f}$ is Cohen--Macaulay. If $i\in Q_0$ is a boundary vertex, then $e_iK\subset eK=0$ since $e(f)$ is injective. Since for any $j\in Q_0$ we have $\rk_Z(e_jK)=\rk_Z(e_iK)=0$, and $e_jK$ is free over $Z$, it follows that $e_jK=0$ for all $j$, and thus that $K=0$.
\end{proof}

\begin{lemma}
\label{l:res-inj}
For $M,N\in\CM(A)$, the restriction $\Hom_A(M,N)\to\Hom_B(eM,eN)$ is injective.
\end{lemma}
\begin{proof}
As in any adjunction, the restriction map can be factored as
\[\Hom_A(M,N)\longrightarrow\Hom_A(M,FeN)\isoto\Hom_B(eM,eN),\]
where the first map is $\Hom_A(M,-)$ applied to the counit $N\to FeN$, and the second is adjunction. In this case, the counit map restricts to the identity $eN\to eN$ and so is injective by Lemma~\ref{l:inj-on-boundary}. Since $\Hom_A(M,-)$ is left exact, the restriction map is injective as required.
\end{proof}

One immediate consequence of these lemmas is the following.
\begin{proposition}
\label{p:F'-is-min}
Let $M\in\CM(B)$ and $N\leq FM$. Then $eN=M$ if and only if $F'M\leq N$.
\end{proposition}

\begin{proof}
Note that the statement makes the canonical identification $eFM=M$.
For the backwards implication, 
note that the map $\Ftil M\to FM$ restricts to the identity $M\to M$ on the boundary
and therefore $eF'M=M$. 
Hence, if $F'M\leq N$, then $eN$ is sandwiched between $eF'M=M$ and $eFM=M$, so $eN=M$.
For the forward implication, the left and right adjunctions provide universal (unit and counit) maps
\[
  \smallF eN \to N \to FeN.
\]
By general properties of adjunctions, the composition of these maps restricts to $\id\colon eN\to eN$ on the boundary (using our canonical identification), and so by Lemma~\ref{l:res-inj} agrees with the inclusion map $F'eN\to FeN$, which also has this restriction. If $eN=M$, then a similar argument shows that the second map is the given inclusion $N\to FM$. Thus we have shown in this case that the inclusion of $F'M$ into $FM$ factors over that of $N$ into $FM$, and so $F'M\leq N$.
\end{proof}

\begin{proposition} 
\label{p:matchings-to-modules}
Let $M\in\CM(B)$ with $\rk(M)=1$. Then there is a bijection
\[\theta\colon\{N\leq FM : eN=M\}\to\{\mu:eN_\mu\isom M\}\]
determined by $\theta(N)=\mu$ when $N\isom N_\mu$.
\end{proposition}

\begin{proof}
Since $eFM=M$ has rank $1$, we also have $\rk(FM)=1$ and so by Proposition~\ref{p:pm-mods} we may choose a perfect matching module structure, in the sense of Definition~\ref{d:pmm}, on $FM$. This induces such a structure on any $N\leq FM$ by Lemma~\ref{l:pm-submods}. Thus $N\isom N_\mu$ for a perfect matching $\mu$, unique by Corollary~\ref{c:proj-pm-mods}. In addition, by restriction to the boundary, the perfect matching module structure on $N$ 
induces a preferred isomorphism $eN_\mu\isoto M$, and so $\theta$ is a well-defined map.

For injectivity, suppose $\mu=\theta(N)$. Being a perfect matching module, $N$ is the image of a canonical map $N_\mu\to FM$, as in Definition~\ref{d:pmm}. Since $eN=M$, this map restricts to the preferred isomorphism $eN_\mu\isoto M$ on the boundary and, by Lemma~\ref{l:res-inj}, there can only be one such map. Thus $N$ is uniquely determined by $\mu$.
For surjectivity, let $\mu$ be a perfect matching with $eN_\mu\isom M$. This induces an isomorphism $FeN_\mu\isoto FM$ and, precomposing with the unit of the adjunction, a monomorphism $N_\mu\to FM$. This map restricts to the given isomorphism $eN_\mu\isoto M$ on the boundary, and so its image $N$ has $eN=M$ and $\theta(N)=\mu$.
\end{proof}

\begin{remark}
\label{r:combinatorics}
We can use Proposition~\ref{p:F'-is-min} to rewrite the domain
of $\theta$ in Proposition~\ref{p:matchings-to-modules} as 
\[ \{N\leq FM : eN=M\}  = \{N : \smallF M \leq N\leq F M\}. 
\]
We can also use Corollary~\ref{c:bdy-val-mod} to rewrite the 
codomain of $\theta$ in purely combinatorial terms:
\[ \{\mu\colon eN_\mu\isom M\}=\{\mu:\pmbdry{\mu}=I\},
\]
where $I\subset\C_1$ is the unique $k$-subset such that $\rho(M)\isom M_I$ \cite[Prop.~5.2]{JKS}. 
\end{remark}

\begin{lemma}
\label{l:supp-restr}
Let $D$ be a Postnikov diagram, $B$ the boundary algebra of $A_D$ and $M\in\CM(B)$ with $\rk(M)=1$. Then
\begin{enumerate}
\item \label{it:sr1} 
any two perfect matchings of $Q(D)$ with boundary module $M$ coincide on all arrows not incident with the set $S\subset Q_0$ of vertices on which $FM/F'M$ is supported, and
\item \label{it:sr2} 
if $\mu$ is such a perfect matching, and $\omega$ is a cycle of arrows in $Q(D)$ bounding a face and passing through $S$, then the unique arrow of $\mu$ in $\omega$ is incident with $S$.
\end{enumerate}
\end{lemma}

\begin{proof}
Using the bijection $\theta$ of Proposition~\ref{p:matchings-to-modules}, 
let $\mu_0=\theta(FM)$, let $\mu$ be another perfect matching with boundary module $M$ and let 
$N=\theta^{-1}(\mu)$.

Since $F'M\leq N\leq FM$ by Proposition~\ref{p:F'-is-min}, the module $N$ coincides with $FM$ away from the vertices supporting $FM/F'M$. Thus if $a\in Q_1$ is an arrow not incident with these vertices, then $e_{ta}N=e_{ta}FM$ and $e_{ha}N=e_{ha}FM$, and since $N$ is a submodule of $FM$, the action of $a$ is the same in the two module structures. It follows that $a$ is an arrow of $\mu$ if and only if it is an arrow of $\mu_0$, establishing \eqref{it:sr1}.

We first prove \eqref{it:sr2} for the perfect matching $\mu_0=\theta(FM)$. 
If every arrow of $\omega$ has both head and tail in $S$, then there is nothing to prove, so assume otherwise. 
Then there must be an arrow $a$ of $\omega$ with $ha\in S$ but $ta\notin S$. 
Since $FM/F'M$ is $0$ at $ta$, we have $e_{ta}FM=e_{ta}F'M$. 
Since $F'M$ is a submodule of $FM$, the action of $a$ takes the $Z$-generator of $e_{ta}F'M$ to an element of $e_{ha}F'M$, which is properly contained in $e_{ha}FM$ since $ha\in S$. Thus this image cannot be the $Z$-generator of the codomain, so the arrow map on $a$ is not an isomorphism and $a\in\mu_0$.

The statement for any other matching $\mu$ with boundary $M$ then follows from \eqref{it:sr1}, since $\mu$ agrees with $\mu_0$ on the arrows of $\omega$ not incident with $S$, meaning none of these arrows can appear in $\mu$.
\end{proof}

This result will help us to show that the computation of the set of perfect matchings in Proposition~\ref{p:matchings-to-modules} can be reduced to a potentially much smaller computation, only involving the vertices on which $FM/F'M$ is supported. 
To explain how this works, it will be helpful to use the description of a perfect matching as a set of edges in a bipartite graph, rather than as a set of arrows in the dual quiver.

\begin{proposition} \label{p:pm-supp}
In the setting of Lemma~\ref{l:supp-restr}, let $\plabic_M$ be the graph consisting of those edges and nodes of $\plabic(D)$ incident with the tiles corresponding to the vertices of $Q(D)$ supporting $FM/F'M$. Then the perfect matchings of $Q(D)$ with boundary module $M$ are in bijection with those of $\plabic_M$ via intersection, i.e.\ by taking a perfect matching $\mu$ to the set of edges of $\plabic_M$ dual to arrows of $\mu$. 
Precomposing with the bijection from Proposition~\ref{p:matchings-to-modules}, 
we obtain a bijection
\[
  \{N\leq FM : eN=M\} \to \{\mu:\text{$\mu$ is a perfect matching of $\plabic_M$}\}.
\]
\end{proposition}

\begin{proof}
Let $\mu$ be a perfect matching of $Q(D)$ with boundary module $M$. Each node $v$ of $\plabic_M$ corresponds to a face of $Q(D)$ whose boundary cycle $\omega$ intersects $S$, the support of $FM/F'M$. The edges of $\plabic_M$ incident with $v$ are dual to arrows of $\omega$ incident with $S$, and by Lemma~\ref{l:supp-restr}(2), one of these arrows is the unique arrow of $\omega$ lying in $\mu$. Thus intersection indeed gives a perfect matching of $\plabic_M$.

It remains to show that any perfect matching of $\plabic_M$ arises in this way. 
Let $\mu_0$ be the matching of $Q(D)$ such that $FM\cong N_{\mu_0}$, 
and let $\mu$ be a perfect matching of $\plabic_M$. 
Then we may take $\hat{\mu}$ to be the set of arrows of $Q(D)$ dual to the edges of $\mu$, 
together with those arrows in $\mu_0$ not incident with $S$. 
By Lemma~\ref{l:supp-restr}(2) again, $\hat{\mu}$ is a perfect matching of $Q(D)$. 
By construction, the intersection of $\hat{\mu}$ with $\plabic_M$ is $\mu$, 
and it remains to check that $\hat{\mu}$ has the correct boundary module.

Since $eF'M=M=eFM$, the support of $FM/F'M$ does not contain any boundary vertices, and hence $\plabic_M$ contains none of the half-edges. Thus $\pmbdry{\hat{\mu}}=\pmbdry{\mu_0}$. 
By Corollary~\ref{c:bdy-val-mod}, we see that $\hat{\mu}$ has the same boundary module as $\mu_0$, namely $M$.
\end{proof}

\begin{remark}
The set $\{N\leq FM:eN=M\}$ is naturally a poset under inclusion. 
It has the unique maximal element $FM$ and, by Proposition~\ref{p:F'-is-min}, the unique minimal element $F'M$. 
Thus the bijection in Proposition~\ref{p:matchings-to-modules} puts a poset structure on the set of perfect matchings $\mu$ with $eN_\mu\isom M$, 
that is, with $\pmbdry{\mu}=I$ for the appropriate $k$-subset $I$ (cf.~Remark~\ref{r:combinatorics}),
and there are unique maximal and minimal matchings (similar to \cite[Defn.~4.7]{MSW}). 
We also get a poset structure on the perfect matchings of the subgraph $\plabic_M\subset\plabic(D)$ from Proposition~\ref{p:pm-supp}. When the full subquiver of $Q(D)$ on the vertices supporting $FM/F'M$ is an orientation of an $\mathsf{A}_n$ quiver, so in particular $FM/F'M$ is an indecomposable string module, 
Proposition~\ref{p:pm-supp} corresponds to \cite[Thm.~3.9]{CS3}.
\end{remark}

The quotient $FM/F'M$ has another description, which plays a key role later on.
To obtain this description, we use that the dimer algebra of a Postnikov diagram is internally 3-Calabi--Yau \cite[Defn.~2.1]{Pre}, 
which is proved in \cite[Thm.~3.7]{Pre3}, and has the following consequences.

\begin{proposition} \label{p:double-cen}
Let $D$ be a connected Postnikov diagram, 
with associated dimer algebra $A=A_D$ and boundary algebra $B=eAe$, and set $T=eA$. 
\begin{enumerate}
\item[(i)] The natural map $A \to \End_B(T)\op$ is an isomorphism of algebras.
\item[(ii)]  $\Ext^1_B(T,T)=0$.
\end{enumerate}
In particular, 
\begin{enumerate}
\item[(iii)] The natural map $Ae \to \Hom_B(T,B)$ is an isomorphism of $A$-modules.
\item[(iv)]  $\Ext^1_B(T,B)=0$.
\end{enumerate}
\end{proposition}

\begin{proof}
Statements (i) and (ii) are among the conclusions of \cite[Thm.~4.1]{Pre}, and the pair $(A,e)$ satisfies the assumptions of this theorem by \cite[Thm.~3.7, Prop.~4.4]{Pre3}.
Then (iii) and (iv) follow immediately, since $Te=B$.
\end{proof}

It in fact follows from \cite[Thm.~3.7]{Pre3} together with the general theory from \cite{Pre} that $T$ is a cluster-tilting object in the Frobenius category $\GP(B)$ of Gorenstein projective $B$-modules, in which $B$ is injective. We will return to this in Section~\ref{sec:newCC}, but for now we need only the resulting vanishing of extension groups in Proposition~\ref{p:double-cen}.

\begin{corollary}
\label{cor:smallF}
In the setting of Proposition~\ref{p:double-cen}, let $M\in\CM(B)$. Then $F'M$ is the subspace of $FM$ consisting of maps factoring through a projective $B$-module, 
and hence 
\begin{equation*}
 FM/F'M = \stHom_B(\cto,M) = G \syz M,
\end{equation*}
where $\syz M$ is a first syzygy of $M$, i.e.\ the kernel of a projective cover,
and $G =\Ext^1_B(\cto,-)$.
\end{corollary}

\begin{proof}
By Proposition~\ref{p:double-cen}(iii) the map $\iota_M$ of \eqref{eq:iota} 
may be identified with the composition map
\[
  \Hom_B(\cto,B)\otimes_B \Hom_B(B,M) \to \Hom_B(\cto,M)
\]
and the first equality follows.

Consider a short exact sequence
\begin{equation}\label{eq:syzdef}
  0\lra{} \syz M \lra{} \projcov M \lra{} M \lra{} 0
\end{equation}
where $\projcov M \to M$ is a projective cover of $M$. 
Applying the functor $F=\Hom_B(\cto,-)$ to \eqref{eq:syzdef} yields the long exact sequence
\begin{equation}\label{eq:les-syz}
 0\lra{} F\syz M \lra{} F\projcov M \lra{} F M \lra{} G \syz M \lra{} 0 
\end{equation}
where the final zero follows from Proposition~\ref{p:double-cen}(iv),
since $\projcov M\in \add (B)$.
Thus the second equality follows.
\end{proof}

While $\syz M$ depends on the choice of projective cover $\projcov M$, since we do not insist that the cover is minimal, the $A$-module $G\syz M$ is independent of this choice. 
By Corollary~\ref{cor:smallF}, the image of the middle map in \eqref{eq:les-syz} is $F'M$, 
yielding the two exact sequences
\begin{align}
  0\lra{} F\syz M &\lra{} F\projcov M \lra{} \smallF M \lra{} 0, \label{eq:FPi}\\
 0\lra{} \smallF M &\lra{} F M \lra{} G \syz M \lra{} 0. \label{eq:GOm}
\end{align}

\section{A projective resolution}  \label{sec:proj-res}

\newcommand{\up}[3]{p^{#1}_{#2,#3}}
\newcommand{\comisom}{\psi}

In this section, we construct an explicit projective resolution for each perfect matching module. This will play a key role for us later on, both in Section~\ref{sec:MuSp} when determining which perfect matchings describe the indecomposable projective modules of the dimer algebra of a Postnikov diagram, and in Sections~\ref{sec:newMS} and \ref{sec:newCC} when we relate the combinatorial information appearing in Marsh--Scott's dimer partition function to the homological information in the Caldero--Chapoton cluster character formula.

Let $Q$ be a dimer model with boundary. Recall that the dimer algebra $A=A_Q$ is a $Z$-algebra, for $Z=\powser{\CC}{t}$, as in Definition~\ref{d:centralelement}.

Given a perfect matching $\mu$ of $Q$, let $Q_1^\mu=Q_1\setminus\mu$. We write
\[Q_2^\mu=\{f\cup f':\text{$f,f'\in Q_2$ with boundaries sharing some arrow of $\mu$}\}.\]
In other words, $Q_2^\mu$ is obtained from the set of faces of $Q$ by merging those faces adjacent along an arrow in the matching $\mu$, and deleting those whose intersection with $\mu$ is a boundary arrow.
In what follows, it will be convenient to identify $Q_2^\mu$ with the set of defining relations $r(\beta)=p_\beta^+-p_\beta^-\in\compl{\CC Q}$ corresponding to internal arrows $\beta\in\mu$; we do this by identifying $r(\beta)$ with the union of the two faces containing $\beta$.

We will also want to consider various (projective) $A$-modules of the form
\[\bigoplus_{x\in X}Ae_{hx}\otimes_Z e_{tx}N_\mu,\]
where $X$ is some set together with head and tail maps $h,t\colon X\to Q_0$, and $N_\mu$ is the module attached to $\mu$ in Definition~\ref{d:n-mu}. For the rest of the section, we will write $\otimes=\otimes_Z$. Since $e_jN_\mu=Z$ for all $j$ by definition, each element of $Ae_{hx}\otimes e_{tx}N_\mu$ is of the form $a\otimes 1$ for some unique $a\in Ae_{hx}$. Denoting the image of $a\otimes 1$ under the map $Ae_{hx}\otimes e_{tx}N_\mu\to\bigoplus_{x\in X}Ae_{hx}\otimes e_{tx}N_\mu$ by $a\otimes[x]$, each element of this direct sum can be written
\[
\sum_{x\in X}a_x\otimes[x],
\]
for some unique elements $a_x\in Ae_{hx}$. 
To define an $A$-module homomorphism $\varphi\colon\bigoplus_{x\in X}Ae_{hx}\otimes e_{tx}N_\mu\to M$, it suffices to specify $\varphi(e_{hx}\otimes[x])\in e_{hx}M$ for each $x\in X$, which may be done freely.

Now consider the complex
\[
\res{\mu}\colon 
\bigoplus_{r\in Q_2^\mu} Ae_{hr}\otimes e_{tr}N_\mu  \lra{\partial_2}
\bigoplus_{\alpha\in Q_1^\mu} Ae_{h\alpha}\otimes e_{t\alpha}N_\mu \lra{\partial_1}
\bigoplus_{j\in Q_0} Ae_j\otimes e_jN_\mu \lra{\partial_0}
N_\mu,
\]
whose terms are in homological degrees $2$, $1$, $0$ and $-1$, and whose maps are defined as follows. 
First, $\partial_0$ is just the action of $A$ on $N_\mu$.
For $a\in e_kAe_j$, we also write $a\colon Ae_k\to Ae_j$ for right multiplication by $a$, and denote by $a_*\colon e_jN_\mu\to e_kN_\mu$ the action of $a$ on $N_\mu$. 
Then, for any $\alpha\in Q_1^\mu$, the $\alpha$ component of $\partial_1$ is
\[
(\alpha\otimes1,-1\otimes\alpha_*)\colon Ae_{h\alpha}\otimes e_{t\alpha}N_\mu\to(Ae_{t\alpha}\otimes e_{t\alpha}N_\mu)\oplus(Ae_{h\alpha}\otimes e_{h\alpha}N_\mu).
\]
Since $\alpha$ is unmatched, $\alpha_*$ is the identity $Z\to Z$, so we have
\[
\partial_1(e_{h\alpha}\otimes[\alpha])=\alpha\otimes[t\alpha]-e_{h\alpha}\otimes[h\alpha].
\]
For any path $p=\alpha_m\dotsm\alpha_1$ of $Q$ and any arrow $\alpha\in Q_1$, we define
\[\Delta_\alpha(p)=\sum_{\alpha_i=\alpha}\alpha_m\dotsm\alpha_{i+1}\otimes(\alpha_{i-1}\dotsm\alpha_1)_*\colon Ae_{hp}\otimes e_{tp}N_\mu\to Ae_{h\alpha}\otimes e_{t\alpha}N_\mu.\]
The components of $\partial_2$ are then
\[\partial_2^{r(\beta),\alpha}=\Delta_\alpha(p^+_\beta)-\Delta_\alpha(p^-_\beta)\colon Ae_{t\beta}\otimes e_{h\beta}N_\mu\to Ae_{h\alpha}\otimes e_{t\alpha}N_\mu\]
for $\beta\in\mu$ internal and $\alpha\notin\mu$. Since $\beta\in\mu$, none of the arrows of $p^+_\beta$ or $p^-_\beta$ are in $\mu$, and so writing $p_\beta^+=\alpha^+_m\dotsm\alpha^+_1$ and $p_\beta^-=\alpha^-_\ell\dotsm\alpha^-_1$, we have
\[\partial_2(e_{t\alpha}\otimes[r(\beta)])=\sum_{i=1}^m\alpha^+_m\dotsm\alpha^+_{i+1}\otimes[\alpha^+_i]-\sum_{i=1}^\ell\alpha^-_\ell\dotsm\alpha^-_{i+1}\otimes[\alpha^-_i],\]
and in particular $\partial_2$ takes values in the appropriate sum of projective modules. In the above formula, an empty product of arrows is interpreted as the appropriate idempotent (for example, $\alpha^+_m\dotsm\alpha^+_{m+1}$ should be read as $e_{h\alpha^+_m}$).

The goal of this section is to prove, in the case that $A=A_D$ for $D$ a connected Postnikov diagram, that $\res{\mu}$ is exact; in other words, its non-negative degree part is
a projective resolution of the perfect matching module $N_\mu$. A priori, we will show this for any dimer algebra $A=A_Q$ which is thin, and for which the cell complex $Q$ has $\H^2(Q)=0$, although it will then follow from exactness of $\res{\mu}$ that $|Q|$ is the disc.

Using the grading $\deg_\mu$ from Definition~\ref{d:grading}, each map $\partial_i$ in $\res{\mu}$ has degree $0$, 
since this is true of every arrow in $Q_1^\mu$. 
This makes $\res{\mu}$ into a graded complex, which is exact if and only if its degree $d$ part $(\res{\mu})_d$ is exact for all $d$. 
Moreover, as vector spaces, each complex $(\res{\mu})_d$ decomposes as the direct sum
\[(\res{\mu})_d=\bigoplus_{i\in Q_0}e_i(\res{\mu})_d,\]
so, extending the refinement by degree, $\res{\mu}$ is exact if and only if $e_i(\res{\mu})_d$ is exact for all $i\in Q_0$ and $d\in\ZZ$.

\begin{remark}\label{rem:A0}
The degree $0$ part of $\res{\mu}$ is a complex of modules for the algebra $A_0$, 
which can be presented as the path algebra of the quiver $(Q_0,Q_1^\mu)$ 
modulo the ideal of relations generated by $r\in Q_2^\mu$, i.e.\ $r=r(\beta)$ for $\beta\in\mu$. 
In fact, $(\res{\mu})_0$ is the start of the standard resolution (see for example \cite[1.2]{BK}) of the $A_0$-module $(N_\mu)_0$,
which is given by $\CC$ at each vertex, with all arrows (in $Q_1^\mu$) acting as the identity.
Thus $(\res{\mu})_0$ is always exact in homological degrees $1$, $0$ and $-1$.
\end{remark}

In order to study the complexes $e_i(\res{\mu})_d$, we interpret them topologically. 
Recall that the quiver with faces $Q$ can be thought of as a cell complex $(Q_0,Q_1,Q_2)$ for the topological space $|Q|$.
Given a subset $S\subset Q_0$, we denote by $Q[S]$ the full subcomplex of $Q$ with vertex set $S$,
that is, the edges of $Q[S]$ are those edges of $Q$ with both endpoints in $S$ 
and the faces of $Q[S]$ are those faces of $Q$ incident only with vertices in $S$. 
The geometric realisation $|Q[S]|$ is naturally embedded into $|Q|$. 
We also consider the cell complex $Q^\mu=(Q_0,Q_1^\mu,Q_2^\mu)$, 
used above in the construction of the chain complex $\res{\mu}$, 
and its full subcomplexes $Q^\mu[S]$ for $S\subset Q_0$.

For any perfect matching $\mu$ on $Q$, vertex $i\in Q_0$ and $d\geq 0$,
we define the subset
\begin{align*}
 S(\mu,i,d) &= \{j\in Q_0: \text{there is a path $p\colon j\to i$ with $\deg_\mu(p)=d$} \} \\
  &= \{j\in Q_0: (e_i A e_j)_d\neq 0\}
 \end{align*}
 
\begin{lemma}
\label{l:heads-in-S}
An arrow $\alpha\in Q_1^\mu$, respectively a face $r\in Q_2^\mu$, lies in $Q^\mu[S(\mu,i,d)]$ if and only if $h\alpha$, respectively $hr$, does.
\end{lemma}
\begin{proof}
Any $\alpha\in Q_1^\mu$ has $\deg_\mu(\alpha)=0$. Thus if $j\in S(\mu,i,d)$, so that there is a path $p\colon j\to i$ with $\deg_\mu(p)=d$, then any $\alpha\in Q_1^\mu$ with $h\alpha=j$ determines a path $p\alpha\colon t\alpha\to i$ with $\deg_\mu(p\alpha)=d$, and hence $t\alpha\in S(\mu,i,d)$. Thus any such $\alpha$ has both endpoints in $S(\mu,i,d)$ and so lies in $Q^\mu[S(\mu,i,d)]$.

Similarly, every arrow in the boundary of $r\in Q_2^\mu$ has degree $0$, and any vertex in this boundary begins a path consisting of these arrows and ending at $hr$. Thus if $hr\in S(\mu,i,d)$, so is every vertex incident with $r$.
\end{proof}

If $A$ is thin, in the sense of Definition~\ref{d:thin}, then any two paths $p\colon j\to i$ with $\deg_\mu(p)=d$, as appearing in the definition of $S(\mu,i,d)$, are F-term equivalent, and so all determine the same element of $A$, which we denote by $\up{d}{i}{j}$. 
This element is then a preferred basis for the one-dimensional vector space $(e_i A e_j)_d$. Recall from Proposition~\ref{p:thin} that if $A=A_D$ for some Postnikov diagram $D$, then $A$ is thin.

\begin{proposition}
\label{prop:cohom-computation}
When $A$ is thin, the complex $e_i(\res{\mu})_d$ computes the reduced cohomology of the cell complex 
$Q[S(\mu,i,d)]$ with coefficients in $\CC$.
\end{proposition}

\begin{proof}
Write $S=S(\mu,i,d)$. 
The first step is to observe that $Q[S]$ is homotopy equivalent to $Q^\mu[S]$.
Indeed, these two complexes differ only when there is an $\alpha\in\mu$ such that $t\alpha\in S$.
If $\alpha$ is contained in two faces of $Q[S]$, then these faces are merged in $Q^\mu[S]$, 
leaving the geometric realisation unchanged. 
If $\alpha$ is contained in only a single face $F$ of $Q[S]$, 
then it lies in the boundary of the geometric realisation 
and we can contract $F$ onto the union of its other edges in this realisation, 
corresponding to the removal of $F$ in $Q^\mu[S]$. 
Since $\mu$ is a perfect matching, the contractions operate independently of one another
and collectively describe a homotopy equivalence between $|Q[S]|$ and $|Q^\mu[S]|$.
Thus we can instead prove the result for the reduced cohomology of the cell complex 
$Q^\mu[S]$.

Each face $r\in Q^\mu_2[S]$ is the union of two faces of $Q[S]$, one black and one white; 
we write $r^+$ for the set of arrows of $Q^\mu_1[S]$ lying in the boundary of the black face, 
and $r^-$ for the set of arrows of $Q^\mu_1[S]$ lying in the boundary of the white face. 
Then the chain complex computing the reduced cohomology of the cell complex, 
using an anticlockwise orientation on faces and the given orientation of the edges,
has non-zero terms 
\[
\zeta\colon 
\bigoplus_{r\in Q_2^\mu[S]} \CC\cdot r \lra{\delta_2}
\bigoplus_{\alpha\in Q_1^\mu[S]} \CC\cdot\alpha \lra{\delta_1}
\bigoplus_{j\in S} \CC\cdot j \lra{\delta_0} \CC,
\]
in homological degrees $2,1,0,-1$.
The maps are defined on generators as follows.
\begin{gather*}
\delta_0(1\cdot j)=1,\qquad
\delta_1(1\cdot\alpha)=1\cdot t{\alpha}-1\cdot h{\alpha},\\
\delta_2(1\cdot r)=\sum_{\alpha^+\in r^+}1\cdot\alpha^+-\sum_{\alpha^-\in r^-}1\cdot\alpha^-
\end{gather*}
We now construct an isomorphism of complexes $\comisom\colon \zeta \to e_i(\res{\mu})_d$.
In homological degree $-1$, we have, by definition, $e_i N_\mu=Z$, so $(e_i N_\mu)_d=\CC\cdot t^d$
and so we start with $\comisom_{-1}\colon 1\mapsto t^d$.

To compare the other terms, we note that each summand in the higher degree terms of $\res{\mu}$ is of the form $Ae_{hx}\otimes e_{tx}N_\mu$, for some cell $x$ of $Q^\mu$. 
As before, we denote the single generator of $e_{tx}N_\mu$ by $[x]$,
which has $\deg_\mu [x]=0$.
Thus the corresponding term of $e_i(\res{\mu})_d$ is $(e_i Ae_{hx}\otimes e_{tx}N_\mu)_d$,
which, as $A$ is thin, is either a one-dimensional vector space with basis $\up{d}{i}{hx}\otimes [x]$, if $hx\in S$, or 
zero, if $hx\notin S$. 
Hence we may define
\begin{gather*}
\comisom_0(1\cdot j) =\up{d}{i}{j}\otimes [j],\quad
\comisom_1(1\cdot\alpha) =\up{d}{i}{h\alpha}\otimes [\alpha],\quad
\comisom_2(1\cdot r) =\up{d}{i}{hr}\otimes [r].
\end{gather*}
Once we have checked that this is a morphism of complexes, it follows that it is a (well-defined) isomorphism by Lemma~\ref{l:heads-in-S},
because $\alpha\in Q_1^\mu[S]$ if and only if $h\alpha\in S$ and $r\in Q_2^\mu[S]$ if and only if $hr\in S$.
There is nothing to prove in homological degree $0$ because $Q_0^\mu[S]=S$, as already used in writing down the complex $\zeta$.

It remains to check that the maps $\comisom_\bullet$ commute with the differentials.
For $j\in S$,
\[
\partial_0\comisom_0(1\cdot j)=
\partial_0(\up{d}{i}{j}\otimes[j])=t^d = \comisom_{-1}\delta_0(1\cdot j),
\]
where the middle equality follows precisely because $\deg_\mu \up{d}{i}{j} = d$, 
so $\up{d}{i}{j}$ acts on $N_\mu$ as multiplication by $t^d$. 
For $\alpha\in Q_1^\mu[S]$, we check 
$\partial_1\comisom_1(1\cdot \alpha)= \comisom_0\delta_1(1\cdot \alpha)$, i.e.
\[
\partial_1\bigl( \up{d}{i}{h\alpha}\otimes [\alpha] \bigr) 
= \up{d}{i}{h\alpha}\alpha\otimes[t\alpha]-\up{d}{i}{h\alpha}\otimes[h\alpha]
=\comisom_0(1\cdot t\alpha-1\cdot h\alpha),
\]
because $\up{d}{i}{h\alpha}\alpha=\up{d}{i}{t\alpha}$, that is, it is a path $t\alpha\to i$ of degree $d$.

Finally, when $r=r(\beta)\in Q_2^\mu[S]$, write $p_\beta^+=\alpha^+_M\dotsm\alpha^+_1$ 
and $p_\beta^-=\alpha^-_L\dotsc\alpha^-_1$, 
so that $r^+=\{\alpha^+_1,\dotsc,\alpha^+_M\}$ and $r^-=\{\alpha^-_1,\dotsc,\alpha^-_L\}$.
Then $\partial_2\comisom_2(1\cdot r)= \comisom_1\delta_2(1\cdot r)$, i.e.
\begin{align*}
\partial_2\bigl(\up{d}{i}{hr}\otimes[r]\bigr)
&= \sum_{m=1}^M \up{d}{i}{hr} \alpha^+_M\dotsm\alpha^+_{m+1}\otimes[\alpha^+_m]
- \sum_{\ell=1}^L\up{d}{i}{hr} \alpha^-_L\dotsm\alpha^-_{\ell+1}\otimes[\alpha^-_\ell]\\
&=\sum_{m=1}^M \up{d}{i}{h\alpha^+_m}\otimes[\alpha^+_m]-\sum_{\ell=1}^L\up{d}{i}{h\alpha^-_\ell}\otimes[\alpha^-_\ell] \\
&=\comisom_1\left( \sum_{\alpha^+\in r^+}1\cdot\alpha^+-\sum_{\alpha^-\in r^-}1\cdot\alpha^- \right),
\end{align*}
completing the proof.
\end{proof}

\begin{lemma}
\label{lem:highdeg}
Assume $A$ is thin. Then for any $i\in Q_0$ and sufficiently large $d$, the cohomology of the complex $e_i(\res{\mu})_d$ is the reduced cohomology of $|Q|$.
\end{lemma}

\begin{proof}
By Proposition~\ref{prop:cohom-computation}, 
the cohomology of $e_i(\res{\mu})_d$ is that of $Q[S(i,\mu,d)]$. 
On the other hand, every $j\in Q_0$ admits some path $j\to i$, hence one of minimal degree. So, if $d$ is larger than the maximum, over $j\in Q_0$, of these minimal degrees,
then $S(i,\mu,d)=Q_0$ and $Q[S(i,\mu,d)]=Q$.
\end{proof}

\begin{lemma}
\label{lem:d2inj}
Assume $A$ is thin and $\H^2(Q)=0$. Then, for any $\mu$, the map $\partial_2$ in $\res{\mu}$ is injective.
\end{lemma}

\begin{proof}
As $A$ is  thin, $\bigoplus_{r\in Q_2^\mu}Ae_{hr}\otimes e_{tr}N_\mu$ is a free $Z$-module and hence so is the submodule $K=\ker{\partial_2}$. 
However, by Lemma~\ref{lem:highdeg} and the assumption on $Q$, 
we have $K_d=0$, for sufficiently large $d$, and hence $K=0$.
\end{proof}

\begin{lemma}
\label{lem:deg0}
Assume $A$ is thin and $\H^2(Q)=0$. Then the complex $(\res{\mu})_0$ is exact for all $\mu$.
\end{lemma}

\begin{proof}
Injectivity of $\partial_2$ follows from Lemma~\ref{lem:d2inj} 
and the exactness elsewhere has already been noted in 
Remark~\ref{rem:A0}.
\end{proof}

We are now ready to complete the proof that $\res{\mu}$ is exact under the assumptions of Lemma~\ref{lem:deg0}.
Our strategy is to show that each of the subsets $S=S(\mu,i,d)$ is equal to $S(\nu,i,0)$, 
for some other matching $\nu$.
Then the cohomology of $Q[S]$ is computed by both of the complexes $e_i(\res{\mu})_d$ and
$e_i(\res{\nu})_0$. Since the second of these complexes is exact by Lemma~\ref{lem:deg0}, it will follow that $e_i(\res{\mu})_d$ is also exact.
To construct the new matching $\nu$, we use the following two results.

\begin{lemma}
\label{lem:degree-change}
Assume $A$ is thin. Let $i\in Q_0$ be a vertex, $\alpha\in Q_1$ be an arrow, and $\mu$ be a matching of $Q$. For each $j\in Q_0$, choose a minimal degree path $p_j\colon j\to i$, and for each $\alpha\in Q_1$ write $\varepsilon_\alpha=\deg_\mu(\alpha)$; i.e.\ $\varepsilon_\alpha=1$ when $\alpha\in\mu$, and $\varepsilon_\alpha=0$ otherwise.
Then
\[\deg_\mu(p_{t\alpha})-\varepsilon_\alpha\leq\deg_\mu(p_{h\alpha})\leq\deg_\mu(p_{t\alpha})+1-\varepsilon_\alpha.\]
It follows that $\deg_\mu(p_{t\alpha})$ can be strictly smaller than $\deg_\mu(p_{h\alpha})$ only when $\alpha\notin\mu$, and strictly larger only when $\alpha\in\mu$, the difference being bounded by $1$ in each case.
\end{lemma}

\begin{proof}
The path $p_{h\alpha}\alpha\colon t\alpha\to i$ has degree $\deg_\mu(p_{h\alpha})+\varepsilon_\alpha$, and the first inequality follows. Consider a path $q\colon h\alpha\to t\alpha$ completing $\alpha$ to the boundary of a face of $Q$. Since $\mu$ is a perfect matching, this boundary has degree $1$, and so $\deg_\mu(q)=1-\varepsilon_\alpha$. The second inequality then follows by considering the path $p_{t\alpha}q\colon h\alpha\to i$, of degree $\deg_\mu(p_{t\alpha})+1-\varepsilon_\alpha$.
\end{proof}

\begin{proposition} \label{prop:rotate}
Assume $A$ is thin, and let $\mu$ be a matching of $Q$, $i\in Q_0$ and $d\geq1$. 
Then there exists a matching $\nu$ such that $S(\mu,i,d)=S(\nu,i,d-1)$.
\end{proposition}

\begin{proof}
As in Lemma~\ref{lem:degree-change}, choose a minimal degree path $p_j\colon j\to i$ for each $j\in Q_0$. Define
\begin{align*}
X&=\{\alpha\in Q_1:\deg_\mu(p_{t\alpha})=d,\ \deg_\mu(p_{h\alpha})=d-1\},\\
Y&=\{\beta\in Q_1:\deg_\mu(p_{t\alpha})=d-1,\ \deg_\mu(p_{h\alpha})=d\}.
\end{align*}
Then, by Lemma~\ref{lem:degree-change}, we have $X\subset\mu$, and $\mu\cap Y=\varnothing$. 
Let $\nu=(\mu\setminus X)\cup Y$. We claim first that $\nu$ is a perfect matching, and secondly that, for any $j\in Q_0$,
\[
\deg_\nu(p_{j})= 
\begin{cases}
\deg_\mu(p_{j}) & \text{if $\deg_\mu(p_{j})\leq d-1$},\\
\deg_\mu(p_{j})-1 & \text{if $\deg_\mu(p_{j})\geq d$}.
\end{cases}
\]
From these two claims it immediately follows that $S(\mu,i,d)=S(\nu,i,d-1)$, because
\[
S(\mu,i,d)=\{j\in Q_0:\deg_{\mu}(p_{j})\leq d\}.
\]
 
We now prove the claims, beginning with the statement that $\nu$ is a perfect matching. Let $F$ be a face of $Q$, the boundary of which contains exactly one arrow $\alpha\in\mu$, which lies either in $X$ or $\mu\setminus X$. For $F$ to intersect $\nu$ in exactly one arrow, we must show that the boundary of $F$ contains exactly one arrow from $Y$ in the case that $\alpha\in X$, and no arrows from $Y$ if $\alpha\in\mu\setminus X$.

First assume $\alpha\in X$, and consider the path $q$ completing $\alpha$ to $\bdry F$. By the assumption on $\alpha$, we have $\deg_\mu(p_{hq})=d$ and $\deg_\mu(p_{tq})=d-1$. Moreover, $q$ contains no arrows of $\mu$ and so $\deg_\mu(p_{h\gamma})\geq\deg_{\mu}(p_{t\gamma})$ for every arrow $\gamma$ of $q$, by Lemma~\ref{lem:degree-change}. It follows that these two degrees are equal for all but one $\gamma$, which is the unique element of $\partial F\cap Y$.

Now assume that $\beta\in\bdry F\cap Y\ne\varnothing$, and let $\alpha$ be the unique arrow of $\mu$ in $\bdry F$; we aim to show that $\alpha\in X$. Let $q$ be the path completing $\beta$ to $\bdry F$, so that $\deg_\mu(p_{hq})=d-1$ and $\deg_\mu(p_{tq})=d$ by the assumption on $\beta$. Using Lemma~\ref{lem:degree-change} as in the previous case, $\deg_\mu(p_{h\gamma})\geq\deg_{\mu}(p_{t\gamma})$ for every arrow $\gamma\ne\alpha$ of $q$, since these arrows are not in $\mu$, whereas $\deg_\mu(p_{h\alpha})\geq \deg_\mu(p_{t\alpha})-1$ since $\alpha\in\mu$. Comparing to $\deg_\mu(p_{hq})$ and $\deg_\mu(p_{tq})$ we see that all of these inequalities are in fact equalities, and so $\alpha\in X$.

Now to prove the second claim, concerning the degrees $\deg_\nu(p_j)$, observe first that for any arrow $\alpha$ of the path $p_{j}$ we get inequalities
\[
\deg_\mu(p_{t\alpha})-1\leq\deg_\mu(p_{h\alpha})\leq\deg_\mu(p_{t\alpha}).
\]
Indeed, the first inequality holds for any arrow $\alpha\in Q_1$, since $\deg_\mu(p_{t\alpha})\leq\deg_\mu(p_{h\alpha}\alpha)$ and $\deg_\mu(\alpha)\leq 1$. For the second, write $p_j=p_1\alpha p_2$. By minimality of $\deg_\mu(p_j)$, we must have $\deg_\mu(p_1)=\deg_\mu(p_{h\alpha})$ and $\deg_\mu(p_1\alpha)=\deg_\mu(p_{t\alpha})$, and so the second inequality follows.

From these inequalities, it follows that the degrees $\deg_\mu(p_\ell)$ are weakly decreasing as $\ell$ runs through the vertices of $p_j$ in the direction of the path, and two successive degrees in this sequence differ by at most $1$. Thus if $\deg_\mu(p_{j})\leq d-1$, so that $j\in S(\mu,i,d-1)$, 
then $p_{j}$ only passes through vertices in this subset, 
so it contains no arrows of $X$ or $Y$. 
Hence the matchings $\mu$ and $\nu$ agree on all arrows of  $p_{j}$ and
$\deg_\nu(p_{j})=\deg_\mu(p_{j})$. On the other hand, if $\deg_\mu(p_{j})\geq d$, 
then the path $p_{j}$ contains a unique arrow $\alpha\in X\subset\mu$ and no arrow of $Y$. 
By construction, the arrow $\alpha$ does not appear in $\nu$, 
but the matchings $\mu$ and $\nu$ agree on the other arrows of the path. 
Thus $\deg_\nu(p_{j})=\deg_\mu(p_{j})-1$,
completing the argument.
\end{proof}

Putting everything together, we obtain the following theorem.

\begin{theorem}
\label{thm:match-res}
Let $Q$ be a dimer model with boundary such that $A_Q$ is thin and $\H^2(Q)=0$, and let $N_\mu$ be the $A$-module corresponding to a perfect matching $\mu$. 
Then the complex (extended by zeroes)
\[
\res{\mu}\colon
\bigoplus_{f\in Q_2^\mu}Ae_{hf}\otimes e_{tf}N_\mu \lra{\partial_2}
\bigoplus_{\alpha\in Q_1^\mu}Ae_{h\alpha}\otimes e_{t\alpha}N_\mu \lra{\partial_1}
\bigoplus_{j\in Q_0}Ae_j\otimes e_jN_\mu \lra{\partial_0} N_\mu
\]
is exact, yielding a projective resolution of $N_\mu$. This result applies in particular to the case that $A=A_D$ for a connected Postnikov diagram $D$.
\end{theorem}

\begin{proof}
To recap, we show that $e_i(\res{\mu})_d$ is exact, for every vertex $i\in Q_0$ and degree $d\geq 0$.
By Proposition~\ref{prop:cohom-computation}, 
$e_i(\res{\mu})_d$ computes the (reduced) cohomology of $Q[S]$, for $S=S(\mu,i,d)$.
By applying Proposition~\ref{prop:rotate} inductively, we may construct a matching $\nu$ for which $S=S(\nu,i,0)$.
Hence the (reduced) cohomology of $Q[S]$ vanishes, because 
(again by Proposition~\ref{prop:cohom-computation}) 
it is computed by $e_i(\xi_\nu)_0$, which is exact by Lemma~\ref{lem:deg0}.
Thus $e_i(\res{\mu})_d$ is exact, as required.

If $D$ is a connected Postnikov diagram, then $\H^2(Q(D))=0$ since $|Q(D)|$ is a disc, and $A_D$ is thin by Proposition~\ref{p:thin}.
\end{proof}

\begin{corollary}
If $Q$ is a dimer model with boundary admitting a perfect matching $\mu$, and such that $\H^2(Q)=0$ and $A_Q$ is thin, then $|Q|$ is a disc.
\end{corollary}
\begin{proof}
The complex $\res{\mu}$ is exact by Theorem~\ref{thm:match-res}, and so $e_i(\res{\mu})_d$ must also be exact for all $d\geq0$ and all $i\in Q_0$. By Lemma~\ref{lem:highdeg}, the reduced cohomology of $|Q|$ vanishes, and so $|Q|$ is a contractible compact surface with boundary, i.e.\ a disc.
\end{proof}

\newcommand{\vmut}{\textup{int}}
\newcommand{\amut}{\textup{int}}
\newcommand{\vfro}{\textup{bdry}}
\newcommand{\afro}{\textup{bdry}}
\newcommand{\bl}{\operatorname{bl}}
\newcommand{\wh}{\operatorname{wh}}
\newcommand{\wtms}[1][]{\wt[#1]_{\operatorname{MS}}}

As an immediate consequence of Theorem~\ref{thm:match-res}, $N_\mu$ has finite projective dimension 
and so we can associate to it a class $[N_\mu]$ in the (free abelian) Grothendieck group $\Kgp(\proj A)$.
Indeed, the resolution enables us to write several explicit expressions for this class,
written in the canonical basis $[P_j]=[Ae_j]$, for $j\in Q_0$.

We do this for the case $A=A_D$ for a connected Postnikov diagram $D$.
Recall from Definition~\ref{d:dimermodel} that vertices and arrows of $Q$ are either internal (abbreviated `int' below) or boundary (abbreviated `bdry'). Note that every internal arrow $\gamma$ of $Q$ is contained in a unique `black cycle', the boundary of a black face of $Q$, which we denote by $\bl(\gamma)$ and which corresponds to a black node of the dual bipartite graph $\plabic(D)$.
We write $\bl_0(\gamma)$ for the vertices in this cycle and 
\begin{equation}\label{eq:bl'0}
	\bl'_0(\gamma) = \bl_0(\gamma)\setmin\{t\gamma,h\gamma\}.
\end{equation}
Similarly, $\gamma$ is also contained in a unique white cycle $\wh(\gamma)$ passing through vertices $\wh_0(\gamma)$, and we write
\begin{equation}\label{eq:wh'0}
\wh'_0(\gamma)=\wh_0(\gamma)\setmin\{t\gamma,h\gamma\}.
\end{equation}
Without loss of generality (i.e.\ without changing $A_D$ up to isomorphism), we may assume that $D$ is standardised (Remark~\ref{r:boundary-convention}). If $D$ is $\white$-standardised then the boundary arrows of $Q(D)$ are $\alpha_i=\canCB(x_i)$ for $i\in\C_1$, whereas if $D$ is $\black$-standardised then the boundary arrows are $\beta_i=\canCB(y_i)$ for $i\in\C_1$.
Recall that a perfect matching $\mu$ is a subset of $Q_1$, with boundary value $\bdry\mu\subset\C_1$ (Definition~\ref{d:bdry-value}). For standardised $D$, the description of $\bdry\mu$ simplifies---it is either the set of $i\in\C_1$ such that the boundary arrow $\alpha_i$ is in $\mu$, if $D$ is $\white$-standardised, or the set of $i\in\C_1$ such that the boundary arrow $\beta_i$ is not in $\mu$ if $D$ is $\black$-standardised.

For an arrow $\gamma\in Q_1$, write
\begin{equation}\label{eq:mswts}
 \wtms[\white](\gamma) = \sum_{ j\in \bl'_0(\gamma) } [P_j],
 \qquad
 \wtms[\black](\gamma) = \sum_{ j\in \wh'_0(\gamma) } [P_j],
 \end{equation}
where $\bl'_0(\gamma)$ is the truncated black cycle \eqref{eq:bl'0}, and $\wh'_0(\gamma)$ is the truncated white cycle \eqref{eq:wh'0}.
For a perfect matching $\mu$ of $Q$, we may then define
\begin{equation}
\label{eq:wtms-mu}
\wt[\white](\mu) = \sum_{\substack{ \gamma\in \mu \\ \amut }} \wtms[\white](\gamma),
 \qquad
 \wt[\black](\mu) = \sum_{\substack{ \gamma\in \mu \\ \amut }} \wtms[\black](\gamma).
\end{equation}
The weight $\wtms[\white](\gamma)$ in \eqref{eq:mswts} is the edge weight used by Marsh--Scott \cite{MaSc} to write down a formula for a twisted Pl\"ucker coordinate as a dimer partition function, and $\wtms[\black](\gamma)$ is a natural variant of it.
Strictly speaking, Marsh--Scott define $\wt(e)$, for an edge $e$ of the dual bipartite graph, in terms of face weights. Recall also from Remark~\ref{r:conventions} that in \cite{MaSc} the colours black and white have opposite meanings to here. We will return to the Marsh--Scott formula in Section~\ref{sec:newMS}.

\begin{proposition}
\label{prop:Nmu-class}
Let $D$ be a $\white$-standardised connected Postnikov diagram with dimer algebra $A=A_D$, and let $\mu$ be a perfect matching of $Q(D)$. In $\Kgp(\proj A)$, the class of $N_\mu$ is given by the formula
\begin{equation}\label{eq:Nmu-class-white}
  [N_\mu] = \sum_{\substack{ j\in Q_0 \\ \vmut }} [P_j]
  + \sum_{i\in \bdry\mu} [P_{h\alpha_i}]
  - \wt[\white](\mu).
\end{equation}
If instead $D$ is $\black$-standardised, then
\begin{equation}\label{eq:Nmu-class-black}
[N_\mu] = \sum_{\substack{ j\in Q_0 \\ \vmut }} [P_j]
+ \sum_{i\notin \bdry\mu} [P_{h\beta_i}]
- \wt[\black](\mu).
\end{equation}
\end{proposition}

\begin{proof}
Using Theorem~\ref{thm:match-res}, recalling that $Q_1^\mu=Q_1\setminus\mu$ and noting that the faces $f\in Q_2^\mu$ correspond one-to-one to
the internal arrows $\alpha\in\mu$ in such a way that $hf=t\alpha$, we get
\begin{equation}
\label{eq:proj-res-class}
  [N_\mu] = \sum_{j\in Q_0} [P_j]
  - \sum_{\gamma\notin\mu} [P_{h\gamma}]
  + \sum_{\substack{ \gamma\in\mu \\ \amut }} [P_{t\gamma}].
\end{equation}
When $D$ is standardised, each boundary vertex is the head of a unique boundary arrow and so we can write the first term above as
\begin{equation}
\label{eq:vertex-partition}
\sum_{j\in Q_0}[P_j]=\sum_{\substack{j\in Q_0\\\vmut}}[P_j]+\sum_{\substack{ \gamma\in Q_1 \\ \afro }}[P_{h\gamma}].
\end{equation}

If $D$ is $\white$-standardised, then the boundary arrows are the clockwise arrows $\alpha_i$ for $i\in\C_1$, and those not in the matching $\mu$ are precisely those for which $i\notin\bdry\mu$. Thus substituting \eqref{eq:vertex-partition} into \eqref{eq:proj-res-class} and simplifying yields
\[
  [N_\mu] = \sum_{\substack{ j\in Q_0 \\ \vmut }} [P_j]
  + \sum_{i\in \bdry\mu} [P_{h\alpha_i}]
  - \sum_{\substack{ \gamma\notin\mu \\ \amut }} [P_{h\gamma}]
  + \sum_{\substack{ \gamma\in\mu \\ \amut }} [P_{t\gamma}].
\]
Thus, to prove \eqref{eq:Nmu-class-white}, it remains to prove
\begin{equation}\label{eq:wt-mu}
\wt[\white](\mu) =
 \sum_{\substack{ \gamma\in Q_1 \\ \amut }} \wt_\mu(\gamma),
\end{equation}
where
\begin{equation}\label{eq:wt-alpha}
\wt_\mu(\gamma) = \begin{cases} -[P_{t\gamma}], & \gamma\in\mu, \\ \phantom{-}[P_{h\gamma}], & \gamma\notin\mu. \end{cases}
\end{equation}
To this end, we observe that
\[
  \sum_{\substack{ \gamma\in Q_1 \\ \amut }} \wt_\mu(\gamma)
  = \sum_{\substack{ \gamma\in Q_1 \\ \amut }} [P_{h\gamma}] 
  - \sum_{\substack{ \gamma\in\mu \\ \amut }}\bigl(  [P_{h\gamma}] + [P_{t\gamma}] \bigr).
\]
Every internal arrow is in a unique black cycle and, since $D$ is $\white$-standardised, every arrow in a black cycle is internal. Since each black cycle contains a unique (internal) arrow of $\mu$, we can rewrite the preceding expression as
\begin{align*}
  \sum_{\substack{ \gamma\in Q_1 \\ \amut }} \wt_\mu(\gamma)
  &= \sum_{\substack{ \gamma\in \mu \\ \amut }} \Bigl( \sum_{j\in\bl_0(\gamma)} [P_j] \Bigr)
  - \sum_{\substack{ \gamma\in\mu \\ \amut }}\bigl(  [P_{h\gamma}] + [P_{t\gamma}] \bigr)\\
  &=\sum_{\substack{\gamma\in\mu\\\amut}}\sum_{j\in\bl'_0(\gamma)}[P_j]\\
  &=\sum_{\substack{\gamma\in\mu\\\amut}}\wtms[\white](\gamma),
\end{align*}
which yields \eqref{eq:wt-mu} and hence \eqref{eq:Nmu-class-white}.

On the other hand, if $D$ is $\black$-standardised then the boundary arrows are the anticlockwise arrows $\beta_i$ for $i\in\C_1$, which are not in $\mu$ if and only if $i\in\bdry\mu$, and so
\[
[N_\mu] = \sum_{\substack{ j\in Q_0 \\ \vmut }} [P_j]
+ \sum_{i\notin \bdry\mu} [P_{h\beta_i}]
- \sum_{\substack{ \gamma\notin\mu \\ \amut }} [P_{h\gamma}]
+ \sum_{\substack{ \gamma\in\mu \\ \amut }} [P_{t\gamma}],
\]
and so \eqref{eq:Nmu-class-black} will follow from
\begin{equation}\label{eq:wt-mu-black}
\wt[\black](\mu) =
 \sum_{\substack{ \gamma\in Q_1 \\ \amut }} \wt_\mu(\gamma).
\end{equation}
This is proved similarly to \eqref{eq:wt-mu}, using the fact that white cycles contain only internal arrows, because $D$ is $\black$-standardised.
\end{proof}

\begin{remark}
\label{r:black-vs-white}
We use the notation $\wt[\white]$ and $\wt[\black]$ to emphasise that these functions should be applied to matchings of $\white$-standardised and $\black$-standardised diagrams respectively. It can happen that a $\white$-standardised diagram and a $\black$-standardised diagram have isomorphic dimer algebras (for example, by starting with an arbitrary diagram and then standardising it in each way as in Remark~\ref{r:boundary-convention}), in which case this isomorphism induces a bijection between the two sets of perfect matchings via Proposition~\ref{p:pm-mods}.
However, the value of $\wt[\white]$ on a matching of the $\white$-standardised diagram typically does not agree with the value of $\wt[\black]$ on the corresponding matching of the $\black$-standardised diagram, despite the fact that the right-hand sides of \eqref{eq:wt-mu} and \eqref{eq:wt-mu-black} appear to coincide, since the two quivers have different sets of internal arrows.
\end{remark}

\newcommand{\matlat}{\mathbb{M}}
\newcommand{\matcls}[1][]{\eta^{#1}}
\newcommand{\matdeg}{\operatorname{deg}}

\begin{remark} \label{rem:eta-def}
Consider the reduced cochain complex of the quiver with faces $Q$
\begin{equation}\label{eq:cochain}
 \ZZ \lra{c} \ZZ^{Q_0} \lra{d} \ZZ^{Q_1} \lra{d} \ZZ^{Q_2},
\end{equation}
where the first map is the inclusion of the constant functions and the other two are the coboundary maps.
Note that the faces are all oriented so that second coboundary map $d\colon \ZZ^{Q_1} \to \ZZ^{Q_2}$ 
is simply the face-arrow incidence matrix, with all coefficients $0$ or $1$.
Since $|Q|$ is contractible, this complex (with $0$ added at both ends) is exact.

Let $w\in \ZZ^{Q_2}$ be the function with constant value $1$ on faces and
let $\matlat=d^{-1} \ZZ w$ be the sublattice in $\ZZ^{Q_1}$ of functions with the same sum around every face.
Define $\matdeg\colon \matlat \to\ZZ$ to give the value of that sum, that is,
$df=\matdeg(f)w$, for all $f\in\matlat$.
Then \eqref{eq:cochain} restricts to the exact sequence
\begin{equation}\label{eq:restr-cochain}
 \ZZ \lra{c} \ZZ^{Q_0} \lra{d} \matlat \lra{\matdeg} \ZZ.
\end{equation}

Observe that perfect matchings $\mu$, as in Definition~\ref{d:pm}, may be characterised as 
functions $\mu\in\matlat$ such that $\mu(\gamma)\geq 0$ for all $\gamma\in Q_1$ and $\matdeg(\mu)=1$.
We can then also observe that \eqref{eq:proj-res-class} from (the proof of) 
Proposition~\ref{prop:Nmu-class} can be formulated as 
\begin{equation}
\label{eq:matcls-value}
[N_\mu] = \matcls(\mu),
\end{equation}
where $\matcls\colon\matlat \to \Kgp(\proj A)$ is defined by
\begin{equation}
\label{eq:matcls-grey}
  \matcls(f) = \matdeg(f) \sum_{j\in Q_0} [P_j]
  - \sum_{\gamma\in Q_1} (\matdeg(f)-f(\gamma)) [P_{h\gamma}]
  + \sum_{\substack{ \gamma\in Q_1 \\ \amut }} f(\gamma) [P_{t\gamma}].
\end{equation}
\newcommand{\gambar}{\overline{\gamma}}

Note that the matching lattice $\matlat$ and the map $\matcls$ are 
insensitive to the addition of boundary digons
to the quiver with faces $Q$, as in Remark~\ref{r:boundary-convention}.
More precisely, suppose that $\gamma\in Q_1$ is a boundary arrow and that we add a digon with boundary 
$\gamma\gambar$, where $\gambar$ is opposite to $\gamma$ and becomes the new boundary arrow.
Then we can uniquely extend $f\in\matlat$ from the old to new $Q$ by setting $f(\gambar)=\matdeg(f)-f(\gamma)$.
The formula on the right of \eqref{eq:matcls-grey} gains two new terms, which cancel.

In the special case that $Q$ is $\white$-standardised, 
the derivation of \eqref{eq:Nmu-class-white} from \eqref{eq:proj-res-class} generalises to give
\begin{equation}\label{eq:matcls-white}
 \matcls(f) = 
 \matdeg(f)\sum_{\substack{ j\in Q_0 \\ \vmut }} [P_j]
  + \sum_{i\in \C_1} f(\alpha_i) [P_{h\alpha_i}]
  - \sum_{\substack{ \gamma\in Q_1 \\ \amut }} f(\gamma) \wtms[\white](\gamma) .
\end{equation}
Similarly, when $Q$ is $\black$-standardised, the derivation of \eqref{eq:Nmu-class-black} gives
\begin{equation}\label{eq:matcls-black}
\matcls(f) =
 \matdeg(f)\sum_{\substack{ j\in Q_0 \\ \vmut }} [P_j]
  + \sum_{i\in \C_1} f(\beta_i) [P_{h\beta_i}]
 - \sum_{\substack{ \gamma\in Q_1 \\ \amut }} f(\gamma) \wtms[\black](\gamma) .
\end{equation}
\end{remark}

\newcommand{\intchi}{\chi}

\begin{remark}\label{rem:eta-in-MS}
The map $\matcls$ appears implicitly in \cite{MuSp}.
There $\matlat$ appears as the kernel of the map 
$\ZZ\oplus\ZZ^{Q_1}\to \ZZ^{Q_2}\colon (n,f)\mapsto nw-df$.
They also consider the map 
$\ZZ\oplus\ZZ^{Q_1}\to \ZZ^{Q_0}$ given by
\begin{equation} \label{eq:matcls-MS}
	(n,f)\mapsto  n \sum_{j\in Q_0} (1-B_j)p_j
	+ \sum_{\gamma\in Q_1} f(\gamma) \bigl( p_{h\gamma} + \intchi_\gamma p_{t\gamma} \bigr),
\end{equation}
where $\{p_j:j\in Q_0\}$ is the standard basis of $\ZZ^{Q_0}$, while $B_j=\#\{ \gamma\in Q_1 : h\gamma=j \}$ and $\intchi_\gamma$ is $1$ (resp.\ $0$) when $\gamma$ is internal
(resp.\ on the boundary). Identifying $\ZZ^{Q_0}$ with $\Kgp(\proj{A})$ using $p_j\mapsto[P_j]$ and comparing to \eqref{eq:matcls-grey}, we see that $\matcls$ is obtained by restricting this second map to $\matlat$.

For comparison, in \cite[Lemma~5.1]{MuSp} these two maps are combined into a single map 
$X\colon\ZZ\oplus\ZZ^{Q_1}\to \ZZ^{Q_0}\oplus\ZZ^{Q_2}$
and described in terms of the bipartite graph dual to $Q$.
However, $\matcls$ itself appears more explicitly in the proof of \cite[Prop~5.5]{MuSp}.
The facts that $X$ and $\matcls$ are isomorphisms are the content of these two results in \cite{MuSp}.
\end{remark}

\newcommand{\fd}{\operatorname{fd}}
\newcommand{\exchmat}{\beta}

With our interpretation of $\matcls$ in terms of projective resolution, 
we can give a more conceptual proof of the fact that it is an isomorphism.

\begin{lemma} \label{lem:matcls-isom}
The map $\matcls\colon\matlat \to \Kgp(\proj A)$, defined in \eqref{eq:matcls-grey}, is an isomorphism.
\end{lemma}

\begin{proof}
By Corollary~\ref{c:proj-pm-mods},
every indecomposable projective $P_j=Ae_j$ is (isomorphic to) some perfect matching module $N_{\mu_j}$. 
Hence $[P_j]=\matcls(\mu_j)$ by Proposition~\ref{prop:Nmu-class} and so,
since $\{[P_j] : j\in Q_0\}$ is a basis of $\Kgp(\proj{A})$, we see that $\matcls$ is surjective.
However, as the sequence \eqref{eq:restr-cochain} is exact, $\matlat$ has the same rank (namely $|Q_0|$) as
$\Kgp(\proj{A})$ and so $\matcls$ is an isomorphism, as required.
\end{proof}

\begin{corollary} \label{cor:rank1-rigid}
Let $M_1,M_2\in\CM(A)$. If $\rk(M_1)=\rk(M_2)=1$ and $[M_1]=[M_2]$ in $\Kgp(\proj A)$, then $M_1\isom M_2$.
\end{corollary}

\begin{proof}
As $\rk(M_i)=1$, Proposition~\ref{p:pm-mods} implies that $M_i\isom N_{\mu_i}$, 
for perfect matchings $\mu_i\in\matlat$, and so $[M_i]=\matcls(\mu_i)$.
Since $\matcls$ is injective, the fact that $[M_1]=[M_2]$ implies that $\mu_1=\mu_2$ 
and thus $N_{\mu_1}\isom N_{\mu_2}$, as required.
\end{proof}

One consequence of Corollary~\ref{cor:rank1-rigid} is that, to identify the matching $\mu$
for which $P_j\isom N_\mu$, as in Corollary~\ref{c:proj-pm-mods}, it suffices to show that $[N_\mu]=[P_j]$.
We do this in the next section, using the calculation \eqref{eq:Nmu-class-white} of $[N_\mu]$.

Lemma~\ref{lem:matcls-isom} has a further consequence for the `cluster ensemble sequence'
\begin{equation}\label{eq:clu-ens}
 \ZZ \lra{c} \ZZ^{Q_0} \lra{\exchmat} \ZZ^{Q_0} \lra{\rk} \ZZ,
\end{equation}
where we identify the first  $\ZZ^{Q_0}$ with $\Kgp(\fd A)$ via its basis of simples $[S_i]$, for $i\in Q_0$,
and the second $\ZZ^{Q_0}$ with $\Kgp(\proj A)$ via its basis of projectives $[P_i]$, for $i\in Q_0$.
As before, $c$ is the inclusion of constant functions, while $\rk [P_i] =1$ for all $i$.
The map $\exchmat$ corresponds to projective resolution, but can just be described combinatorially as
\begin{equation}\label{eq:beta}
 \exchmat [S_i] = [P_i] - \sum_{a:ta=i} [P_{ha}]  + \sum_{a:ha=i} \intchi_a [P_{ta}]  - \intchi_i [P_i]
\end{equation}
where $\intchi_a$ (resp.\ $\intchi_i$) is 1 or 0 depending on whether the arrow $a\in Q_1$ 
(resp.\ vertex $i\in Q_0$) is internal or on the boundary.
Note that $\exchmat$ is an extension of the exchange matrix (or its negative, depending on the convention used), when $Q$ is interpreted as the ice quiver of a cluster algebra seed as in \cite{GL}.

\begin{proposition}
\label{p:clu-ens-exact}
The cluster ensemble sequence \eqref{eq:clu-ens} is exact.
\end{proposition}
\begin{proof}
Two straightforward (but not entirely trivial) calculations, which we describe below, show that the map $\matcls$ fits into the following commutative diagram.
\begin{equation}
\label{eq:clu-ens-comp}
\begin{tikzcd}
\ZZ\arrow{r}{c}\arrow[equal]{d}&
\ZZ^{Q_0}\arrow{r}{d}\arrow[equal]{d}&
\matlat\arrow{r}{\matdeg}\arrow{d}{\matcls}&
\ZZ\arrow[equal]{d}
\\
\ZZ\arrow{r}{c}&
\ZZ^{Q_0}\arrow{r}{\beta}&
\ZZ^{Q_0}\arrow{r}{\rk}&
\ZZ
\end{tikzcd}
\end{equation}
Since $\eta$ is an isomorphism by Lemma~\ref{lem:matcls-isom}, this diagram describes an isomorphism of complexes from the exact cochain complex \eqref{eq:cochain} to the cluster ensemble sequence \eqref{eq:clu-ens}, which is therefore also exact.

The first calculation is that $\matcls (d[S_i]) = \exchmat[S_i]$. 
We start by noting that $\matdeg(d[S_i]) = 0$, so we calculate from \eqref{eq:matcls-value} that
\[
 \matcls (d[S_i]) = \sum_{a:ha=i} \bigl([P_{ha}] + \intchi_a [P_{ta}]  \bigr)
  - \sum_{a:ta=i} \bigl([P_{ha}] + \intchi_a [P_{ta}] \bigr),
\]
using the indicator function $\chi_a$ to convert a sum over internal arrows into a sum over all arrows.
Thus, comparing to \eqref{eq:beta}, we need to prove that
\[
  \sum_{a:ha=i} [P_{ha}] 
  - \sum_{a:ta=i} \intchi_a [P_{ta}]
  = (1-\intchi_i)[P_i].
\]
When $i$ is internal, all incident arrows $a$ are internal 
and there are as many with $ha=i$ as with $ta=i$.
When $i$ is on the boundary, there is one more arrow $a$ with $ha=i$ than with $ta=i$,
when we ignore any boundary arrows of the latter type.

The second calculation is that $\rk(\matcls(f))=\matdeg(f)$. 
For this we observe, from \eqref{eq:matcls-value} and the fact that $\rk[P_i]=1$ for all $i\in Q_0$, that 
\[
\rk(\matcls(f)) = \matdeg(f) \bigl( |Q_0| - |Q_1| \bigr)+ \sum_{a\in Q_1} f(a)(1+\intchi_a).
\]
The sum on the right-hand side is equal to $\matdeg(f) |Q_2|$, as $f\in\matlat$.
But $|Q_0| - |Q_1|+ |Q_2|=1$, that is, the Euler characteristic of the disc.
\end{proof}

\begin{remark}
The commutativity of \eqref{eq:clu-ens-comp} in fact shows that the exactness of \eqref{eq:clu-ens} is equivalent to $\eta$ being an isomorphism; one direction is as given in the proof of Proposition~\ref{p:clu-ens-exact}, while the converse follows from the five lemma.

Indeed, if one calculates the maps $\matcls$ and $\exchmat$ combinatorially in 
the inconsistent example in Figure~\ref{f:inconsistent}, 
then one finds that $\matcls$ is not an isomorphism and the sequence \eqref{eq:clu-ens} is not exact.
\end{remark}

\section{Muller--Speyer matchings} \label{sec:MuSp}

\newcommand{\MSmap}{\mathfrak{M}}
\newcommand{\DSwedge}{\mathfrak{DS}}
\newcommand{\alteta}{\hat{\eta}}

Let $D$ be a Postnikov diagram, with $Q=Q(D)$ and $A=A_D$, and consider the map $\alteta\colon\matlat\to\ZZ^{Q_0}$ given by
\begin{equation}
\label{eq:matcls-alt}
  \alteta(f) = \matdeg(f) \sum_{j\in Q_0}p_j
- \sum_{\gamma\in Q_1} (\matdeg(f)-f(\gamma))p_{h\gamma}
+ \sum_{\substack{ \gamma\in Q_1 \\ \amut }} f(\gamma)p_{t\gamma}
\end{equation}
where $\{p_j:j\in Q_0\}$ is the standard basis of $\ZZ^{Q_0}$. When $D$ is connected, we use the isomorphism $\ZZ^{Q_0}\to\K_0(\proj{A})\colon p_j\mapsto[P_j]$ as a (silent) identification, so that the map $\alteta$ in \eqref{eq:matcls-alt} is identified with the map $\matcls$ in \eqref{eq:matcls-grey}. In this case, Lemma~\ref{lem:matcls-isom} shows that $\alteta$ is an isomorphism.

In this section we will want to evaluate $\alteta$ on perfect matchings $\mu\in\matlat$, for which we have (cf.\ \eqref{eq:proj-res-class})
\begin{equation}
	\label{eq:matcls-mu}
	\alteta(\mu) = \sum_{i\in Q_0}p_i
	- \sum_{\gamma\notin\mu}p_{h\gamma}
	+ \sum_{\substack{ \gamma\in\mu \\ \amut }}p_{t\gamma}.
\end{equation}

In \cite[\S5.2]{MuSp}, without requiring $D$ to be connected, Muller--Speyer defined a special matching $\MSmat_j$, associated to 
any vertex $j\in Q_0$
(or more strictly to a face of the dual plabic graph $\plabic(D)$), by 
\begin{equation}
\label{eq:MSmat}
 \alpha \in \MSmat_j \iff j \in \DSwedge(\alpha)
\end{equation}
where $\DSwedge(\alpha)$ is the downstream wedge of the arrow $\alpha\in Q_1$, 
as illustrated in Figure~\ref{f:ds-wedge}(a).
One of their results \cite[Cor.~5.6]{MuSp} is that $\{ \MSmat_j : j\in Q_0\}$ is a basis of $\matlat$,
which can be formulated as saying that
\begin{equation}\label{eq:MSmap}
 \MSmap \colon \ZZ^{Q_0} \to \matlat \colon p_j \mapsto \MSmat_j
\end{equation}
is an isomorphism.
To make the comparison, note that  \cite[\S5.3]{MuSp} actually uses $-\MSmap$ to 
describe a monomial map between the tori $\Gm^{|Q_1|}/\Gm^{|Q_2|-1}$ and $\Gm^{|Q_0|}$ whose character lattices are $\matlat$ and $\ZZ^{Q_0}$ respectively. 
Furthermore these tori are described in terms of the bipartite graph dual to the quiver with faces~$Q$.

Our main goal in this section is to show that, in fact, $\alteta=\MSmap^{-1}$.
We do this by showing that $\alteta \compo \MSmap = \id$,
i.e.\ $\alteta(\MSmat_j)=p_j$ for all $j\in Q_0$.
In the case that $D$ is connected, this means that $\matcls(\MSmat_j)=[P_j]$ and so, since we also have $\matcls(\MSmat_j)=[N_{\MSmat_j}]$ by \eqref{eq:matcls-value}, we may conclude via Corollary~\ref{cor:rank1-rigid} that $P_j\isom N_{\MSmat_j}$.

An example of a Muller--Speyer matching $\MSmat_j$ is shown in Figure~\ref{f:ms-example} 
and one can verify in this case that $N_{\MSmat_j}\isom P_j$.

\begin{figure} [h]
\begin{tikzpicture}[scale=2.5,yscale=-1]
\foreach \n/\m/\a in {1/4/0, 2/3/0, 3/2/5, 4/1/10, 5/7/0, 6/6/-3, 7/5/0}
{ \coordinate (b\n) at (\bstart-\seventh*\n+\a:1.0);}

\foreach \n/\m in {8/1, 9/2, 10/3, 11/4, 14/5, 15/6, 16/7}
  {\coordinate (b\n) at ($0.65*(b\m)$);}

\coordinate (b13) at ($(b15) - (b16) + (b8)$);
\coordinate (b12) at ($(b14) - (b15) + (b13)$);

\foreach \n/\x\y in {13/-0.03/-0.03, 12/-0.22/0.0, 14/-0.07/-0.03, 11/0.05/0.02, 16/-0.02/0.02}
  {\coordinate (b\n)  at ($(b\n) + (\x,\y)$); } 

\foreach \e/\f/\t in {2/9/0.5, 4/11/0.5, 5/14/0.5, 7/16/0.5, 
 8/9/0.5, 9/10/0.5, 10/11/0.5,11/12/0.5, 12/13/0.45, 8/13/0.6, 
 14/15/0.5, 15/16/0.6, 12/14/0.45, 13/15/0.4, 8/16/0.6}
{\coordinate (a\e-\f) at ($(b\e) ! \t ! (b\f)$); }

\draw [strand] plot[smooth]
coordinates {(b1) (a8-16) (a15-16) (b6)}
[postaction=decorate, decoration={markings,
 mark= at position 0.2 with \strarrow,
 mark= at position 0.5 with \strarrow, 
 mark= at position 0.8 with \strarrow }];
 
\draw [strand] plot[smooth]
coordinates {(b6) (a14-15) (a12-14)(a11-12) (a10-11) (b3)}
[postaction=decorate, decoration={markings,
 mark= at position 0.15 with \strarrow, mark= at position 0.35 with \strarrow,
 mark= at position 0.53 with \strarrow, mark= at position 0.7 with \strarrow,
 mark= at position 0.87 with \strarrow }];
 
\draw [strand] plot[smooth]
coordinates {(b3) (a9-10) (a8-9) (b1)}
[postaction=decorate, decoration={markings,
 mark= at position 0.2 with \strarrow,
 mark= at position 0.5 with \strarrow, 
 mark= at position 0.8 with \strarrow }];

\draw [strand] plot[smooth]
coordinates {(b2) (a9-10) (a10-11) (b4)}
 [postaction=decorate, decoration={markings,
 mark= at position 0.2 with \strarrow,
 mark= at position 0.5 with \strarrow, 
 mark= at position 0.8 with \strarrow }];

\draw [strand] plot[smooth]
coordinates {(b4) (a11-12) (a12-13) (a13-15) (a15-16) (b7)}
[postaction=decorate, decoration={markings,
 mark= at position 0.15 with \strarrow, mark= at position 0.35 with \strarrow,
 mark= at position 0.55 with \strarrow, mark= at position 0.7 with \strarrow,
 mark= at position 0.87 with \strarrow }];

\draw [strand] plot[smooth]
coordinates {(b7) (a8-16) (a8-13) (a12-13) (a12-14) (b5)}
[postaction=decorate, decoration={markings,
 mark= at position 0.15 with \strarrow, mark= at position 0.315 with \strarrow,
 mark= at position 0.5 with \strarrow, mark= at position 0.7 with \strarrow,
 mark= at position 0.88 with \strarrow }];

\draw [strand] plot[smooth]
coordinates {(b5) (a14-15) (a13-15) (a8-13) (a8-9) (b2)}
[postaction=decorate, decoration={markings,
 mark= at position 0.13 with \strarrow, mark= at position 0.33 with \strarrow,
 mark= at position 0.5 with \strarrow, mark= at position 0.7 with \strarrow,
 mark= at position 0.88 with \strarrow }];

\foreach \n/\m\a in {1/124/0, 2/234/-1, 3/345/1, 4/456/10, 5/256/5, 6/267/0, 7/127/0}
{ \draw [\quivcolor] (\bstart+\seventh/2-\seventh*\n+\a:1) node (q\m) {\tiny $\bullet$}; }
\foreach \m/\a/\r in {247/\bstart/0.42, 245/10/0.25, 257/210/0.36}
{ \draw [\quivcolor] (\a:\r) node (q\m) {\tiny $\bullet$}; }
\draw [\quivcolor, thick] (q247) circle (0.05);
\foreach \t/\h/\a in {234/124/-20, 234/345/19, 456/345/-16, 456/256/22, 127/267/-20,
 127/124/19, 247/127/11, 245/234/18, 247/245/34, 245/257/15,
 267/257/6, 257/256/0, 245/456/-16}
{ \draw [quivarrow]  (q\t) edge [bend left=\a] (q\h); }

\foreach \t/\h/\a in {257/247/14, 124/247/13, 345/245/0, 256/267/21, 256/245/-9}
{ \draw [quivarrow, ultra thick, violet]  (q\t) edge [bend left=\a] (q\h); }
\end{tikzpicture}
\caption{The Muller--Speyer matching $\MSmat_j$ for $j$ the circled vertex.}
\label{f:ms-example}
\end{figure}
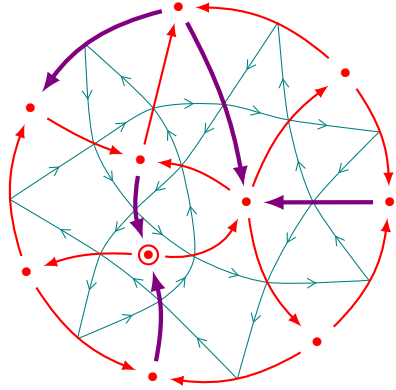

Interestingly, Muller--Speyer \cite[proof of Prop.~5.5]{MuSp} also show that $\alteta=\MSmap^{-1}$,
but by instead proving that $\MSmap \compo \alteta = \id$. They deduce this identity by defining larger matrices 
\[ X\colon \ZZ\oplus \ZZ^{Q_1}\to \ZZ^{Q_0}\oplus \ZZ^{Q_2}
\quad\text{and}\quad 
X'\colon \ZZ^{Q_0}\oplus \ZZ^{Q_2}\to \ZZ\oplus \ZZ^{Q_1},
\] 
with $\MSmap$ a component of $X'$,
and showing \cite[Lemma~5.1]{MuSp} that $X'\compo X=\id$ (cf.\ Remark~\ref{rem:eta-in-MS}).
It then follows that $X\compo X'=\id$, and one component of this identity implies that 
each $\MSmat_j$ is indeed a matching.

\begin{figure} [h]
\pgfmathsetmacro{\brad}{1.0}
\pgfmathsetmacro{\qang}{30}
\pgfmathsetmacro{\alpA}{\bstart-\seventh}
\pgfmathsetmacro{\alpB}{\bstart-3*\seventh+5}
\pgfmathsetmacro{\alpC}{\bstart-4*\seventh-5}
\pgfmathsetmacro{\alpD}{\bstart-6*\seventh}
\begin{tikzpicture}[scale=2]
\begin{scope}
\foreach \n/\a in {1/\alpA, 3/\alpB, 4/\alpC, 6/\alpD}
  \coordinate (b\n) at (\a:\brad);
\foreach \n/\a/\r in {1/\qang/0.15, 2/180+\qang/0.35}
  \coordinate (q\n) at (\a:\r);
\coordinate (m-ab) at ($(q1) ! 0.6 ! (q2)$); 

\draw [strand, white, fill=teal!10] (b1) to [out= -130, in=\qang+45] (m-ab) 
 to [out=\qang-45, in= 180] (b3) arc (\alpB:\alpA:\brad);

\draw [strand] (b6) to [out=0, in=180+\qang-45] (m-ab) to [out=\qang-45, in= 180] (b3)
[postaction=decorate, decoration={markings,
 mark= at position 0.25 with \strarrow,  mark= at position 0.75 with \strarrow }];
\draw [strand] (b4) to [out=120, in=180+\qang+45] (m-ab) to [out=\qang+45, in= -130] (b1)
[postaction=decorate, decoration={markings,
 mark= at position 0.25 with \strarrow,  mark= at position 0.75 with \strarrow }];

\draw [quivarrow] ($(q2) ! 0.15 ! (q1)$) to ($(q2) ! 0.85 ! (q1)$) ;
\foreach \n in {1,2}
 \draw [\quivcolor] (q\n) node {\tiny $\bullet$};

\draw [boundary] (0,0) circle (\brad);
\end{scope}
\begin{scope}[shift={(3,0)}]
\pgfmathsetmacro{\qrad}{0.45}
\pgfmathsetmacro{\sang}{20}
\pgfmathsetmacro{\alpA}{\bstart-\seventh}
\pgfmathsetmacro{\alpB}{\bstart-3*\seventh+5}
\pgfmathsetmacro{\alpC}{\bstart-5*\seventh-5}
\pgfmathsetmacro{\alpD}{\bstart-7*\seventh}
\foreach \n/\a in {1/\alpA, 2/\alpB, 3/\alpC, 4/\alpD}
  \coordinate (b\n) at (\a:\brad);
 \coordinate (q1) at (\qang:0.15);
\foreach \n/\m/\a in {2/1/90, 3/2/180, 4/3/270}
 \coordinate (q\n) at ($(q\m)+(\qang+\a:\qrad)$);
\foreach \n/\m/\r in {1/2/0.4, 2/3/0.4, 3/4/0.4, 4/1/0.4}
  \coordinate (a\n\m) at ($(q\n) ! \r ! (q\m)$); 
\foreach \n/\a in {12/-30, 23/60, 34/150, 41/-120}
 \coordinate (c\n) at ($(a\n)+(\a+\qang:0.15)$);

\draw [strand, white, fill=teal!20] (b1) arc (\alpA:\alpB:\brad)
(b2) to [in=\qang-45,  out=180+\alpB+\sang]
(a41) to [out=\qang+45, in=\qang-135] (a12) to [out=\qang+45,  in=180+\alpA+\sang] (b1);

\draw [strand, white, fill=teal!10] (b2) arc (\alpB:\alpC:\brad)
(b3) to  [in=\qang-135,  out=180+\alpC+\sang] 
(a34) to [out=\qang-45, in=\qang+135] (a41) to [out=\qang-45,  in=180+\alpB+\sang] (b2);

\draw [strand, white, fill=teal!20] (b3) arc (\alpC:\alpD:\brad)
(b4) to [in=\qang+135,  out=180+\alpD-\sang]
(a23) to [out=\qang-135, in=\qang+45] (a34) to [out=\qang-135,  in=180+\alpC+\sang] (b3);

\draw [strand, white, fill=teal!10] (b4) arc (\alpD:\alpA-360:\brad)
(b1) to [in=\qang+45,  out=180+\alpA+\sang] 
(a12) to [out=\qang+135, in=\qang-45] (a23) to [out=\qang+135,  in=180+\alpD-\sang] (b4);

\draw [strand] (c41) to [out=\qang+60, in=\qang-135] (a41) to [out=\qang+45, in=\qang-135] (a12) 
 to [out=\qang+45,  in=180+\alpA+\sang] (b1)
[postaction=decorate, decoration={markings, 
 mark= at position 0.32 with \strarrow, mark= at position 0.75 with \strarrow }];
\draw [strand] (c12) to [out=\qang+150, in=\qang-45] (a12) to [out=\qang+135, in=\qang-45] (a23) 
 to [out=\qang+135,  in=180+\alpD-\sang] (b4)
[postaction=decorate, decoration={markings, 
 mark= at position 0.35 with \strarrow, mark= at position 0.75 with \strarrow }];
\draw [strand] (c23) to [out=\qang+240, in=\qang+45] (a23) to [out=\qang-135, in=\qang+45] (a34) 
 to [out=\qang-135,  in=180+\alpC+\sang] (b3)
[postaction=decorate, decoration={markings, 
 mark= at position 0.32 with \strarrow, mark= at position 0.75 with \strarrow }];
\draw [strand] (c34) to [out=\qang-30, in=\qang+135] (a34) to [out=\qang-45, in=\qang+135] (a41) 
 to [out=\qang-45,  in=180+\alpB+\sang] (b2)
[postaction=decorate, decoration={markings, 
 mark= at position 0.27 with \strarrow, mark= at position 0.75 with \strarrow }];

\foreach \n/\m/\r in {1/2/0.5, 2/3/0.5, 3/4/0.5, 4/1/0.5}
  \draw [quivarrow] ($(q\n) ! 0.15 ! (q\m)$) to ($(q\n) ! 0.85 ! (q\m)$);
\foreach \n in {1,2,3,4}
 \draw [\quivcolor] (q\n) node {\tiny $\bullet$};

\draw [boundary] (0,0) circle (\brad);
\end{scope}
\end{tikzpicture}
\caption{(a) downstream wedge of an arrow, (b) wedges round a face}
\label{f:ds-wedge}
\end{figure}
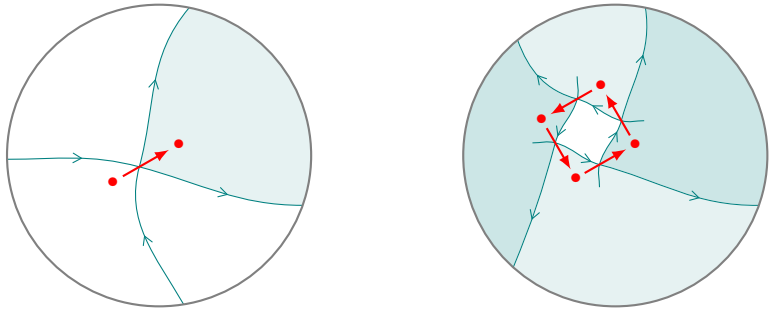

On the other hand, we will show directly that $\MSmat_j$ is a matching, 
by observing that the downstream wedges of the 
arrows in a face of $Q$ partition the vertices $Q_0$, 
as illustrated in Figure~\ref{f:ds-wedge}(b).
This is a special case of a more general wedge-covering property that we now explain.

Recall that, by the Jordan curve theorem, the complement of a simple closed curve in the disc has one component not intersecting the boundary, which we call the \emph{inside} (the other components being the \emph{outside}). Furthermore, the curve is the boundary of its inside.

\newcommand{\strapol}{\mathcal{P}}

\begin{definition} \label{d:strand-poly}
In a Postnikov diagram $D$, a \emph{strand polygon} $\strapol$ is an oriented simple closed curve consisting of a collection of contiguous segments, either of the boundary of the disc or of strands oriented consistently with $\strapol$. We further require that $\strapol$ turns towards its inside at each \emph{vertex}, 
i.e.\ the point at which one segment ends and the next begins.
See Figure~\ref{f:strandpoly} for examples and Figure~\ref{f:notstrandpoly} for non-examples.
Note that edges of the strand polygon may cross other strands in the diagram.

Each vertex $v$ of the polygon has a \emph{tendril} $f_v$, defined as follows. If the preceding edge $e_v$ is a strand segment, then $f_v$ is the continuation of the strand from $e_v$ until it ends on the boundary of the disc. If $e_v$ is a boundary segment, then $f_v$ is the point $v$ by definition.

Each vertex $v$ of the polygon determines a \emph{(downstream) wedge}, which is the subset of the disc bounded by the tendril $f_v$, the edge $e_w$ following $v$, the tendril $f_w$, and the boundary segment from the endpoint of $f_v$ to that of $f_w$ in the direction (clockwise or anticlockwise) of the orientation of $\strapol$.
This construction is illustrated in Figure~\ref{fig:wedge}.
\end{definition}

\begin{figure} [h]
\begin{tikzpicture}[scale=1.5]
\pgfmathsetmacro{\eps}{0.2}
\begin{scope}[shift={(0,0)}]
\pgfmathsetmacro{\brad}{0.7}
\foreach \n/\a in {1/20, 2/90, 3/170, 4/240, 5/330}
  \coordinate (b\n) at (\a:\brad);
\foreach \t/\h/\outang/\inang in {1/2/170/-70, 2/3/240/30, 3/4/-40/100, 4/5/30/180,
 5/1/110/-110}
{ \draw [strand, thick, densely dotted] (b\t)--++(\outang+180:\eps);
  \draw [strand, thick, densely dotted] (b\h)--++(\inang+180:\eps);
   \draw [strand, thick] (b\t) to [out=\outang, in=\inang] (b\h)
  [postaction=decorate, decoration={markings, 
 mark= at position 0.55 with \strarrow}]; 
} 
\end{scope}
\begin{scope}[shift={(3,0)}]
\pgfmathsetmacro{\brad}{0.5}
\pgfmathsetmacro{\bbrad}{0.9}
\foreach \n/\a/\rad in {1/20/\bbrad, 2/80/\brad, 3/180/\brad, 4/240/\brad, 5/-60/\bbrad}
  \coordinate (b\n) at (\a:\rad);
  
\foreach \t/\h/\outang/\inang in {2/3/240/30, 3/4/-40/100}
{ \draw [strand, thick, densely dotted] (b\t)--++(\outang+180:\eps);
  \draw [strand, thick, densely dotted] (b\h)--++(\inang+180:\eps);
   \draw [strand, thick] (b\h) to [out=\inang, in=\outang] (b\t)
  [postaction=decorate, decoration={markings, 
 mark= at position 0.55 with \strarrow}]; 
} 
\foreach \t/\h/\outang/\inang in {4/5/10/140}
{ \draw [strand, thick, densely dotted] (b\t)--++(\outang+180:\eps);
   \draw [strand, thick] (b\h) to [out=\inang, in=\outang] (b\t)
  [postaction=decorate, decoration={markings, 
 mark= at position 0.55 with \strarrow}]; 
} 
\foreach \t/\h/\outang/\inang in {1/2/190/-50}
{ \draw [strand, thick, densely dotted] (b\h)--++(\inang+180:\eps);
   \draw [strand, thick] (b\h) to [out=\inang, in=\outang] (b\t)
  [postaction=decorate, decoration={markings, 
 mark= at position 0.55 with \strarrow}]; 
} 

\draw [boundary, thick, densely dotted] (b1) arc (20:35:\bbrad); 
\draw [boundary, thick, densely dotted] (b5) arc (-60:-75:\bbrad);
\draw [boundary] (b1) arc (20:-60:\bbrad);

\end{scope}
\end{tikzpicture}
\caption{Strand polygons}
\label{f:strandpoly}
\end{figure}
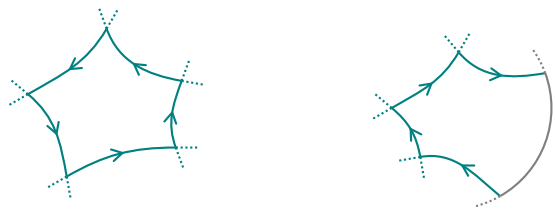

\begin{figure} [h]
\begin{tikzpicture}[scale=1.5]
\pgfmathsetmacro{\brad}{0.7}
\pgfmathsetmacro{\eps}{0.2}
\begin{scope}[shift={(0,0)}]
\foreach \n/\a in {1/20, 2/90, 3/170, 4/240, 5/330}
  \coordinate (b\n) at (\a:\brad);
\foreach \t/\h/\outang/\inang in {2/1/-70/170, 3/2/30/240, 3/4/-40/100, 4/5/30/180, 5/1/110/-110}
{ \draw [strand, thick, densely dotted] (b\t)--++(\outang+180:\eps);
  \draw [strand, thick, densely dotted] (b\h)--++(\inang+180:\eps);
   \draw [strand, thick] (b\t) to [out=\outang, in=\inang] (b\h)
  [postaction=decorate, decoration={markings, 
 mark= at position 0.55 with \strarrow}]; 
} 
\end{scope}
\begin{scope}[shift={(3,0)}]
\foreach \n/\a in {1/20, 3/170, 4/240, 5/330}
  \coordinate (b\n) at (\a:\brad);
  \coordinate (b2) at (90:0.3*\brad);
\foreach \t/\h/\outang/\inang in {1/2/170/50, 2/3/140/30, 3/4/-40/100, 4/5/30/180,
 5/1/110/-110}
{ \draw [strand, thick, densely dotted] (b\t)--++(\outang+180:\eps);
  \draw [strand, thick, densely dotted] (b\h)--++(\inang+180:\eps);
   \draw [strand, thick] (b\t) to [out=\outang, in=\inang] (b\h)
  [postaction=decorate, decoration={markings, 
 mark= at position 0.55 with \strarrow}]; 
} 
\end{scope}
\end{tikzpicture}
\caption{Not strand polygons}
\label{f:notstrandpoly}
\end{figure}
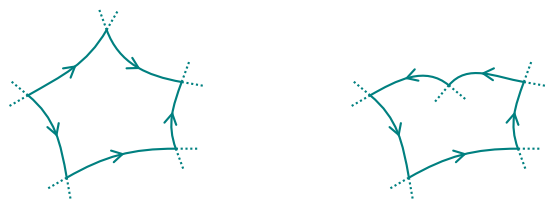
\begin{figure}[h]
\begin{tikzpicture}[scale=2.7]
\pgfmathsetmacro{\brad}{1.0}
\pgfmathsetmacro{\qang}{30}
\pgfmathsetmacro{\qrad}{0.45}
\pgfmathsetmacro{\sang}{20}
\pgfmathsetmacro{\alpA}{\bstart-\seventh}
\pgfmathsetmacro{\alpB}{\bstart-3*\seventh+5}
\pgfmathsetmacro{\alpC}{\bstart-5*\seventh-5}
\pgfmathsetmacro{\alpD}{\bstart-7*\seventh}
\foreach \n/\a in {1/\alpA, 2/\alpB, 3/\alpC, 4/\alpD}
  \coordinate (b\n) at (\a:\brad);
\coordinate (q1) at (\qang:0.15);
\foreach \n/\m/\a in {2/1/90, 3/2/180, 4/3/270}
\coordinate (q\n) at ($(q\m)+(\qang+\a:\qrad)$);
\foreach \n/\m/\r in {1/2/0.4, 2/3/0.4, 3/4/0.4, 4/1/0.4}
\coordinate (a\n\m) at ($(q\n) ! \r ! (q\m)$);
\foreach \n/\a in {12/-30, 23/60, 34/150, 41/-120}
\coordinate (c\n) at ($(a\n)+(\a+\qang:0.15)$);

\draw [strand, white, fill=teal!20] (b1) arc (\alpA:\alpB:\brad)
(b2) to [in=\qang-45,  out=180+\alpB+\sang]
(a41) to [out=\qang+45, in=\qang-135] (a12) to [out=\qang+45,  in=180+\alpA+\sang] (b1);

\draw [strand] (c41) to [out=\qang+60, in=\qang-135] (a41) to [out=\qang+45, in=\qang-135] (a12)
to [out=\qang+45,  in=180+\alpA+\sang] (b1)
[postaction=decorate, decoration={markings,
        mark= at position 0.32 with \strarrow, mark= at position 0.75 with \strarrow }];
\draw [strand] (c12) to [out=\qang+150, in=\qang-45] (a12) to [out=\qang+135, in=\qang-45] (a23)
to [out=\qang+135,  in=180+\alpD-\sang] (b4)
[postaction=decorate, decoration={markings,
        mark= at position 0.35 with \strarrow, mark= at position 0.75 with \strarrow }];
\draw [strand] (c23) to [out=\qang+240, in=\qang+45] (a23) to [out=\qang-135, in=\qang+45] (a34)
to [out=\qang-135,  in=180+\alpC+\sang] (b3)
[postaction=decorate, decoration={markings,
        mark= at position 0.32 with \strarrow, mark= at position 0.75 with \strarrow }];
\draw [strand] (c34) to [out=\qang-30, in=\qang+135] (a34) to [out=\qang-45, in=\qang+135] (a41)
to [out=\qang-45,  in=180+\alpB+\sang] (b2)
[postaction=decorate, decoration={markings,
        mark= at position 0.27 with \strarrow, mark= at position 0.75 with \strarrow }];

\draw [boundary] (0,0) circle (\brad);
\newcommand{\size}{\small}

\draw (a41) circle(0.015) [fill=black] node [below right = -1.8pt] {\size $v$};
\draw( a12) circle(0.015) [fill=black] node [above right = -2pt] {\size $w$};

\draw (45:0.1) node {\size $e_w$};
\draw (67:0.68) node {\size $f_w$};
\draw (-35:0.62) node {\size $f_v$};
\draw (-163:0.25) node {\size $e_v$};

\end{tikzpicture}
\caption{A strand polygon and its tendrils. The wedge at $v$ is shaded.}
\label{fig:wedge}
\end{figure}
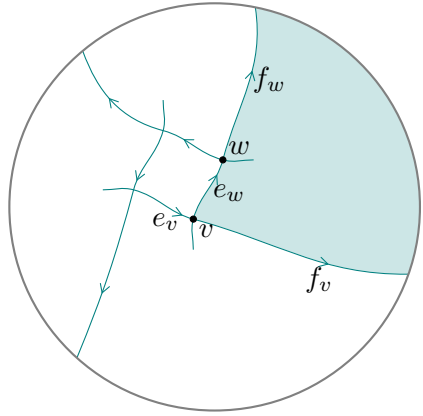

Note that whenever $v$ is on the boundary, $f_v$ is just the single point $v$ (by definition or by construction). The turning condition implies that non-trivial tendrils start by moving into the outside of $\strapol$.

Each wedge is well-defined (and wedge-shaped) because of condition (b2) in the definition of a Postnikov diagram, which implies that the strand segments $f_v$ and $e_w\cup f_w$ intersect only at $v$. 
If $v$ is on the boundary, then its wedge is trivial if the next edge of $\strapol$ is a boundary segment, and 
otherwise is just one side of the strand on which the next edge lies.

Note that the boundary of an oriented region of $D$, corresponding to a face $F$ of  $Q$,
is an example of a strand polygon. The inside of the polygon is the oriented region and
the downstream wedges of its vertices are those of the arrows in $\bdry F$
(see Figure~\ref{f:ds-wedge}(b)).

\begin{lemma}
\label{l:no-return}
Let $\strapol$ be a strand polygon in $D$. Then the tendrils of $\strapol$ meet $\strapol$ only at their starting vertices.
\end{lemma}
\begin{proof}
Fix a vertex $v$ of $\strapol$ and let $e_v$, from $u$ to $v$, and $e_w$, from $v$ to $w$, be the edges incident with $v$. Let $s_v$ be the strand containing $e_v$ and $f_v$. We may assume that $f_v$ is non-trivial, i.e.\ that $v$ is not on the boundary.

We first observe that $f_v$ only intersects $e_v$ and $e_w$ at its start $v$. A second intersection of $f_v$ with $e_v$ would imply either a self-intersection of $s_v$, contradicting (b1), or that $s_v$ is a closed loop, or that $u$ lies on the boundary and $s_v$ is a lollipop with both endpoints at $u$. But lollipops have no crossings by Proposition~\ref{p:no-big-lollipops}, and $s_v$ has a crossing at $v$. Similarly, $f_v$ cannot have a second intersection with $e_w$, since this would either contradict condition (b2) because $s_v$ already intersects $e_w$ at $v$, or imply that $s_v$ contributes both edges $e_v$ and $e_w$, and so has a self-intersection at $v$. 

If $f_v$ meets $\strapol$ again, let $\alpha$ be the piece of $f_v$ from $v$ to its second meeting with $\strapol$ (which may not be a vertex of $\strapol$) and let $\gamma$ be the path in $\strapol$ completing $\alpha$ to a simple closed curve in such a way that the inside $R$ of $\alpha\cup\gamma$ is entirely outside of $\strapol$. Note that the interior of $\gamma$ never intersects the boundary of the disc, because it always has $R$ on one side and the inside of $\strapol$ on the other. On the other hand, our initial observation that $f_v$ meets $e_v$ and $e_w$ only at their intersection $v$ implies that the interior of $\gamma$ must contain either $u$ or $w$.

If $u$ is in the interior of $\gamma$, then in particular it is not on the boundary. Thus we may consider the non-trivial tendril $f_u$, contained in the strand $s_u$ crossing $s_v$ at $u$. Since $f_u$ starts by entering $R$, it must cross $\alpha\cup\gamma$ to reach the boundary. Since $s_u$ crosses $s_v$ at $u$, it cannot then cross $\alpha\subset s_v$ or $e_u\subset s_v$ without violating (b2), and so $f_u$ must exit $R$ by crossing $\gamma$ before $e_v$. 

Applying the same argument to $f_u$, we construct curves $\alpha'$ and $\gamma'$ bounding $R'\subset R$. The curve $\gamma'$ must be contained in $\gamma$, but it ends at $u$ and so has at least one fewer vertex in its interior. As $\gamma$ contained only finitely many vertices, by iterating this procedure we eventually arrive at $\bar{\gamma}$ with no interior vertices. But then the corresponding $\bar{\alpha}$ starts at a vertex $\bar{v}$ of the polygon and ends on the preceding edge, violating (b1).

The case that $w$ is in the interior of $\gamma$ is similar, using instead the tendril $f_w$. This tendril cannot intersect $f_v$, because the underlying strands met at $v$, and cannot intersect $e_w$ since this would be a self-intersection. Thus $f_w$ has the same pathological behaviour as $f_v$, but cuts off fewer vertices of $\strapol$, leading inductively to a contradiction.
\end{proof}

\begin{proposition}
\label{pro:wedge-covering}
Let $\strapol$ be a strand polygon in a Postnikov diagram $D$. Then the tendrils of $\strapol$ are disjoint. In particular, the tendrils end on the boundary in the same cyclic order as they start on the polygon, and the outside of $\strapol$ is partitioned by the downstream wedges of its vertices.
\end{proposition}

\begin{proof}

Consider an edge $e_v$, from $u$ to $v$. We first observe that the tendrils $f_v$ and $f_u$ do not cross. If $e_v$ is a boundary segment then $f_u$ and $f_v$ are distinct single points and there is nothing to prove, so we may assume that $e_v$ lies on a strand $s_v$. If the edge ending at $u$ is a boundary segment, then $f_u$ is the single point $u$, which is not on $f_v$ by Lemma~\ref{l:no-return}. Otherwise, $u$ lies on a strand $s_u$, which crosses $s_v$ at $u$, and so $f_v$ cannot intersect $f_u$ without violating (b2).

So assume there is some $v'\ne v$ for which $f_v$ and $f_{v'}$ cross, and let $\alpha$ be the path which follows $f_v$ until its first crossing with $f_{v'}$ and then follows $f_{v'}$ backwards until reaching $\strapol$. By Lemma~\ref{l:no-return}, $\alpha$ is a simple curve from $v$ to $v'$, intersecting $\strapol$ only at these points. Let $\gamma$ be the curve in $\strapol$ such that the inside $R$ of $\alpha\cup\gamma$ is entirely outside of $\strapol$. By the preceding paragraph, $v$ and $v'$ are not the two ends of a single edge, and so there is at least one vertex $v''$ of $\strapol$ in the interior of $\gamma$.

We may now argue similarly to Lemma~\ref{l:no-return}. The vertex $v''$ is not on the boundary, so the segment $f_{v''}$ begins by entering $R$, but it must leave $R$ before terminating. It cannot do so through $\strapol$ by Lemma~\ref{l:no-return}, so it must meet $\alpha$, either on $f_v$ or $f_{v'}$. Thus we may replace either $f_v$ or $f_{v'}$ by $f_{v''}$ and run the argument again. This replaces $\gamma$ by a curve $\gamma'$ containing fewer vertices, and hence leads inductively to a contradiction.

Thus we have proved the disjointness, and the remaining two statements follow directly from this.
\end{proof}

\begin{corollary}
For each $j\in Q_0$, the subset $\MSmat_j\subseteq Q_1$ is a perfect matching.
\end{corollary}
\begin{proof}
Apply Proposition~\ref{pro:wedge-covering}, taking $\strapol$ to be the boundary of the oriented region of $D$ corresponding to $j$.
\end{proof}

This recovers part of \cite[Thm.~5.3]{MuSp}, as promised. The more general covering property from Proposition~\ref{pro:wedge-covering} also enables us to prove the main objective of this section.

\begin{theorem}
\label{t:MSmatch}
Let $Q=Q(D)$ for $D$ a Postnikov diagram. For each $j\in Q_0$, we have $\alteta(\MSmat_j)=p_j$,
where $\MSmat_j$ is the Muller--Speyer matching \eqref{eq:MSmat}.
\end{theorem}

\begin{proof}
We need to calculate the coefficient of $p_i$
in the formula \eqref{eq:matcls-mu} for $\alteta(\MSmat_j)$
and show that this coefficient is $1$ when $i=j$ and $0$ otherwise. Since the first sum in \eqref{eq:matcls-mu} contributes $1$ for each $p_i$, 
what we need to show is
\begin{equation}\label{eq:count}
\#\{  \gamma\not\in\MSmat_j : h\gamma= i\}
- \#\{  \text{int } \gamma\in\MSmat_j : t\gamma= i\}
= \begin{cases} 1 & \text{if $i\neq j$}\\ 0 & \text{if  $i=j$} \end{cases}
\end{equation}
Since the matching $\MSmat_j$ contains all arrows $\gamma$ with $h\gamma= j$
and no arrows with $t\gamma= j$, the case $i=j$ is immediate.

For the case $i\neq j$ we consider the union of the alternating region $R_i$ corresponding to the vertex $i\in Q_0$ with the clockwise (oriented) regions adjacent to $R_i$, corresponding to the clockwise faces in $Q_2$ that have $i$ as a vertex. The boundary $\strapol_i^\white$ of this union is made up of the clockwise edges and boundary edges of $R_i$ together with the edges of each adjacent clockwise region, except the (necessarily unique) edge shared with $R_i$. These bounding edges are all distinct, as are the points at which they meet; the main ingredient here is that any point in the disc is incident with at most one clockwise region.

Hence $\strapol_i^\white$ is a simple closed curve, and it is even a strand polygon as follows. We observe that all vertices of $\strapol_i^\white$ apart from those at the ends of a boundary edge are corners of clockwise regions, and $\strapol_i^\white$ turns towards the region at these points. Moreover, the boundary edges of $\strapol_i^\white$ are edges of $R_i$, which is inside $\strapol_i^\white$.
See Figure~\ref{fig:strapol-i-white} for examples.

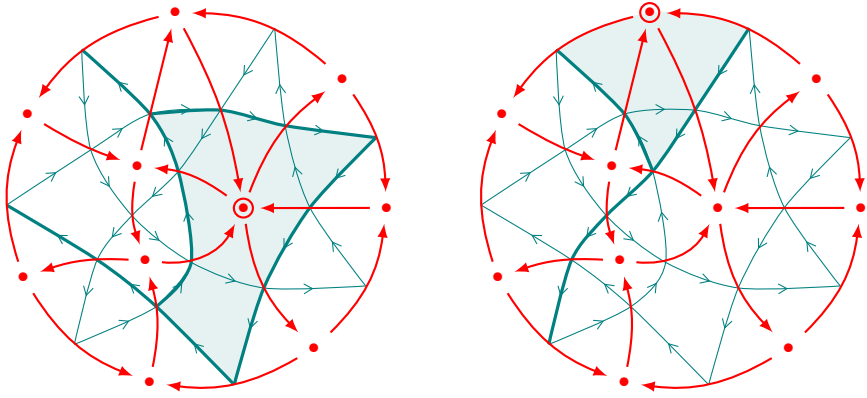
\begin{figure} [h]
	\pgfmathsetmacro{\dotcircrad}{0.05}
	\begin{tikzpicture}[scale=2.5,yscale=-1, baseline=(BB.base)]
		\coordinate (BB) at (0,0);
		\foreach \n/\m/\a in {1/4/0, 2/3/0, 3/2/5, 4/1/10, 5/7/0, 6/6/-3, 7/5/0}
		{ \coordinate (b\n) at (\bstart-\seventh*\n+\a:1.0);}
		
		\foreach \n/\m in {8/1, 9/2, 10/3, 11/4, 14/5, 15/6, 16/7}
		{\coordinate (b\n) at ($0.65*(b\m)$);}
		
		\coordinate (b13) at ($(b15) - (b16) + (b8)$);
		\coordinate (b12) at ($(b14) - (b15) + (b13)$);
		
		\foreach \n/\x\y in {13/-0.03/-0.03, 12/-0.22/0.0, 14/-0.07/-0.03, 11/0.05/0.02, 16/-0.02/0.02}
		{\coordinate (b\n)  at ($(b\n) + (\x,\y)$); } 

		\foreach \e/\f/\t in {2/9/0.5, 4/11/0.5, 5/14/0.5, 7/16/0.5,
			8/9/0.5, 9/10/0.5, 10/11/0.5,11/12/0.5, 12/13/0.45, 8/13/0.6,
			14/15/0.5, 15/16/0.6, 12/14/0.45, 13/15/0.4, 8/16/0.6}
		{\coordinate (a\e-\f) at ($(b\e) ! \t ! (b\f)$); }
		
		\draw [strand, white, fill=teal!10]
		(b1) to (a8-16) to [bend left = 8] (a8-13) to [bend left = 5] (a12-13) to [bend left = 5] (a12-14)
		to (a11-12) to (a10-11) to (b3) to (a9-10) to (a8-9) to (b1);

		\draw [strand] plot[smooth]
		coordinates {(b1) (a8-16) (a15-16) (b6)}
		[postaction=decorate, decoration={markings,
			mark= at position 0.2 with \strarrow,
			mark= at position 0.5 with \strarrow,
			mark= at position 0.8 with \strarrow }];
		
		\draw [strand] plot[smooth]
		coordinates {(b6) (a14-15) (a12-14)(a11-12) (a10-11) (b3)}
		[postaction=decorate, decoration={markings,
			mark= at position 0.15 with \strarrow, mark= at position 0.35 with \strarrow,
			mark= at position 0.53 with \strarrow, mark= at position 0.7 with \strarrow,
			mark= at position 0.87 with \strarrow }];
		
		\draw [strand] plot[smooth]
		coordinates {(b3) (a9-10) (a8-9) (b1)}
		[postaction=decorate, decoration={markings,
			mark= at position 0.2 with \strarrow,
			mark= at position 0.5 with \strarrow,
			mark= at position 0.8 with \strarrow }];
		
		\draw [strand] plot[smooth]
		coordinates {(b2) (a9-10) (a10-11) (b4)}
		[postaction=decorate, decoration={markings,
			mark= at position 0.2 with \strarrow,
			mark= at position 0.5 with \strarrow,
			mark= at position 0.8 with \strarrow }];
		
		\draw [strand] plot[smooth]
		coordinates {(b4) (a11-12) (a12-13) (a13-15) (a15-16) (b7)}
		[postaction=decorate, decoration={markings,
			mark= at position 0.15 with \strarrow, mark= at position 0.35 with \strarrow,
			mark= at position 0.55 with \strarrow, mark= at position 0.7 with \strarrow,
			mark= at position 0.87 with \strarrow }];
		
		\draw [strand] plot[smooth]
		coordinates {(b7) (a8-16) (a8-13) (a12-13) (a12-14) (b5)}
		[postaction=decorate, decoration={markings,
			mark= at position 0.15 with \strarrow, mark= at position 0.315 with \strarrow,
			mark= at position 0.5 with \strarrow, mark= at position 0.7 with \strarrow,
			mark= at position 0.88 with \strarrow }];
		
		\draw [strand] plot[smooth]
		coordinates {(b5) (a14-15) (a13-15) (a8-13) (a8-9) (b2)}
		[postaction=decorate, decoration={markings,
			mark= at position 0.13 with \strarrow, mark= at position 0.33 with \strarrow,
			mark= at position 0.5 with \strarrow, mark= at position 0.7 with \strarrow,
			mark= at position 0.88 with \strarrow }];
		
		\draw [strand, very thick] plot[smooth]
		coordinates {(b3) (a9-10) (a8-9) (b1)};
		\draw [strand, very thick] plot[smooth]
		coordinates {(b1) (a8-16) (a15-16) (b6)};
		\draw [strand, very thick] plot[smooth]
		coordinates {(a8-16) (a8-13) (a12-13) (a12-14) (b5)}; \draw [strand, very thick] plot[smooth]
		coordinates {(a12-14)(a11-12) (a10-11) (b3)}; 

		\foreach \n/\m\a in {1/124/0, 2/234/-1, 3/345/1, 4/456/10, 5/256/5, 6/267/0, 7/127/0}
		{ \draw [\quivcolor] (\bstart+\seventh/2-\seventh*\n+\a:1) node (q\m) {\tiny $\bullet$}; }
		\foreach \m/\a/\r in {247/\bstart/0.42, 245/10/0.25, 257/210/0.36}
		{ \draw [\quivcolor] (\a:\r) node (q\m) {\tiny $\bullet$}; }
		\draw [\quivcolor, thick] (q245) circle (\dotcircrad);
\foreach \t/\h/\a in {234/124/-20, 234/345/19, 456/345/-16, 456/256/22, 256/267/21, 127/267/-20,
 127/124/19, 124/247/13, 247/127/11, 245/234/18, 247/245/34, 257/247/14, 245/257/15,
 267/257/6, 257/256/0, 256/245/-9, 345/245/0, 245/456/-16}
{ \draw [quivarrow]  (q\t) edge [bend left=\a] (q\h); }
	\end{tikzpicture}
	\qquad
	\begin{tikzpicture}[scale=2.5,yscale=-1, baseline=(BB.base)]
		\coordinate (BB) at (0,0);
		\foreach \n/\m/\a in {1/4/0, 2/3/0, 3/2/5, 4/1/10, 5/7/0, 6/6/-3, 7/5/0}
		{ \coordinate (b\n) at (\bstart-\seventh*\n+\a:1.0);}
		
		\foreach \n/\m in {8/1, 9/2, 10/3, 11/4, 14/5, 15/6, 16/7}
		{\coordinate (b\n) at ($0.65*(b\m)$);}
		
		\coordinate (b13) at ($(b15) - (b16) + (b8)$);
		\coordinate (b12) at ($(b14) - (b15) + (b13)$);
		
		\foreach \n/\x\y in {13/-0.03/-0.03, 12/-0.22/0.0, 14/-0.07/-0.03, 11/0.05/0.02, 16/-0.02/0.02}
		{\coordinate (b\n)  at ($(b\n) + (\x,\y)$); } 

		\foreach \e/\f/\t in {2/9/0.5, 4/11/0.5, 5/14/0.5, 7/16/0.5,
			8/9/0.5, 9/10/0.5, 10/11/0.5,11/12/0.5, 12/13/0.45, 8/13/0.6,
			14/15/0.5, 15/16/0.6, 12/14/0.45, 13/15/0.4, 8/16/0.6}
		{\coordinate (a\e-\f) at ($(b\e) ! \t ! (b\f)$); }
		
		\draw [strand, white, fill=teal!10] (b4) to [bend left = 28] (b5)
		to (a12-14) to (a12-13) to (a11-12) to (b4);
		
		\draw [strand] plot[smooth]
		coordinates {(b1) (a8-16) (a15-16) (b6)}
		[postaction=decorate, decoration={markings,
			mark= at position 0.2 with \strarrow,
			mark= at position 0.5 with \strarrow,
			mark= at position 0.8 with \strarrow }];
		
		\draw [strand] plot[smooth]
		coordinates {(b6) (a14-15) (a12-14)(a11-12) (a10-11) (b3)}
		[postaction=decorate, decoration={markings,
			mark= at position 0.15 with \strarrow, mark= at position 0.35 with \strarrow,
			mark= at position 0.53 with \strarrow, mark= at position 0.7 with \strarrow,
			mark= at position 0.87 with \strarrow }];
		
		\draw [strand] plot[smooth]
		coordinates {(b3) (a9-10) (a8-9) (b1)}
		[postaction=decorate, decoration={markings,
			mark= at position 0.2 with \strarrow,
			mark= at position 0.5 with \strarrow,
			mark= at position 0.8 with \strarrow }];
		
		\draw [strand] plot[smooth]
		coordinates {(b2) (a9-10) (a10-11) (b4)}
		[postaction=decorate, decoration={markings,
			mark= at position 0.2 with \strarrow,
			mark= at position 0.5 with \strarrow,
			mark= at position 0.8 with \strarrow }];
		
		\draw [strand] plot[smooth]
		coordinates {(b4) (a11-12) (a12-13) (a13-15) (a15-16) (b7)}
		[postaction=decorate, decoration={markings,
			mark= at position 0.15 with \strarrow, mark= at position 0.35 with \strarrow,
			mark= at position 0.55 with \strarrow, mark= at position 0.7 with \strarrow,
			mark= at position 0.87 with \strarrow }];
		
		\draw [strand] plot[smooth]
		coordinates {(b7) (a8-16) (a8-13) (a12-13) (a12-14) (b5)}
		[postaction=decorate, decoration={markings,
			mark= at position 0.15 with \strarrow, mark= at position 0.315 with \strarrow,
			mark= at position 0.5 with \strarrow, mark= at position 0.7 with \strarrow,
			mark= at position 0.88 with \strarrow }];
		
		\draw [strand] plot[smooth]
		coordinates {(b5) (a14-15) (a13-15) (a8-13) (a8-9) (b2)}
		[postaction=decorate, decoration={markings,
			mark= at position 0.13 with \strarrow, mark= at position 0.33 with \strarrow,
			mark= at position 0.5 with \strarrow, mark= at position 0.7 with \strarrow,
			mark= at position 0.88 with \strarrow }];
		
		\draw [strand, very thick] plot[smooth]
		coordinates { (a12-13) (a12-14) (b5)}; \draw [strand, very thick] plot[smooth]
		coordinates {(b4) (a11-12) (a12-13) (a13-15) (a15-16) (b7)};

		\foreach \n/\m\a in {1/124/0, 2/234/-1, 3/345/1, 4/456/10, 5/256/5, 6/267/0, 7/127/0}
		{ \draw [\quivcolor] (\bstart+\seventh/2-\seventh*\n+\a:1) node (q\m) {\tiny $\bullet$}; }
		\foreach \m/\a/\r in {247/\bstart/0.42, 245/10/0.25, 257/210/0.36}
		{ \draw [\quivcolor] (\a:\r) node (q\m) {\tiny $\bullet$}; }
		\draw [\quivcolor, thick] (q256) circle (\dotcircrad);
\foreach \t/\h/\a in {234/124/-20, 234/345/19, 456/345/-16, 456/256/22, 256/267/21, 127/267/-20,
 127/124/19, 124/247/13, 247/127/11, 245/234/18, 247/245/34, 257/247/14, 245/257/15,
 267/257/6, 257/256/0, 256/245/-9, 345/245/0, 245/456/-16}
{ \draw [quivarrow]  (q\t) edge [bend left=\a] (q\h); }
	\end{tikzpicture}
	\caption{Two examples of the strand polygon $\strapol_i^\white$ and its tendrils;
		the vertex $i$ is circled in each case.} \label{fig:strapol-i-white}
\end{figure}

Note that every quiver vertex $j\neq i$ is outside of $\strapol_i^\white$, and hence by Proposition~\ref{pro:wedge-covering} is contained in a unique downstream wedge of this polygon. Consider a vertex of $\strapol_i^\white$ not incident with $R_i$, i.e.\ a vertex of one of the clockwise regions added to $R_i$ to obtain $\strapol_i^\white$. Then the wedge of $\strapol_i^\white$ at this vertex is the wedge of the corresponding arrow, which is not incident with $i$ but lies in a clockwise face incident with $i$.
On the other hand, the remaining vertices of $\strapol_i^\white$ either start a boundary edge, in which case the wedge is trivial, or end a boundary edge, in which case the wedge is one side of the strand starting at this vertex.

Consider an arrow $\gamma$ with $h\gamma=i$. If $\gamma$ does not lie in a clockwise face, then $\gamma$ is a boundary arrow at which a boundary edge of $\strapol_i^\white$ ends, and we let $W_\gamma$ be the wedge of $\strapol_i^\white$ at its vertex on $\gamma$, which is the complement of the wedge of $\gamma$. Otherwise, $\gamma$ lies in a clockwise face $F$ in which the next arrow $\gamma'$ has $t\gamma'=i$. If $\gamma'$ is internal, let $W_\gamma$ be the union of wedges of the vertices of $\strapol_i^\white$ on $\bdry F$, whereas if $\gamma'$ is a boundary arrow, let $W_\gamma$ be the union of these wedges together with the wedge of the vertex of $\strapol_i^\white$ on $\gamma'$ (which is just the wedge of $\gamma'$ in this case). Note that every arrow incident with $i$ is either one of the arrows $\gamma$ or $\gamma'$ considered above, or is a boundary arrow with tail at $i$, and thus irrelevant to the calculation \eqref{eq:count}. Note further that every wedge of $\strapol_i^\white$, and hence every quiver vertex $j\ne i$, is contained in $W_\gamma$ for a unique arrow $\gamma$ with $h\gamma=i$.

Suppose $j\in W_\gamma$. If $\gamma$ is not in a clockwise face, this means that $j$ is not in the downstream wedge of $\gamma$, and so $\gamma\notin\MSmat_j$ counts $1$ on the left-hand side of \eqref{eq:count}. Otherwise, $\gamma$ is followed by $\gamma'$ in a clockwise face $F$. If $\gamma'$ is internal, then $j\in W_\gamma$ means that $j$ is in the wedge of some arrow in $F$ different from $\gamma$ and $\gamma'$. Then $\gamma\notin\MSmat_j$ counts $1$ in \eqref{eq:count}, whereas $\gamma'\notin\MSmat_j$ counts $0$. If $\gamma'$ is on the boundary, then its wedge is also contained in $W_\gamma$, so we could have $\gamma'\in\MSmat_j$, but it contributes $0$ in \eqref{eq:count} anyway. Thus whenever $j\in W_\gamma$, the total contribution of $\gamma$ and $\gamma'$ (if it exists) to \eqref{eq:count} is $1$.

Now suppose $j\notin W_\gamma$. If $\gamma$ is not in a clockwise face, this means that $\gamma\in\MSmat_j$ counts $0$ in \eqref{eq:count}. If $\gamma$ is followed by a boundary arrow $\gamma'$ in a clockwise face $F$, this means that $\gamma\in\MSmat_j$, since the wedges of all other arrows of $F$ are contained in $W_\gamma$. Thus $\gamma$ counts $0$ in \eqref{eq:count}, as does $\gamma'$ since it is on the boundary. On the other hand, if $\gamma$ is followed by an internal arrow $\gamma'$, then $W_\gamma$ consists of the wedges of arrows in $F$ different from $\gamma$ and $\gamma'$, so either $\gamma\in\MSmat_j$ or $\gamma'\in\MSmat_j$. If $\gamma\in\MSmat_j$ it counts $0$ in \eqref{eq:count}, as does $\gamma'\notin\MSmat_j$. If $\gamma'\in\MSmat_j$ it counts $-1$ in \eqref{eq:count}, while $\gamma\notin\MSmat_j$ counts $1$. In any case, the total contribution of $\gamma$ and $\gamma'$ is $0$.

Summing up, we see that the total contribution in \eqref{eq:count} of all arrows incident with $i\ne j$ is $1$, as required.
\end{proof}

Note that we could equally well have used the strand polygon $\strapol_i^\black$, 
bounding the union of $R_i$ with its adjacent anticlockwise regions, in place of $\strapol_i^\white$ in the preceding proof. As already observed, Theorem~\ref{t:MSmatch} leads directly to the following results.

\begin{corollary}
\label{c:eta=MSinv}
We have $\MSmap^{-1}=\alteta$.
\end{corollary}
\begin{proof}
Theorem~\ref{t:MSmatch} proves that $\alteta\compo\MSmap=\id$, which is sufficient because they are maps between lattices of the same rank.
\end{proof}

Alternatively, having shown that $\alteta\circ\MSmap=\id$, we may reach the same conclusion by using that $\MSmap$ (or $\alteta$, if $D$ is connected) is known to be an isomorphism.

\begin{corollary}
\label{c:MSmatch}
If $D$ is connected, then the indecomposable projective $A_D$-module $P_j=Ae_j$ is isomorphic to the matching module $N_{\MSmat_j}$.
\end{corollary}
\begin{proof}
Combining \eqref{eq:matcls-value} with Theorem~\ref{t:MSmatch} gives $[N_{\MSmat_j}]=\matcls(\MSmat_j)=[P_j]$. The result then follows from Corollary~\ref{cor:rank1-rigid}.
\end{proof}

Muller--Speyer also consider \cite[\S5.6]{MuSp} the matching $\MSmat_j\dual$ defined analogously to $\MSmat_j$ but using the upstream wedge of an arrow in place of its downstream wedge. Theorem~\ref{t:MSmatch} also allows us to identify the perfect matching module $N_{\MSmat_j\dual}$.

\begin{corollary}
\label{c:MSmatch-op}
If $D$ is connected, then $(e_jA)\dual:=\Hom_Z(e_jA,Z)$ is isomorphic to the matching module $N_{\MSmat_j\dual}$ for each $j\in Q_0$.
\end{corollary}
\begin{proof}
Consider the opposite diagram $D\op$ (Definition~\ref{d:opposite}), for which $Q_{D\op}=Q\op$ and $A_{D\op}=A\op$ (Remark~\ref{r:opposite}). We write $\MSmat_j\op$ for the Muller--Speyer matching of $Q\op$ associated to vertex $j$, to distinguish this from the Muller--Speyer matching of $Q$ for this vertex. Applying Corollary~\ref{c:MSmatch} to $D\op$ shows that the $A\op$-module $e_jA=A\op e_j$ is isomorphic to $N_{\MSmat_j\op}$.

Since $Q$ and $Q\op$ have the same set of arrows, we may also view $\MSmat_j\op$ as a matching of $Q$, where it coincides with $\MSmat_j\dual$. Moreover, the set of arrows of $Q\op$ acting as $t$ on a rank one $A\op$-module $N$ agrees with the set of arrows of $Q$ acting as $t$ on the rank one $A$-module $N\dual$ (that is, $({}_{A\op}N_\mu)\dual\isom {}_AN_\mu$ for any perfect matching $\mu$), and so we conclude that $(e_jA)\dual\isoto N_{\MSmat_j\dual}$ as required.
\end{proof}

\section{Labelling} \label{sec:labels}

Let $D$ be a Postnikov diagram with quiver $Q=Q(D)$. Following, among others, \cite[\S3]{Sco} and \cite[\S4]{MuSp}, we associate a \emph{label} $I_j\subset\C_1$, to each vertex $j\in Q_0$. To define this label, note that each strand in $D$ divides the disc into two parts: the left-hand side and right-hand side, relative to the orientation of the strand.

\begin{definition}
\label{d:label}
For $j\in Q_0$, define the \emph{(left) source label} $I_j\subset\C_1$ to consist of those $i\in\C_1$ such that the strand of $D$ starting at $i$ has $j$ on its left-hand side.
\end{definition}

\begin{figure}[h]
\begin{tikzpicture}[scale=3,baseline=(bb.base),yscale=-1]

\path (0,0) node (bb) {};

\foreach \n/\m/\a in {1/4/0, 2/3/0, 3/2/5, 4/1/10, 5/7/0, 6/6/-3, 7/5/0}
{ \coordinate (b\n) at (\bstart-\seventh*\n+\a:1.0);
  \draw (\bstart-\seventh*\n+\a:1.1) node {$\m$}; }

\foreach \n/\m in {8/1, 9/2, 10/3, 11/4, 14/5, 15/6, 16/7}
  {\coordinate (b\n) at ($0.65*(b\m)$);}

\coordinate (b13) at ($(b15) - (b16) + (b8)$);
\coordinate (b12) at ($(b14) - (b15) + (b13)$);

\foreach \n/\x\y in {13/-0.03/-0.03, 12/-0.22/0.0, 14/-0.07/-0.03, 11/0.05/0.02, 16/-0.02/0.02}
  {\coordinate (b\n)  at ($(b\n) + (\x,\y)$); } 

\foreach \e/\f/\t in {2/9/0.5, 4/11/0.5, 5/14/0.5, 7/16/0.5, 
 8/9/0.5, 9/10/0.5, 10/11/0.5,11/12/0.5, 12/13/0.45, 8/13/0.6, 
 14/15/0.5, 15/16/0.6, 12/14/0.45, 13/15/0.4, 8/16/0.6}
{\coordinate (a\e-\f) at ($(b\e) ! \t ! (b\f)$); }

\draw [strand] plot[smooth]
coordinates {(b1) (a8-16) (a15-16) (b6)}
[postaction=decorate, decoration={markings,
 mark= at position 0.2 with \strarrow,
 mark= at position 0.5 with \strarrow, 
 mark= at position 0.8 with \strarrow }];
 
\draw [strand] plot[smooth]
coordinates {(b6) (a14-15) (a12-14)(a11-12) (a10-11) (b3)}
[postaction=decorate, decoration={markings,
 mark= at position 0.15 with \strarrow, mark= at position 0.35 with \strarrow,
 mark= at position 0.53 with \strarrow, mark= at position 0.7 with \strarrow,
 mark= at position 0.87 with \strarrow }];
 
\draw [strand] plot[smooth]
coordinates {(b3) (a9-10) (a8-9) (b1)}
[postaction=decorate, decoration={markings,
 mark= at position 0.2 with \strarrow,
 mark= at position 0.5 with \strarrow, 
 mark= at position 0.8 with \strarrow }];

\draw [strand] plot[smooth]
coordinates {(b2) (a9-10) (a10-11) (b4)}
 [postaction=decorate, decoration={markings,
 mark= at position 0.2 with \strarrow,
 mark= at position 0.5 with \strarrow, 
 mark= at position 0.8 with \strarrow }];

\draw [strand] plot[smooth]
coordinates {(b4) (a11-12) (a12-13) (a13-15) (a15-16) (b7)}
[postaction=decorate, decoration={markings,
 mark= at position 0.15 with \strarrow, mark= at position 0.35 with \strarrow,
 mark= at position 0.55 with \strarrow, mark= at position 0.7 with \strarrow,
 mark= at position 0.87 with \strarrow }];

\draw [strand] plot[smooth]
coordinates {(b7) (a8-16) (a8-13) (a12-13) (a12-14) (b5)}
[postaction=decorate, decoration={markings,
 mark= at position 0.15 with \strarrow, mark= at position 0.315 with \strarrow,
 mark= at position 0.5 with \strarrow, mark= at position 0.7 with \strarrow,
 mark= at position 0.88 with \strarrow }];

\draw [strand] plot[smooth]
coordinates {(b5) (a14-15) (a13-15) (a8-13) (a8-9) (b2)}
[postaction=decorate, decoration={markings,
 mark= at position 0.13 with \strarrow, mark= at position 0.33 with \strarrow,
 mark= at position 0.5 with \strarrow, mark= at position 0.7 with \strarrow,
 mark= at position 0.88 with \strarrow }];

\foreach \n/\m/\l/\a in {1/124/134/0, 2/234/123/-1, 3/345/127/1, 4/456/167/10, 5/256/367/5, 6/267/356/0, 7/127/345/0}
{ \draw [\quivcolor] (\bstart+\seventh/2-\seventh*\n+\a:1) node (q\m) {\scriptsize $\l$}; }

\foreach \m/\l/\a/\r in {247/135/\bstart/0.42, 245/137/10/0.25, 257/357/210/0.36}
{ \draw [\quivcolor] (\a:\r) node (q\m) {\scriptsize $\l$}; }

\foreach \t/\h/\a in {234/124/-19, 234/345/19, 456/345/-16, 456/256/21, 256/267/21, 127/267/-20,
 127/124/19, 124/247/13, 247/127/11, 245/234/18, 247/245/34, 257/247/14, 245/257/15,
 267/257/6, 257/256/0, 256/245/-9, 345/245/0, 245/456/-16}
{ \draw [quivarrow]  (q\t) edge [bend left=\a] (q\h); }

\end{tikzpicture}
 \caption{The source labels $I_j$ for a Postnikov diagram of type $(3,7)$.}
 \label{f:source-labels}
\end{figure}
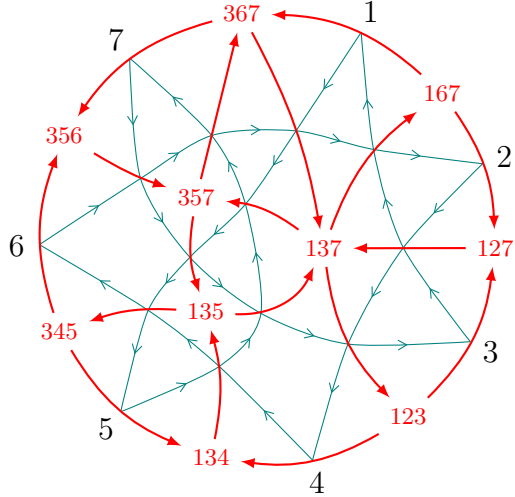

Figure~\ref{f:source-labels} shows the labels $I_j$, drawn in place of the quiver vertices, in our running example (cf. Figure~\ref{f:quiver37}).
The significance of these labels comes from cluster algebras, which we will discuss further in Section~\ref{sec:newMS}: each Postnikov diagram determines an initial seed for a cluster algebra structure on the corresponding (open) positroid variety \cite{GL,SSBW}, for which the initial cluster variables are restrictions of the Pl\"ucker coordinates $\Pluck{I_j}$ for $j\in Q_0$.

We are now able to interpret these labels algebraically, as follows. Write $A=A_D$ and $B=eAe$, and let $\rho\colon\CM(B)\to\CM(\jksalg)$ be the restriction functor from Proposition~\ref{p:restriction}.

\begin{theorem}
\label{t:labels}
Let $D$ be a connected Postnikov diagram. For each $j\in Q_0$, the indecomposable projective $A$-module $Ae_j$ satisfies
\begin{equation}
\label{eq:boundary-projectives}
\rho(eAe_j)\isom M_{I_j},
\end{equation}
so that $\rho(eA)=\bigoplus_{j\in Q_0}M_{I_j}$. In particular, a Postnikov diagram $D$ with $n$ strands has type $(k,n)$ if and only if each $I_j$ has cardinality $k$.
\end{theorem}
\begin{proof}
Since $Ae_j\isom N_{\MSmat_j}$ by Corollary~\ref{c:MSmatch}, it follows from Proposition~\ref{p:res-pm-mods} that
\[\rho(eAe_j)\isom M_{\bdry\MSmat_j}.\]
But by \cite[Thm.~5.3]{MuSp}, the boundary value $\bdry\MSmat_j$ is precisely the set $I_j$. In particular, if $D$ has type $(k,n)$ then $I_j$ has cardinality $k$ by Proposition~\ref{p:bdry-value-card}.
\end{proof}

It is shown in \cite{Sco} that when $D$ is a $(k,n)$-diagram, meaning that the associated permutation has $\pi_D(i)=i+k\pmod{n}$ and that the seed attached to $D$ generates a cluster algebra structure on the Grassmannian $\grass{k}{n}$, then each label $I_j$ has cardinality $k$. It thus follows from Theorem~\ref{t:labels} that $(k,n)$-diagrams have type $(k,n)$, as promised in Definition~\ref{d:type}.

\begin{corollary}
\label{c:BKM}
For any connected Postnikov diagram $D$, there is an isomorphism
\[A_D\isoto\End_\jksalg\Bigl(\bigoplus_{j\in Q_0}M_{I_j}\Bigr)\op.\]
\end{corollary}
\begin{proof}
Writing $A=A_D$, Proposition~\ref{p:double-cen}(i) provides an isomorphism
\[A\isoto\End_B(eA)\op=\End_B\Bigl(\bigoplus_{j\in Q_0} eAe_j\Bigr)\op.\]
Since $\rho\colon\CM(B)\to\CM(\jksalg)$ is fully faithful by Proposition~\ref{p:restriction}, and $\rho(eAe_j)\isom M_{I_j}$ by Theorem~\ref{t:labels}, we get a further isomorphism
\[\End_B\Bigl(\bigoplus_{j\in Q_0} eAe_j\Bigr)\op\isoto\End_\jksalg\Bigl(\bigoplus_{j\in Q_0}M_{I_j}\Bigr)\op,\]
as required.
\end{proof}

\begin{remark}
\label{r:target-labels}
Several sources, including \cite{MuSp,OPS}, also consider the \emph{target labels} $I_j\dual$ consisting of those $i\in\C_1$ such that the strand ending at $i$ has $j$ on its left-hand side (see Fig.~\ref{f:target-labels}). The analogous statement to Theorem~\ref{t:labels} for these labels is that
\begin{equation}
\label{eq:boundary-injectives}
\rho((e_jAe)\dual)\isom M_{I_j\dual},
\end{equation}
where $(-)\dual=\Hom_Z(-,Z)$---this follows from Corollary~\ref{c:MSmatch-op} together with the analogue of \cite[Thm.~5.3]{MuSp} for the upstream wedge matchings $\MSmat_j\dual$, showing that $\bdry\MSmat_j\dual=I_j\dual$.

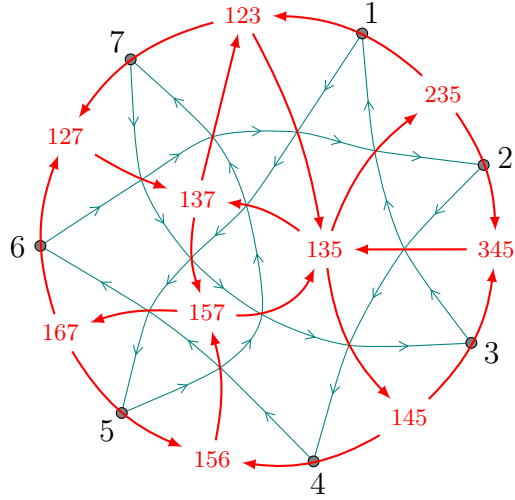
\begin{figure}[h]
\begin{tikzpicture}[scale=3,baseline=(bb.base),yscale=-1]

\path (0,0) node (bb) {};

\foreach \n/\m/\a in {1/4/0, 2/3/0, 3/2/5, 4/1/10, 5/7/0, 6/6/-3, 7/5/0}
{ \coordinate (b\n) at (\bstart-\seventh*\n+\a:1.0);
  \draw (\bstart-\seventh*\n+\a:1.1) node {$\m$}; }

\foreach \n/\m in {8/1, 9/2, 10/3, 11/4, 14/5, 15/6, 16/7}
  {\coordinate (b\n) at ($0.65*(b\m)$);}

\coordinate (b13) at ($(b15) - (b16) + (b8)$);
\coordinate (b12) at ($(b14) - (b15) + (b13)$);

\foreach \n/\x\y in {13/-0.03/-0.03, 12/-0.22/0.0, 14/-0.07/-0.03, 11/0.05/0.02, 16/-0.02/0.02}
  {\coordinate (b\n)  at ($(b\n) + (\x,\y)$); } 

\foreach \e/\f/\t in {2/9/0.5, 4/11/0.5, 5/14/0.5, 7/16/0.5, 
 8/9/0.5, 9/10/0.5, 10/11/0.5,11/12/0.5, 12/13/0.45, 8/13/0.6, 
 14/15/0.5, 15/16/0.6, 12/14/0.45, 13/15/0.4, 8/16/0.6}
{\coordinate (a\e-\f) at ($(b\e) ! \t ! (b\f)$); }

\draw [strand] plot[smooth]
coordinates {(b1) (a8-16) (a15-16) (b6)}
[postaction=decorate, decoration={markings,
 mark= at position 0.2 with \strarrow,
 mark= at position 0.5 with \strarrow, 
 mark= at position 0.8 with \strarrow }];
 
\draw [strand] plot[smooth]
coordinates {(b6) (a14-15) (a12-14)(a11-12) (a10-11) (b3)}
[postaction=decorate, decoration={markings,
 mark= at position 0.15 with \strarrow, mark= at position 0.35 with \strarrow,
 mark= at position 0.53 with \strarrow, mark= at position 0.7 with \strarrow,
 mark= at position 0.87 with \strarrow }];
 
\draw [strand] plot[smooth]
coordinates {(b3) (a9-10) (a8-9) (b1)}
[postaction=decorate, decoration={markings,
 mark= at position 0.2 with \strarrow,
 mark= at position 0.5 with \strarrow, 
 mark= at position 0.8 with \strarrow }];

\draw [strand] plot[smooth]
coordinates {(b2) (a9-10) (a10-11) (b4)}
 [postaction=decorate, decoration={markings,
 mark= at position 0.2 with \strarrow,
 mark= at position 0.5 with \strarrow, 
 mark= at position 0.8 with \strarrow }];

\draw [strand] plot[smooth]
coordinates {(b4) (a11-12) (a12-13) (a13-15) (a15-16) (b7)}
[postaction=decorate, decoration={markings,
 mark= at position 0.15 with \strarrow, mark= at position 0.35 with \strarrow,
 mark= at position 0.55 with \strarrow, mark= at position 0.7 with \strarrow,
 mark= at position 0.87 with \strarrow }];

\draw [strand] plot[smooth]
coordinates {(b7) (a8-16) (a8-13) (a12-13) (a12-14) (b5)}
[postaction=decorate, decoration={markings,
 mark= at position 0.15 with \strarrow, mark= at position 0.315 with \strarrow,
 mark= at position 0.5 with \strarrow, mark= at position 0.7 with \strarrow,
 mark= at position 0.88 with \strarrow }];

\draw [strand] plot[smooth]
coordinates {(b5) (a14-15) (a13-15) (a8-13) (a8-9) (b2)}
[postaction=decorate, decoration={markings,
 mark= at position 0.13 with \strarrow, mark= at position 0.33 with \strarrow,
 mark= at position 0.5 with \strarrow, mark= at position 0.7 with \strarrow,
 mark= at position 0.88 with \strarrow }];

\foreach \n/\m/\l/\a in {1/124/156/0, 2/234/145/-1, 3/345/345/1, 4/456/235/10, 5/256/123/5, 6/267/127/0, 7/127/167/0}
{ \draw [\quivcolor] (\bstart+\seventh/2-\seventh*\n+\a:1) node (q\m) {\scriptsize $\l$}; }

\foreach \m/\l/\a/\r in {247/157/\bstart/0.42, 245/135/10/0.25, 257/137/210/0.36}
{ \draw [\quivcolor] (\a:\r) node (q\m) {\scriptsize $\l$}; }

\foreach \t/\h/\a in {234/124/-19, 234/345/19, 456/345/-16, 456/256/21, 256/267/21, 127/267/-20,
 127/124/19, 124/247/13, 247/127/11, 245/234/18, 247/245/34, 257/247/14, 245/257/15,
 267/257/6, 257/256/0, 256/245/-9, 345/245/0, 245/456/-16}
{ \draw [quivarrow]  (q\t) edge [bend left=\a] (q\h); }

\end{tikzpicture}
 \caption{The target labels $I_j\dual$ for a Postnikov diagram of type $(3,7)$.}
 \label{f:target-labels}
\end{figure}

Applying Proposition~\ref{p:double-cen} to $D\op$, with dimer algebra $A\op$, yields
\[A\op\isoto\End_{B\op}(Ae)\op\isoto\End_B((Ae)\dual),\]
where the second isomorphism uses that $B$ and $Ae$ are free and finitely generated over $Z$, so that the duality $(-)\dual$ is an equivalence $\CM(B\op)\isoto\CM(B)\op$.
Then the same argument as for Corollary~\ref{c:BKM} shows that there is an isomorphism
\[A_D\isoto\End_\jksalg\Bigl(\bigoplus_{j\in Q_0}M_{I_j\dual}\Bigr)\op.\]
\end{remark}

\begin{remark}
\label{r:BKM-comparison}
When $D$ is a $(k,n)$-diagram, \cite[Thm.~10.3]{BKM} also exhibits an isomorphism of $A_D$ with an endomorphism algebra; using our notation and conventions (see Remark~\ref{r:JKS-comparison}), the isomorphism is
\[A_D\isoto\End_{\widetilde{\jksalg}}\Bigl(\bigoplus_{j\in Q_0}M_{I_j\comp}\Bigr).\]
Here $\widetilde{\jksalg}=\Pi/(y^{n-k}-x^k)$, i.e.\ it is the algebra of Definition~\ref{d:preproj-Bkn} but with parameters $(n-k,n)$, noting that $I_j\comp$ has cardinality $n-k$, and unlike in Corollary~\ref{c:BKM} we do not take the opposite of the endomorphism algebra.

This isomorphism is, however, equivalent to the isomorphism of Remark~\ref{r:target-labels} for the $(n-k,n)$-diagram $D\op$. Each vertex $j$ of $Q(D\op)$ has target label $I_j\comp$, where $I_j$ is the source label of $j$ in $D$. Thus Remark~\ref{r:target-labels} tells us that
\[A_{D\op}=A_D\op\isoto\End_{\widetilde{\jksalg}}\Bigl(\bigoplus_{j\in Q_0}M_{I_j\comp}\Bigr)\op,\]
which is equivalent to \cite[Thm.~10.3]{BKM} by taking opposite algebras.

Similarly, one obtains a fourth isomorphism
\[A_D\isoto\End_{\widetilde{\jksalg}}\Bigl(\bigoplus_{j\in Q_0}M_{(I_j\dual)\comp}\Bigr).\]
of $A_D$ with an endomorphism algebra.
\end{remark}

An important special case of Theorem~\ref{t:labels} is when
$j\in Q_0$ is a boundary vertex, so that $eAe_j=Be_j$,
which is an indecomposable projective for the boundary algebra $B=eAe$.
Thus $\rho(Be_j)=M_{I_j}$ and the labels $I_j$ for the $n$ boundary vertices
$j\in Q^{\bdry}_0$
are precisely the \emph{source necklace} associated to the Postnikov diagram~$D$
or just to the strand permutation
(cf.\ \cite[Prop.~4.3]{MuSp}, where it is called the reverse necklace).
Similarly, the \emph{target necklace} (or just necklace in \cite{MuSp}) consists of the labels
$I_j\dual$ for $j\in Q^{\bdry}_0$ and we have $\rho(e_jB\dual)=M_{I_j\dual}$ by \eqref{eq:boundary-injectives},
noting that $e_jB\dual$ are the indecomposable injective objects in $\CM(B)$.

Given a functor $\rho\colon\mathcal{B}\to\mathcal{C}$, recall that its \emph{essential image} is defined to be the full subcategory of $\mathcal{C}$ on objects isomorphic to $\rho(X)$ for some object $X\in\mathcal{B}$.

\begin{proposition}
\label{pro:positroid}
Consider the restriction functor $\rho\colon\CM(B)\to\CM(C)$. For each $k$-subset $J$, the $C$-module $M_J$ is in the essential image of $\rho$ if and only if $J\in\posid$, where $\posid$ is the positroid associated to $D$.
\end{proposition}
\begin{proof}
By \cite[Thm.~3.1]{MuSp}, the positroid $\posid$ consists of boundary values of perfect matchings on the graph $\plabic(D)$, or, from our point of view, on the quiver $Q(D)$. For any such perfect matching $\mu$ we have $M_{\bdry\mu}\isom\rho(eN_\mu)$ in the essential image of $\rho$ by Proposition~\ref{p:res-pm-mods}. Conversely, assume $M_J\isom\rho(M)$. Then $M$, and hence $FM$, has rank $1$, and so by Corollary~\ref{c:proj-pm-mods} there is a perfect matching $\mu$ of $D$ such that $FM\isom N_\mu$. Then $M_{\bdry\mu}\isom\rho(eN_\mu)\isom\rho(eFM)=\rho(M)\isom M_J$, and so $J=\bdry\mu$.
\end{proof}

An alternative point of view is to note that $\jksalg\subset B\subset \jksalg[t^{-1}]$ 
by the proof of Proposition~\ref{p:restriction}. 
Any $\jksalg$-module $M$ is naturally a subspace of the $\jksalg[t^{-1}]$-module $M[t^{-1}]$, and asking that $M$ is in the essential image of $\rho$ is equivalent to asking if this subspace is a $B$-submodule; if it is, we simply say that $M$ is a $B$-module. 
Using the combinatorics of profiles \cite[\S6]{JKS}, one can show that the combinatorial condition
(see \cite[\S2.1]{MuSp} or \cite[\S5]{Oh11})
that determines whether $J$ is in the positroid, in terms of the source (or reverse \cite{MuSp}, or upper \cite{Oh11}) necklace,
is precisely the condition that $M_J$ is a $B$-module.

\section{The Marsh--Scott formula} \label{sec:newMS}

Let $D$ be a Postnikov diagram, and consider the associated quiver $Q(D)$. 
Being a quiver with frozen vertices (those corresponding to boundary regions of $D$), 
we may associate to it a cluster algebra with frozen variables. 
We may choose whether to adopt the convention that frozen variables are invertible, 
in which case we call the resulting cluster algebra $\openclustalg{D}$, or that they are not, 
in which case we obtain the cluster algebra $\clustalg{D}$. 
In either case, the cluster algebra is defined as a subalgebra of the field of rational functions in the initial cluster variables $x_j$ for $j\in Q_0$.

If $D$ has type $(k,n)$, recall from Section~\ref{sec:labels} that each quiver vertex $j\in Q_0$
determines a $k$-subset $I_j\subset\C_1$, 
and hence a Pl\"ucker coordinate~$\Pluck{I_j}$ on the Grassmannian~$\grass{k}{n}$. 
This yields a natural specialisation map 
\begin{equation}\label{eq:specialise}
\CC(x_j:j\in Q_0)\to\CC(\grass{k}{n})\colon x_j\mapsto\Pluck{I_j},
\end{equation}
taking rational functions in the $x_j$ to rational functions on the $(k,n)$-Grassmannian. 
When $D$ is a $(k,n)$-diagram, this specialisation map restricts to an isomorphism 
$\clustalg{D}\isoto\CC[\grass{k}{n}]$, yielding Scott's cluster structure \cite{Sco} on $\CC[\grass{k}{n}]$. 
For a more general Postnikov diagram $D$, the specialisation restricts to an isomorphism of 
$\openclustalg{D}$ with the (homogeneous) coordinate ring of the open positroid variety corresponding to the 
permutation $\pi_D$ \cite{GL}, yielding the source-labelled cluster structure on this variety.

The Grassmannian $\grass{k}{n}$ carries a birational automorphism called the Marsh--Scott twist \cite[\S2]{MaSc}, or simply the twist, which we denote by $x\mapsto\twist{x}$. By \cite[Prop.~8.10]{MaSc}, if $x$ is a cluster variable in Scott's cluster structure, then $\twist{x}$ is a product of a cluster variable with a monomial in frozen variables; indeed, the twist is even a quasi-cluster automorphism in the sense of Fraser \cite[Rem.~6.6]{Fraser}. Related twist automorphisms exist for open positroid varieties \cite{MuSp}, but in the general case they relate cluster variables in two different cluster algebra structures on the coordinate ring.
It has been shown recently \cite[Thm.~7.2]{Pre4}, \cite[Thm.~B]{CLSBW} that these more general twist maps are quasi-cluster equivalences between the two relevant cluster structures.

The Marsh--Scott formula, introduced in \cite{MaSc} for a uniform Postnikov diagram,
is a certain dimer partition function which was used in \cite{MaSc}
to compute the twisted Pl\"ucker coordinates $\twist{\Pluck{I}}$.
However, as a formula in the associated cluster algebra, it makes sense for
a general $\white$-standardised Postnikov diagram $D$ of type $(k,n)$
and can be written as follows, for any $k$-subset $I\subset\C_1$:
\begin{equation}
\label{eq:MS1}
  \MSform[^\white]{I} = x^{-\wt(D)}\sum_{\mu:\pmbdry{\mu}=I}x^{\wt[\white](\mu)}, 
\end{equation}
where $\wt[\white](\mu)$ is as in \eqref{eq:wtms-mu} and 
\begin{equation}
\label{eq:wtD}
  \wt(D)=\sum_{\substack{ j\in Q_0 \\ \vmut }}  [P_j].
\end{equation}
Thus the formula associates to each $k$-subset $I\subset\C_1$ a formal Laurent polynomial in $\CC[\Kgp(\proj{A})]$, or equivalently a Laurent polynomial 
in the initial cluster variables $x_j:=x^{[P_j]}$ for $j\in Q_0$.
When $I$ is not in the associated positroid $\posid(D)$, so is the boundary value of no matchings,
the formula gives $\MSform[^\white]{I} =0$.
Note, for comparison, that \eqref{eq:MS1} is written in terms of the quiver~$Q(D)$, 
whereas \cite{MaSc} writes an equivalent formula in terms of the bipartite graph~$\plabic(D)$.

To apply their formula, Marsh--Scott need to evaluate it in $\CC[\grass{k}{n}]$ 
using the specialisation \eqref{eq:specialise}, and then prove the following.

\begin{theorem}[{\cite[Thm.~1.1]{MaSc}}]
\label{thm:MaSc}
Let $D$ be a $\white$-standardised $(k,n)$-diagram, let $I\subset\C_1$ be a $k$-subset
and $\twPluck{I}\in\CC[\grass{k}{n}]$ the associated twisted Pl\"ucker coordinate.
Then
\[ \twPluck{I} = \MSform[^\white]{I}|_{x_j\mapsto\Pluck{I_j}}.\]
\end{theorem}

In the remainder of the paper, we give a categorical interpretation of this result by relating 
the Marsh--Scott formula to the more general cluster character formula of Fu--Keller~\cite{FK},
which computes cluster monomials from (reachable) rigid objects in a Frobenius cluster category. 
Almost all our results will apply for all connected Postnikov diagrams, 
except in Section~\ref{sec:MS=CC},
when we come to use Theorem~\ref{thm:MaSc} to interpret $\MSform[^\white]{I}$ 
as a twisted Pl\"ucker coordinate.

To that end, assume that $D$ is connected, of type $(k,n)$ and $\white$-standardised.
Thus the boundary arrows of $Q(D)$ are $\alpha_i$ for $i\in\C_1$, and all of these arrows are oriented clockwise. Given any $I\subset\C_1$, we can define 
\begin{equation}
\label{eq:PI-white}
 P_I^\white = \bigoplus_{i\in I} P_{h\alpha_i} = \bigoplus_{i\in I} A e_{h\alpha_i}.
\end{equation}
This leads immediately to the following way to rewrite \eqref{eq:MS1}.

\begin{proposition} \label{prop:MS2}
\begin{equation} \label{eq:MS2}
  \MSform[^\white]{I} = x^{[P_I^\white]}\sum_{\mu:\pmbdry{\mu}=I}x^{-[N_\mu]}.
\end{equation}
\end{proposition}

\begin{proof}
We use Proposition~\ref{prop:Nmu-class}, noting that
\eqref{eq:wtD} is the first term on the right-hand side of \eqref{eq:Nmu-class-white},
while the second term is $[P_{\bdry\mu}^\white]$.
Hence we can rearrange \eqref{eq:Nmu-class-white} as
\[
  \wt[\white](\mu) - \wt(D) = [P_{\bdry\mu}^\white] - [N_\mu],
\]
to transform \eqref{eq:MS1} into \eqref{eq:MS2}.
\end{proof}

We now want to rewrite \eqref{eq:MS2} in a module-theoretic way, that is, as a function of
the rank 1 module $M\in\CM(B)$ such that $\rho(M)\isom M_I\in\CM(\jksalg)$,
as in Definition~\ref{d:rank1mod}.
Note that such an $M$ will exist provided $\{\mu:\pmbdry{\mu}=I\}$ is non-empty, i.e.\ provided $I$ is an element of the positroid associated to $D$,
by Proposition~\ref{pro:positroid}.
In that case, $M$ and $I$ are equivalent data because $\rho\colon\CM(B)\to \CM(\jksalg)$
is fully faithful, by Proposition~\ref{p:restriction}, while the module $\rho(M)\in\CM(C)$ has rank $1$ 
and every $C$-module of rank $1$ is isomorphic to $M_I$ for a unique $I\subset\C_1$,
by \cite[Prop.~5.2]{JKS}.

As a result we can consider, as a function of such a rank 1 $M\in\CM(B)$ with $\rho(M)\isom M_I$,
the projective $B$-module
\begin{equation} \label{eq:PM-white}
  \canprojcov[\white] M= eP_I^\white = \bigoplus_{i\in I} B e_{h\alpha_i}
\end{equation}
and thereby realise our goal of rewriting \eqref{eq:MS2} module-theoretically.

\begin{theorem} \label{thm:MS3}
Let $M\in \CM(B)$ with $\rk(M)=1$, and let $I\in\C_1$ be the $k$-subset such that $\rho(M)\isom M_I\in\CM(\jksalg)$. Then
\begin{equation} \label{eq:MS3}
  \MSform[^\white]{I} = x^{[F\canprojcov[\white] M]}\sum_{\substack{N\leq F M \\eN=M}} x^{-[N]}.
\end{equation}
\end{theorem}

\begin{proof}
Combining Proposition~\ref{p:matchings-to-modules} and Remark~\ref{r:combinatorics}, 
we have a bijection 
\[
 \theta\colon \{N\leq F M : eN=M \} \to \{\mu : \pmbdry{\mu}=I\} 
\] 
such that $\theta(N)=\mu$ when $N\isom N_\mu$.
On the other hand, the natural map $P_I^\white\to F\canprojcov[\white]{M}$ is an isomorphism by Proposition~\ref{p:double-cen}(iii), since $P_I^\white\in \add(Ae)$.
\end{proof}

We may also observe that $\canprojcov[\white] M$ 
has a more special relationship to $M$.

\begin{lemma}
\label{lem:canprojcov}
For $M$ as in Theorem~\ref{thm:MS3}, 
there is a (non-minimal) projective cover $\canprojcov[\white] M \to M$.
\end{lemma}

\begin{proof}
Since $Be_{h\alpha_i}$ is projective with top at $h\alpha_i$ and the fibre $e_{h\alpha_i}M$ is a free rank one $Z$-module, there is a map $\pi_i\colon Be_{h\alpha_i}\to M$ (unique up to rescaling by $Z^\times$) such that the restriction $e_{h\alpha_i}Be_{h\alpha_i}\to e_{h\alpha_i}M$ to fibres over $h\alpha_i$ is surjective. Let $\pi\colon\canprojcov[\white]{M}\to M$ be the map with components given by the $\pi_i$.

Now consider $\rho(\pi)\colon\rho(\canprojcov[\white]{M})\to\rho(M)\isom M_I$. As a map of vector spaces, this is identical to $\pi$, so it suffices to show that $\rho(\pi)$ is surjective. Note that, since the canonical map $C\to B$ is injective by Proposition~\ref{p:restriction}, the top of any $B$-module $N$ is a quotient of the top of the $C$-module $\rho(N)$. Thus $\top\rho(\canprojcov[\white]{M})$ has all the vertices $e_{h\alpha_i}$, for $i\in I$, in its support, and $\top\rho(M)=\top M_I$ is supported on a subset of these vertices by construction. Since $\rho(\pi)$ maps any $Z$-module generator of $e_{h\alpha_i}Be_{h\alpha_i}$ to a $Z$-module generator of $e_{h\alpha_i}M$, it induces a surjective map $\top\rho(\canprojcov[\white]{M})\to\top{\rho(M)}$ and so is surjective as required.
\end{proof}

The use of $\white$-standardised diagrams in this section reflects the choices made in \cite{MaSc}, but we can equally well work with $\black$-standardised diagrams throughout. 
In this context, given a $k$-subset $I\subset\C_1$, we define
\begin{equation}\label{eq:MS1-black}
\MSform[^\black]{I}=x^{-\wt(D)}\sum_{\mu:\bdry\mu=I}x^{\wt[\black](\mu)},
\end{equation}
and given additionally $M\in\CM(B)$ with $\rho(M)\cong M_I$ we define
\[P_I^\black=\bigoplus_{i\not\in I}P_{h\beta_i},\quad \canprojcov[\black]M=eP_I^\black.
\]
By analogous arguments to those for $\MSform[^\white]{I}$, one may show that
\begin{equation}
\label{eq:MS23-black}
\MSform[^\black]{I}=x^{[P_I^\black]}\sum_{\mu:\pmbdry{\mu}=I}x^{-[N_\mu]}
=x^{[F\canprojcov[\black]M]}\sum_{\substack{N\leq F M \\eN=M}} x^{-[N]}
\end{equation}
and that there is a projective cover $\canprojcov[\black]M\to M$.

Given a diagram $D$ of type $(k,n)$ and a $k$-subset $I$, we may either choose a $\white$-standardisation of $D$ and compute $\MSform[^\white]{I}$, or choose a $\black$-standardisation of $D$ and compute $\MSform[^\black]{I}$. 
Comparing \eqref{eq:MS3} and \eqref{eq:MS23-black}, we see that
\[\MSform[^\black]{I}=x^{[F\canprojcov[\black]{M}]-[F\canprojcov[\white]{M}]}\MSform[^\white]{I}.
\]
Since $\canprojcov[\black]{M}$ and $\canprojcov[\white]{M}$ are projective $B$-modules, it follows that $F\canprojcov[\black]{M}$ and $F\canprojcov[\white]{M}$ are projective $A$-modules with top supported on the boundary vertices, and so $x^{[F\canprojcov[\black]{M}]-[F\canprojcov[\white]{M}]}$ is a Laurent monomial in frozen variables. Thus the conclusions of Marsh--Scott's Theorem~\ref{thm:MaSc} also hold for $\black$-standardised diagrams with $\MSform[^\black]{I}$ in place of $\MSform[^\white]{I}$.

If $D$ is a $\white$-standardised Postnikov diagram of type $(k,n)$, then by the observations of Remark~\ref{r:opposite} its opposite diagram $D\op$ is $\black$-standardised of type $(n-k,n)$. Since $Q(D\op)=Q(D)\op$ has the same set $Q_0$ of vertices as $Q(D)$, and $A_{D\op}=A_D\op$, there is a canonical isomorphism $\Kgp(\proj{A_D})\isoto\Kgp(\proj{A_{D\op}})$ given by $[A_De_i]\mapsto[A_{D\op}e_i]$ for each $i\in Q_0$, which we will treat as an identification, exploiting that the basis of projectives in each Grothendieck group yields an isomorphism with the lattice $\ZZ^{Q_0}$. Thus we may identify the spaces of polynomials with exponents in the two lattices, and view Marsh--Scott formulae computed with respect to $D$ and $D\op$ as taking values in the same Laurent polynomial ring. This allows us to make another comparison of the formulae \eqref{eq:MS1} and \eqref{eq:MS1-black}.

\begin{proposition}
\label{p:MSblack-vs-MSwhite}
Let $D$ be a $\white$-standardised Postnikov diagram of type $(k,n)$ and $I\subset\C_1$ a $k$-subset. Then
\[\MSform[^\white_D]{I}=\MSform[^\black_{D\op}]{I\comp},\]
where each Marsh--Scott formula is calculated using the diagram indicated in the subscript.
\end{proposition}
\begin{proof}
Note that the right-hand side of the claimed formula makes sense, since $D\op$ is a $\black$-standardised diagram of type $(n-k,n)$. As already observed, we have $Q(D\op)=Q(D)\op$, and so $Q(D)$ and $Q(D\op)$ have the same set of arrows, and the same set of perfect matchings. Each perfect matching thus has two boundary values, depending on whether it is viewed as a matching of $Q(D)$ or of $Q(D\op)$, but since a boundary arrow is clockwise in $Q(D)$ if and only if it is anticlockwise in $Q(D\op)$, it follows directly from the definition that these two boundary values are complementary to each other. In particular, the set of perfect matchings of $Q(D)$ with boundary value $I$ is equal to the set of perfect matchings of $Q(D\op)$ with boundary value $I\comp$.

Given a perfect matching $\mu$, we can compute $\wt[\white](\mu)$ viewing $\mu$ as a perfect matching of the $\white$-standardised quiver $Q(D)$, or $\wt[\black](\mu)$ viewing $\mu$ as a perfect matching of the $\black$-standardised quiver $Q(D\op)$, where these weights are defined in \eqref{eq:mswts}. Since $Q(D)$ and $Q(D\op)$ have the same set of faces, but a face is white in $Q(D)$ if and only if it is black in $Q(D\op)$, these two calculations are the same (recalling that we identify $\Kgp(\proj{A_D})$ with $\Kgp(\proj{A_{D\op}})$ using the common set of quiver vertices), which completes the proof.
\end{proof}

\section{The Caldero--Chapoton formula} \label{sec:newCC}

\newcommand{\FCC}{\mathcal{E}} 

Let $D$ be a connected Postnikov diagram of type $(k,n)$ with dimer algebra $A=A_D$ 
and boundary algebra $B=eAe$, and let $T=eA\in\CM(B)$. 
As mentioned in Section~\ref{sec:ind-res}, it follows from \cite[Thm.~3.7]{Pre3} and the general theory from \cite{Pre} 
that $T$ is a cluster-tilting object in the category $\GP(B)$ of Gorenstein projective $B$-modules, 
that this category is a stably $2$-Calabi--Yau Frobenius category, and moreover that $\gldim A\leq 3$.

Thus we may consider the Caldero--Chapoton cluster character formula,
as described by Fu--Keller \cite{FK} in the context of Frobenius categories. 
For each $X\in\GP(B)$, we define $\clucha{T}{X}$ by the formula
\begin{equation}\label{eq:CC}
  \clucha{T}{X} = \cluva^{[FX]} \sum_{E\leq GX} \cluva^{-[E]},
\end{equation}
giving a sum of formal Laurent monomials $x^v$ for $v \in \Kgp(\proj A)$. 
Note that it may be that, for some $v\in \Kgp(\proj A)$, the set
\[
\Gr_v(GX)=\{E \leq GX : [E]=v\}
\]
is infinite and, in this case, we count this set by its Euler characteristic,
i.e. in the sum in \eqref{eq:CC} the coefficient of $x^{-v}$ is $\chi (\Gr_v(GX))$.
By \cite[Thm.~3.3]{FK}, the function $X\mapsto\clucha{T}{X}$ is a cluster character on $\GP(B)$ 
in the sense of \cite[Def.~3.1]{FK}.

\begin{remark}
Here we have used some of the homological properties of $A$, such as the fact that $\gldim{A}\leq 3$, 
to simplify the exponents in the cluster character formula in \cite{FK}; 
an explanation of this may be found in \cite[\S3]{GP} (see also \cite[Rem.~3.5]{FK}),
together with the observation that we can relax the requirement in \cite{FK}
that the Frobenius category is $\Hom$-finite.

The cluster-tilting object $T=eA$ has a natural decomposition into indecomposable summands $T_j=eAe_j$ for $j\in Q_0$, 
yielding a basis $[P_j]=[F T_j]$ for $\Kgp(\proj{A})$. 
We may use this basis to write the formal monomials above as actual monomials in the variables $x_j:=x^{[P_j]}=\clucha{T}{T_j}$.
This is how the formula is written in \cite{GP}, using the Euler pairing to compute the coefficient of each indecomposable 
projective in the expression for an arbitrary K-theory class in this basis.
\end{remark}

Note that the formula $\clucha{T}{X}$ makes sense for any $X\in\CM(B)$ (or even for any $X\in\fgmod{B}$) 
although it is only the restriction of the function $X\mapsto\clucha{T}{X}$ to objects of the stably $2$-Calabi--Yau Frobenius 
category $\GP(B)$ which need have the properties of a cluster character as described in \cite{FK}.

\begin{proposition}
\label{p:CC-syz}
Let $M\in\CM(B)$, and consider any (exact) syzygy sequence
\[0\lra{} \syz M\lra{} PM\lra{} M\lra{} 0,
\]
where the map $PM\to M$ is a (possibly non-minimal) projective cover. Then
\[\clucha{T}{\syz M}
=x^{[FPM]} \sum_{\substack{N\leq F M \\eN=M}} x^{-[N]}.
\]
\end{proposition}

\begin{proof}
Proposition~\ref{p:F'-is-min} tells us that 
\[ \{N\leq FM : eN=M\}  = \{N : \smallF M \leq N\leq F M\},\]
so, using the definition of $\clucha{T}{\syz M}$ from \eqref{eq:CC}, what we wish to prove is that
\begin{equation}
\label{eq:MS5}
x^{[F\syz M]}\sum_{E\leq G\syz M}x^{-[E]}=x^{[FPM]}\sum_{N : \smallF M \leq N\leq F M}x^{-[N]}.
\end{equation}
From the short exact sequence \eqref{eq:GOm}, we know that $G\syz M=FM/F'M$
and so there is a bijection between $\{N : \smallF M\leq N\leq FM\}$ 
and $\{E \leq G\syz M\}$ given by setting $E = N/F'M$.
Combining this with \eqref{eq:FPi}, we obtain
\[
[N] - [E] = [F'M] = [FPM] - [F\syz M]
\]
when $E$ and $N$ are related by this bijection, and thus
\[
 [FPM] - [N]  = [F\syz M] - [E].
\]
as required for \eqref{eq:MS5}.
\end{proof}

Now, for any rank 1 module $M\in \CM(B)$, let $\cansyz[\white]M$ be the syzygy computed as the kernel 
of the projective cover $\canprojcov[\white]M \to M$ from Lemma~\ref{lem:canprojcov}.
The main result of this section is then the following.

\begin{theorem}
\label{thm:MS=CC}
Let $D$ be a connected Postnikov diagram, with dimer algebra $A$ and boundary algebra $B$.
Let $M\in\CM(B)$ be a rank $1$ module, with $\rho(M)\isom M_I$. 
Then we have
\[
\MSform[^\white]{I}=\clucha{T}{\cansyz[\white] M},
\]
where the left-hand side is the Marsh--Scott formula, as in \eqref{eq:MS1}, with respect to a $\white$-standardisation of $D$, 
and the right-hand side is the cluster character \eqref{eq:CC}, with respect to the cluster-tilting object $T=eA$.
\end{theorem}

\begin{proof}
Applying Proposition~\ref{p:CC-syz} in the case that $PM=\canprojcov[\white]{M}$, so that $\syz M=\cansyz[\white]{M}$, we see that
\[\clucha{T}{\cansyz[\white] M}= x^{[F\canprojcov[\white] M]}\sum_{\substack{N\leq F M \\eN=M}} x^{-[N]}.
\]
Then the right-hand side coincides with $\MSform[^\white]{I}$ by Theorem~\ref{thm:MS3}.
\end{proof}

In stating Proposition~\ref{p:CC-syz} and Theorem~\ref{thm:MS=CC}, we used the fact that the formula \eqref{eq:CC} for $\clucha{T}{X}$
makes sense for $X\in\CM(B)$, even though this function is only strictly a cluster character on $\GP(B)$.
However, it turns out that this caveat is not needed, because of the following lemma, the proof of which was pointed out to us by Bernt Tore Jensen.

\begin{lemma}\label{lem:syzygy}
For $M\in\CM(B)$, any syzygy $\syz M$ is in $\GP(B)$.
\end{lemma}

\begin{proof}
We have to prove that $\Ext_B^k(\syz M,B)=0$ for all $M\in \CM(B)$ and all $k\geq 1$.
However, since $\Ext_B^{k+1}(\syz M,B)=\Ext_B^k(\syz^2 M,B)$ and $\CM(B)$ is closed under syzygies,
it suffices to prove that $\Ext_B^1(\syz M,B)=0$, for all $M\in \CM(B)$.

The restriction functor $\rho\colon\CM(B)\to\CM(C)$ is exact and fully faithful by Proposition~\ref{p:restriction} and so, dropping $\rho$ from the notation, we have $\Ext_B^1(M_1,M_2)\subset\Ext_C^1(M_1,M_2)$ for all $M_1,M_2\in \CM(B)$.
We also have $\Ext_C^1(M_1,M_2)=\Ext_C^1(M_2,M_1)^*$, since $\CM(C)$ is stably 2-Calabi--Yau.
Thus it suffices to prove that $\Ext_C^1(B,\syz M)=0$.

Now consider the syzygy sequence  $0\lra{} \syz M \lra{} PM \lra{p} M \lra{} 0$ as a sequence in $\CM(C)$, 
where $p$ is a $B$-projective cover.
Then part of the long exact sequence for $\Hom_C(B,-)$ is
\[
 \Hom_C(B,PM) \lra{p_*} \Hom_C(B,M) \lra{} \Ext_C^1(B,\syz M) \lra{} \Ext_C^1(B,PM)
\]
and $p_*$ is surjective, using again that $\rho$ is fully faithful. Because $B$ is rigid, $\Ext_C^1(B,PM)=0$, and hence we obtain the required result.
\end{proof}

\begin{corollary}
\label{c:gor-dim}
Let $D$ be a connected Postnikov diagram with boundary algebra $B$. Then $B$ is Iwanaga--Gorenstein with Gorenstein dimension at most $2$.
\end{corollary}
\begin{proof}
The algebra $B$ is Noetherian since it is free and finitely generated over $Z$ by Proposition~\ref{p:thin}. Let $M\in\fgmod{B}$ and choose first and second syzygies $\syz M$ and $\syz^2M$. Then $\syz M\in\CM(B)$ since $Z$ is a PID, and so
\[\Ext^3_B(M,B)=\Ext^1_B(\Omega^2M,B)=0\]
by Lemma~\ref{lem:syzygy}. Thus $B$ has injective dimension at most $2$ on the left. Since $B\op$ is the boundary algebra of the connected Postnikov diagram $D\op$, we may apply the same argument to $B\op$ to see that $B$ has injective dimension at most $2$ on the right.
\end{proof}

\begin{remark}
Corollary~\ref{c:gor-dim} improves on the upper bound of $3$ for the Gorenstein dimension of $B$ coming from the general results of \cite{Pre}, applied to connected Postnikov diagrams via \cite[Thm.~3.7]{Pre3}. When $D$ is a $(k,n)$-diagram, so $B\isom C$, its Gorenstein dimension is $1$ by \cite{JKS}; this is the reason why $\CM(C)=\GP(C)$ in this case. We expect that in all other cases the Gorenstein dimension is exactly $2$, and so $\GP(B)$ is a proper subcategory of $\CM(B)$.
\end{remark}

\section{The Marsh--Scott twist}
\label{sec:MS=CC}

When $D$ is a $(k,n)$-diagram, the situation is simpler. 
The canonical map $C\to B$ is an isomorphism, so $\rho\colon\CM(B)\to\CM(C)$ is an equivalence.
Suppressing this equivalence in the notation, it follows from Theorem~\ref{t:labels} 
that the indecomposable summand $T_j=eAe_j$ of the cluster-tilting object $T$ 
is isomorphic to the rank $1$ $C$-module $M_{I_j}$.
Moreover, $\CM(C)=\GP(C)$ is a stably $2$-Calabi--Yau Frobenius category, 
on which $\cluschar{T}$ is an honest cluster character.

\begin{proposition}
\label{p:cc-pluck}
For any $k$-subset $I\subset\C_1$, we have
\begin{equation}
\label{eq:CC=Pluck}
 \clucha{T}{M_I}|_{x_j\mapsto\Pluck{I_j}}=\Pluck{I}.
\end{equation} 
\end{proposition}
\begin{proof}
In \cite{JKS} (see also Remark~\ref{r:JKS-comparison}), Jensen, King and Su exhibit 
a cluster character $\jkscluschar\colon\CM(C)\to\CC[\grass{k}{n}]$ such that 
$\jksclucha{M_I}=\Pluck{I}$ for all $I$.
In particular, since $T_j\isom M_{I_j}$ by Theorem~\ref{t:labels}, 
we have $\jksclucha{T_j}=\Pluck{I_j}$.

On the other hand, the map $X\mapsto\clucha{T}{X}|_{x_j\mapsto\Pluck{I_j}}$ is again a cluster character, because this class of functions is closed under postcomposition with arbitrary maps of rings. By Proposition~\ref{p:double-cen}, for each $j\in Q_0$ we have $GT_j=0$ and
\[\clucha{T}{T_j}=x^{[FT_j]}=x^{[P_j]}=x_j,\]
so that $\clucha{T}{T_j}|_{x_j\mapsto\Pluck{I_j}}=\Pluck{I_j}$.  

Thus the two cluster characters $\cluschar{T}$ (after the substitution $x_j\mapsto\Pluck{I_j}$) 
and $\jkscluschar$ agree on the indecomposable summands $T_j$ of $T$, and hence by the multiplication formula \cite[Def.~3.1(3)]{FK} they agree on all rigid indecomposable objects reachable from $T$, i.e.\ appearing as a summand of some cluster-tilting object obtained from $T$ by a sequence of mutations. 
As a consequence of \cite[Thm.~9.5]{JKS}, this class of objects includes $M_I$ 
for all $k$-subsets $I\subset\C_1$, and so
\[\clucha{T}{M_I}|_{x_j\mapsto\Pluck{I_j}}=\jksclucha{M_I}=\Pluck{I}\]
as required.
\end{proof}

Furthermore, when $D$ is a $(k,n)$-diagram we may interpret the Marsh--Scott formula 
as a twisted Pl\"ucker coordinate via Theorem~\ref{thm:MaSc}, and so Theorem~\ref{thm:MS=CC} becomes
\begin{equation}
\label{eq:CC=twPluck}
\clucha{T}{\cansyz[\white] M_I}|_{x_j\mapsto\Pluck{I_j}}=\twPluck{I}.
\end{equation}
Comparing \eqref{eq:CC=Pluck} and \eqref{eq:CC=twPluck}, we see that the operation $\cansyz[\white]$ on $\CM(C)$ 
can be considered a categorification of the Marsh--Scott twist.

\begin{remark}
\label{r:GLS}
Note that the stable $2$-Calabi--Yau property of $\CM(\jksalg)$ means that, as functors on the stable category, $\syz\isom\tau^{-1}$ is the inverse Auslander--Reiten translation. 
In particular, when $I$ is not an interval, so that $M_I$ is not itself projective, 
any syzygy $\syz M_I$ is indecomposable in the stable category $\underline{\CM}(\jksalg)$ 
and so has a single non-projective indecomposable summand in $\CM(\jksalg)$. 
In view of \eqref{eq:CC=twPluck}, this corresponds to the fact \cite[Prop.~8.10]{MaSc} 
that $\twPluck{I}$ is a product of a single mutable cluster variable with a monomial in frozen variables. 

The fact that $\clucha{T}{\cansyz[\white] M_I}|_{x_j\mapsto\Pluck{I_j}}$ and $\twPluck{I}$ coincide after setting frozen variables to $1$ follows from a result of Gei\ss--Leclerc--Schr\"oer \cite[Thm.~6]{GLS}. Our choice of projective cover $\canprojcov[\white] M_I$ is designed to ensure that the frozen variables appearing in $\clucha{T}{\cansyz[\white] M_I}|_{x_j\mapsto\Pluck{I_j}}$ coincide precisely with those appearing in Marsh--Scott's twisted Pl\"ucker coordinate $\twPluck{I}$.

Defining instead $\cansyz[\black] M_I$ to be the kernel of the projective cover 
$\canprojcov[\black]M_I\to M_I$, we may show in an exactly analogous way that
\[\MSform[^\black]{I}=\clucha{T}{\cansyz[\black] M_I},
\]
so that $\cansyz[\black]$ again categorifies a birational twist automorphism, 
differing from the Marsh--Scott twist by multiplication by a Laurent monomial in frozen variables.
\end{remark}

\section{The Muller--Speyer twist} \label{sec:MuSp=CC}
\newcommand{\net}{\mathsf{net}}
\newcommand{\clu}{\mathsf{clu}}

Muller and Speyer \cite{MuSp} describe twist automorphisms for open positroid varieties in general. 
These maps involve inverting frozen variables and so are not defined on the closed positroid varieties, in contrast to the Marsh--Scott twist for the Grassmannian, i.e.\ the closed uniform positroid variety. 
Indeed, even in the uniform case, Muller--Speyer's twist differs from Marsh--Scott's by multiplication by a Laurent monomial in frozen variables, which can have a non-trivial denominator. Our methods also give categorifications of these more general twists.

A key ingredient in Muller--Speyer's construction is the map $\MSmap\colon\ZZ^{Q_0}\to\matlat$ as defined in \eqref{eq:MSmap}, where $\matlat$ is the matching lattice as introduced in Remark~\ref{rem:eta-def}.
This map is an isomorphism, with inverse $\alteta$ as defined in \eqref{eq:matcls-alt}, by \cite[Prop.~5.5]{MuSp} or Corollary~\ref{c:eta=MSinv}. 

Let $\posid$ be the positroid associated to $D$
and let $\openposidvar=\openposidvar(\posid)$ be the corresponding open positroid variety. In \cite[Sec.~6]{MuSp}, Muller--Speyer define a twist automorphism 
$\MuSptw\colon\CC\bigl[\openposidcone\bigr]\to\CC\bigl[\openposidcone\bigr]$,
where $\CC\bigl[\openposidcone \bigr]$ is the homogeneous coordinate ring, 
i.e.\ the coordinate ring of the cone on $\openposidvar$, and show the following (in the current notation).

\begin{theorem}[{\cite[Prop.~7.10]{MuSp}}]\label{thm:MuSptw}
For any $I\in\posid$, we have
\begin{equation}\label{eq:MuSptw}
 \MuSptw(\Pluck{I})
  =\sum_{\mu:\bdry\mu=I}x^{-\alteta(\mu)}|_{x_j\mapsto\Pluck{I_j}}.
\end{equation}
\end{theorem}

\noindent Note: here and later in this section, 
we abuse notation by writing $\Pluck{I}$ for the restriction of this Pl\"ucker coordinate to $\openposidvar\subset\grass{k}{n}$. The right-hand side of \eqref{eq:MuSptw} is a formal Laurent polynomial in $\CC\bigl[\ZZ^{Q_0}\bigr]$, or equivalently a Laurent polynomial in variables $x_j=x^{p_j}$, where $p_j$ is a standard basis vector as in Section~\ref{sec:MuSp}.

\begin{proof}
By \cite[Thm.~7.1]{MuSp}, we have the following commutative diagram, in which the horizontal maps are isomorphisms.
\[\begin{tikzcd}[column sep=40pt]
\CC\bigl[\matlat\bigr] \arrow{r}{\CC[-\alteta]} &
\CC\bigl[\ZZ^{Q_0}\bigr]
\\
\CC\bigl[\openposidcone\bigr]\arrow{u}{\net}\arrow{r}{\MuSptw} &
\CC\bigl[\openposidcone\bigr]\arrow{u}[swap]{\clu}
\end{tikzcd}
\]
This is the right-hand square in \cite[Thm.~7.1]{MuSp}, rewritten in terms of maps between coordinate rings. Now \eqref{eq:MuSptw} follows by tracing $\Pluck{I}\in\CC\bigl[\openposidcone\bigr]$ through this diagram.

To further aid comparison with \cite{MuSp}, we recall the ingredients in the diagram. Firstly, $\CC\bigl[\matlat\bigr]$ is the coordinate ring of the torus 
(written $\Gm^{|Q_1|}/\Gm^{|Q_2|-1}$ in \cite{MuSp}) 
whose character lattice is $\matlat$.
In other words, it is the ring of formal Laurent polynomials with exponents in $\matlat$.
Similarly, $\CC\bigl[\ZZ^{Q_0}\bigr]$ is the coordinate ring of the torus $\Gm^{Q_0}$.

The map $\CC[-\alteta]$ is the isomorphism of torus coordinate rings induced by the 
map $-\alteta\colon \matlat\to\ZZ^{Q_0}$ of their character lattices,
which is the inverse of $-\MSmap$ by Corollary~\ref{c:eta=MSinv}.
The map $\net$ is given by dimer
partition functions (see \cite[\S3.2]{MuSp})
\[
 \Pluck{I} \mapsto \sum_{\mu:\bdry\mu=I}x^{\mu}
\]
and corresponds to (a lift of) the embedding of the network torus
by the boundary measurement map of \cite{postnikov}.

The map $\clu$ corresponds to the embedding of the cluster torus in $\openposidcone$. More precisely, it is obtained by composing the inverse of the map $\openclustalg{D}\to\CC\bigl[\openposidcone\bigr]$ induced by the substitution $x_j\mapsto\Pluck{I_j}$ from \eqref{eq:specialise}, which is a well-defined isomorphism by \cite[Thm.~3.5]{GL}, with the inclusion $\openclustalg{D}\subset\CC\bigl[\ZZ^{Q_0}\bigr]$.
\end{proof}

Now assume $D$ is connected. 
As in Section~\ref{sec:MuSp}, we identify $\ZZ^{Q_0}$ with $\Kgp(\proj{A})$ by $p_j\mapsto[P_j]$, which identifies $\alteta(\mu)$ with $\matcls(\mu)=[N_\mu]$; see \eqref{eq:matcls-value}. In this case we may, just as in Theorem~\ref{thm:MS3}, rewrite \eqref{eq:MuSptw} as
\begin{equation} \label{eq:MuSptw2}
\MuSptw(\Pluck{I})
  =\sum_{\mu:\bdry\mu=I}x^ {-[N_\mu]}|_{x_j\mapsto\Pluck{I_j}}
 =\sum_{\substack{N\leq FM\\eN=M}} x^{-[N]}|_{x_j\mapsto\Pluck{I_j}},
\end{equation}
for (the unique) $M\in\CM(B)$ with $\rho M\isom M_I$, which exists since $I\in\posid$.

Recall from \eqref{eq:CC} the Fu--Keller cluster character $\Phi_T\colon\GP(B)\to\CC[\Kgp(\proj{A})]$, where $B=B_D$ is the boundary algebra of $D$ and $T=eA_D\in\GP(B)$ is the initial cluster-tilting object.

\begin{theorem}
\label{t:MuSp=CC}
Let $M\in\CM(B)$ such that $\rho M\isom M_I\in\CM(C)$. Let $PM$ be any (possibly non-minimal) projective cover of $M$, fitting into a short exact sequence
\[0\to\syz M\to PM\to M\to 0.\]
Then
\begin{equation}
\label{eq:MuSp=CC}
\MuSptw(\Pluck{I})=\left.\frac{\clucha{T}{\syz M}}{\clucha{T}{PM}}\right|_{x_j\mapsto\Pluck{I_j}}.
\end{equation}
\end{theorem}

\begin{proof}
By Proposition~\ref{p:CC-syz},
\begin{equation}
\label{eq:CC-syz-match}
\clucha{T}{\syz M}=x^{[FPM]}\sum_{\substack{N\leq FM\\eN=M}} x^{-[N]}.
\end{equation}
Since $PM$ is projective, it is in $\add{T}$, so $GPM=0$ and $\clucha{T}{PM}=x^{[FPM]}$. 
Thus we obtain \eqref{eq:MuSp=CC} by rearranging \eqref{eq:CC-syz-match} and using \eqref{eq:MuSptw2}.
\end{proof}

\begin{remark}
Theorem~\ref{t:MuSp=CC} is the analogue for positroid varieties of Gei\ss--Leclerc--Schr\"oer's result \cite[Thm.~6]{GLS} for unipotent cells in Kac--Moody groups. Indeed, the uniform open positroid variety in $\Gr_k^n$ is an example of such a cell (cf.~Remark~\ref{r:GLS}).
\end{remark}

\end{document}